\documentclass[reqno]{amsart}

\usepackage{amssymb,amsmath}
\usepackage{euscript}
\usepackage{graphicx}
\usepackage{tikz}
\usepackage{tikz-cd}
\usepackage{hyperref}
\usepackage{mathrsfs}
\usepackage{dsfont}
\usepackage{mathtools}
\usepackage{microtype}

\usepackage{xy}

\tikzset{
  trim node/.default=1cm,
  trim node/.style={
    overlay,
    append after command={
      ([xshift={+#1}]\tikzlastnode.north west)
      ([xshift={+-#1}]\tikzlastnode.south east)}},
  down and trim/.default=1cm,
  down and trim/.style={
    yshift=-(\pgfmatrixcurrentcolumn-1)*1.5\baselineskip,
    trim node={#1}},
  downup and trim/.default=1cm,
  downup and trim/.style={
    yshift=iseven(\pgfmatrixcurrentcolumn) ? -1.5\baselineskip : 0pt,
    trim node={#1}},
  -|/.style={to path={-|(\tikztotarget)\tikztonodes}},
  |-/.style={to path={|-(\tikztotarget)\tikztonodes}},
  -| sl/.style={-|, xslant=-1},
  |- sl/.style={|-, xslant= 1},
  center picture/.style={
    trim left=(current bounding box.center),
    trim right=(current bounding box.center)}}

\usepackage[textwidth=25mm, textsize=tiny]{todonotes}

\newtheorem{thm}{Theorem}[section]
\newtheorem{cor}[thm]{Corollary}
\newtheorem{lem}[thm]{Lemma}

\newtheorem{prop}[thm]{Proposition}
\newtheorem{defin}[thm]{Definition}
\newtheorem{def-lem}[thm]{Definition-Lemma}

\theoremstyle{remark}

\newtheorem{warning}[thm]{Warning}
\newtheorem{rem}[thm]{Remark}

\newtheorem{example}[thm]{Example}

\newtheorem{construction}[thm]{Construction}
\newtheorem{notation}[thm]{Notation}
\newtheorem{convention}[thm]{Convention}
\newtheorem*{claim}{Claim}

\numberwithin{equation}{section}

\newcommand{\bbA}{\mathbb{A}}
\newcommand{\bbB}{\mathbb{B}}
\newcommand{\bbC}{\mathbb{C}}
\newcommand{\bbD}{\mathbb{D}}

\newcommand{\bbF}{\mathbb{F}}

\newcommand{\bbL}{\mathbb{L}}
\newcommand{\bbM}{\mathbb{M}}

\newcommand{\bbR}{\mathbb{R}}

\newcommand{\bbS}{\mathbb{S}}

\newcommand{\bbX}{\mathbb{X}}

\newcommand{\bbZ}{\mathbb{Z}}

\newcommand{\Sp}{\mathrm{Sp}}
\newcommand{\Ex}{\mathrm{Ex}}
\newcommand{\Ho}{\mathrm{Ho}}
\newcommand{\Cat}{\mathrm{Cat}}
\newcommand{\perf}{\mathrm{Perf}}

\def\bC{\mathbb{C}}

\def\bE{\mathbb{E}}
\def\bF{\mathbb{F}}
\def\bG{\mathbb{G}}

\def\bL{\mathbb{L}}

\def\bP{\mathbb{P}}

\def\bR{\mathbb{R}}
\def\bS{\mathbb{S}}
\def\bT{\mathbb{T}}

\def\bX{\mathbb{X}}
\def\bY{\mathbb{Y}}
\def\bZ{\mathbb{Z}}

\newcommand{\scrA}{\mathscr{A}}
\newcommand{\scrB}{\mathscr{B}}
\newcommand{\scrC}{\mathscr{C}}
\newcommand{\scrD}{\mathscr{D}}
\newcommand{\scrE}{\mathscr{E}}
\newcommand{\scrF}{\mathscr{F}}

\newcommand{\scrM}{\mathscr{M}}

\newcommand{\calC}{\mathcal{C}}
\newcommand{\calD}{\mathcal{D}}

\newcommand{\calG}{\mathcal{G}}

\newcommand{\calL}{\mathcal{L}}

\newcommand{\calN}{\mathcal{N}}

\newcommand{\calP}{\mathcal{P}}

\newcommand{\calW}{\mathcal{W}}

\newcommand{\cA}{\mathcal{A}}

\newcommand{\cC}{\mathcal{C}}
\newcommand{\cD}{\mathcal{D}}

\newcommand{\cH}{\mathcal{H}}

\newcommand{\cL}{\mathcal{L}}

\newcommand{\cP}{\mathcal{P}}

\newcommand{\fraka}{\mathfrak{a}}
\newcommand{\frakb}{\mathfrak{b}}
\newcommand{\frakc}{\mathfrak{c}}
\newcommand{\frakD}{\mathfrak{D}}

\newcommand{\frake}{\mathfrak{e}}

\newcommand{\frakh}{\mathfrak{h}}
\newcommand{\frakl}{\mathfrak{l}}
\newcommand{\frakL}{\mathfrak{L}}

\newcommand{\frakm}{\mathfrak{m}}
\newcommand{\frako}{\mathfrak{o}}

\newcommand{\frakq}{\mathfrak{q}}

\newcommand{\frakU}{\mathfrak{U}}

\newcommand{\frakX}{\mathfrak{X}}
\newcommand{\frakY}{\mathfrak{Y}}
\newcommand{\frakZ}{\mathfrak{Z}}

\DeclareMathOperator{\codim}{codim}
\DeclareMathOperator{\coker}{coker}
\DeclareMathOperator{\crit}{Crit}
\DeclareMathOperator{\grad}{grad}

\DeclareMathOperator{\ind}{ind}

\DeclareMathOperator{\ob}{ob}

\newcommand{\ch}{{\rm ch}}
\newcommand{\Mod}{{\rm Mod}}

\newcommand{\eval}{\mathrm{Eval}}

\newcommand{\flow}{\mathrm{Flow}}
\newcommand{\fr}{\mathrm{fr}}

\newcommand{\ho}{\mathrm{Ho}}
\newcommand{\identity}{\mathrm{Id}}

\newcommand{\std}{\mathrm{std}}

\newcommand{\unit}{\mathds{1}}

\newcommand{\dgnerve}{{\rm N}^{\rm dg}}
\newcommand{\abs}[1]{\lvert#1\rvert}

\title[Bulk-deformations, Floer complex bordism, and GRR]{Bulk-deformations, Floer complex bordism, \\ and Grothendieck-Riemann-Roch}
\author{Kenneth Blakey}
\address{Department of Mathematics, MIT, 182 Memorial Drive, Cambridge, MA 02139, U.S.A.} 
\email{kblakey@mit.edu}

\author{Noah Porcelli}
\address{Max Planck Institute for Mathematics, Vivatsgasse 7, 53111 Bonn, Germany} 
\email{porcelli@mpim-bonn.mpg.de}

\begin{document}

\begin{abstract}
Given a Liouville manifold, we compute a Floer-homotopical invariant -- the complexification of the lift of symplectic cohomology to complex cobordism -- in terms of a classical Floer-theoretic invariant, namely, symplectic cohomology bulk-deformed by the Chern character. We do this by giving an explicit model for the complexified homotopy groups of the ${\rm MU}$-module spectrum associated to a complex-oriented flow category and proving a ``homotopy coherent'' version of the classical Grothedieck-Riemann-Roch theorem. 

Using the aforementioned relation, we establish a computable cohomological criterion, in terms of the pair-of-pants product and the BV operator on symplectic cohomology, for when this ${\rm MU}$ lift cannot be obtained via base change from the sphere spectrum; moreover, we give examples where this holds. Finally, we use this non-base change criterion to detect examples of non-trivial higher-dimensional complex cobordism classes of relative Gromov-Witten type moduli spaces in the context of a smooth complex projective variety relative to an ample smooth divisor.
\end{abstract}

\maketitle
\tableofcontents
\section{Introduction}

\subsection{Background and context}
Symplectic cohomology is one of the main Floer-theoretic invariants associated to a Liouville manifold $M$. Traditionally, this is defined as a cochain complex over a classical ring, e.g., $\bbZ$ or $\bbC$. We denote by $SC^*(M;\bbZ)$ or $SC^*(M;\bbC)$ the aforementioned cochain complex. Floer homotopy theory seeks to understand lifts of symplectic cohomology to more exotic ring spectra. The question of when such a lift exists is intimately tied to the tangent bundle of $M$; indeed, under appropriate conditions on $TM$, symplectic cohomology can be lifted to many ring spectra of interest. For example, if $M$ admits a stable $\bbR$-polarization, i.e., a Lagrangian distribution in a stabilization of $TM$, then symplectic cohomology can be lifted to a module over the sphere spectrum $\bbS$. Moreover, these lifts really do depend on a choice of tangential structure on $M$: for example, different stable $\bbR$-polarizations give rise to different lifts, cf. \cite{BB25}. Also, cf. \cite[Section 1.2]{AGLA25} and \cite[Section 2]{PS25c} for further discussion regarding such conditions and choices for general ring spectra. 

The most general version of spectral symplectic cohomology which exists on any graded Liouville manifold $M$, and which does not depend on any auxiliary choices, is a lift to a module $\bbF^{\rm U}$ over the complex cobordism ring spectrum ${\rm MU}$.\footnote{Recall, graded means $2c_1(M)=0$. In particular, in the ungraded case, one must work with the periodic analogue ${\rm MUP}$ of ${\rm MU}$.} Taking homotopy groups, 
    \begin{equation}
    SH^{-*}(M;{\rm MU})\equiv\pi_*\bbF^{\rm U},
    \end{equation}
yields a module over the classical complex bordism ring $\Omega^{\rm U}_*\cong\pi_*{\rm MU}$; this is beginning to be referred to as \emph{Floer complex bordism} in the literature. Now, $\bbF^{\rm U}$ ``lifts'' ordinary symplectic cochains in the sense that $\bbF^{\rm U}$ recovers ordinary symplectic cochains via base change:
    \begin{equation}
    SC^{-*}(M;\bbZ)\cong\bbF^{\rm U}\otimes_{{\rm MU}}\bbZ.
    \end{equation}

$\Omega^{\rm U}_*$ is a classical object; it is a polynomial algebra in infinitely many generators, cf. \cite{Milnor,Qillen}. In particular, we see
    \begin{equation}\label{eq: weiughewourbg}
    \Omega^{\rm U}_* \otimes_{\bbZ} \bC \cong \bC[\underline b] \equiv \bC[b_1,\ldots,b_n,\ldots],\;\;\deg(b_n)=2n.
    \end{equation}
As a first step to understanding $\bbF^{\rm U}$ in general, the present article discards torsion and studies its complexification (or $\bbC$-linearization), i.e., we study $\bbF^{\rm U}\otimes_\bbS\bbC$ as a module over 
    \begin{equation}\label{eqn:complexification}
    {\rm MU}\otimes_\bbS\bbC\simeq\Omega^{\rm U}_* \otimes_{\bbZ} \bC\simeq\bbC[\underline{b}].
    \end{equation}
Cf. Theorem \ref{thm:main} and Corollary \ref{cor:main} for our main results, Subsection \ref{subsec:applications} for our applications, and Subsection \ref{subsec:methods} for an overview of our methods.
    
\subsection{Main results}
Setting each $b_\rho$ to zero, i.e., base changing $\bbF^{\rm U}\otimes_{\bbS}\bbC$ along $\bbC[\underline{b}]\to\bbC$, recovers $SC^{-*}(M;\bbC)$, so we may view $\bbF^{\rm U}\otimes_{\bbS}\bbC$ as a deformation of $SC^{-*}(M;\bbC)$; it is a sheaf on the infinite-dimensional affine space $\mathrm{Spec}(\Omega^{\rm U}_* \otimes_\bbZ \bC)$ whose restriction to 0 recovers $SC^{-*}(M;\bbC)$.

Meanwhile, a classical source of deformations of ordinary symplectic cohomology (or the wrapped Fukaya category), which arise naturally in mirror symmetry, are \emph{bulk-deformations}, cf. \cite{FOOO,FOOO:ToricII,AA,Sheridan}. Let $\frakU$ be a cohomology class of degree $\ell\geq3$. By choosing a smooth cycle Poincar\'e dual to $\frakU$ and counting Floer trajectories with $q\geq0$ marked points which are constrained to lie on the aforementioned smooth cycle (weighted by $\hbar^q$, where $\hbar$ is a formal variable of degree $2-\ell$), we obtain a module $SH_\frakU^*(M; \bC[\hbar])$ over $\bbC[\hbar]$. Analogously, given a $\mu$-tuple $\frakU\equiv(\frakU_1,\ldots,\frakU_\mu)$, we obtain a module $SH_\frakU^*(M;\bbL)$ over $\bbL\equiv\bbC[\hbar_1,\ldots,\hbar_\mu]$. In fact, we may define bulk-deformations using a different model for (co)cycles, e.g. Morse cycles or de Rham cocyles (i.e., closed differential forms).

Our main result is that the complexification of Floer complex bordism is determined explicitly via bulk-deformations. Let $\ch(M)\in H^*(M;\bbC)$ denote the Chern character of $M$ and $\ch_\rho(M)\in H^{2\rho}(M;\bbC)$ its summand in degree $2\rho$.

\begin{thm}\label{thm:main}
Let $M$ be a graded Liouville manifold of dimension $2n$ and $\frakU^{\frakc\frakh}\equiv\big(\ch_2(M),\ldots,\ch_n(M)\big)$. We have the following isomorphism of $\bbC[\underline{b}]$-modules:
    \begin{equation}\label{eq: rwigpiweg}
    SH^*(M;{\rm MU})\otimes_\bbZ\bbC\cong SH^{-*}_{\frakU^{\frakc\frakh}}(M;\bbL)\otimes_\bbC\bbC[b_n, b_{n+1},\ldots].
    \end{equation}
\end{thm}
\begin{rem}
Some remarks about the previous theorem are in order.
\begin{enumerate}
\item Observe, since $SH^*_{\frakU^{\frakc\frakh}}(M;\bbL)$ is a module over
    \begin{equation}
    \bbL=\bbC[\hbar_1,\ldots,\hbar_{n-1}],\;\;\deg(\hbar_\rho)=-2\rho,
    \end{equation}
$SH^{-*}_{\frakU^{\frakc\frakh}}(M;\bbL)$, which has the reversed grading, is a module over 
    \begin{equation}
    \bbC[b_1,\ldots,b_{n-1}],\;\;\deg(b_\rho)=2\rho
    \end{equation}
which we identify with $\bL$ in the natural grading-reversing way.
\item Note, there is a shift on the right hand side; the $\rho$-th bulk-deformation variable $b_\rho$ corresponds to $\ch_{\rho+1}(M)$.
\item An immediate consequence is that only finitely many of the $b_\rho$'s provide non-trivial deformations. This is clear on the right hand side since there are only finitely many $\ch_\rho(M)$'s, but we are not aware of an \emph{a priori} reason for this to be true for the left hand side.
\end{enumerate}
\end{rem}

Using Theorem \ref{thm:main}, we give a computable obstruction to a further lift of Floer complex bordism to the sphere spectrum. Recall, ordinary symplectic cohomology is equipped with a pair-of-pants product and Batalin-Vilkovisky (BV) operator; these, together, determine a Lie bracket $[-,-]$, cf. \cite{Ritter:TQFT}. One can also define the Lie bracket directly.

\begin{cor}\label{cor:main}
Suppose there is a class $\alpha\in SH^*(M;\bbC)$ such that 
    \begin{equation}\label{eq:18}
    \big[\alpha,\ch_{\rho_0}(M)\big]\neq0
    \end{equation}
for some $\rho_0$, then $\bbF^{\rm U}$ is not homotopy equivalent to $\frakY\otimes_\bbS{\rm MU}$ for any spectrum $\frakY$.

\end{cor}

There is a change-of-coefficients spectral sequence associated to the augmentation $\bC[\underline b] \to \bC$ which computes the homology of any cochain complex over $\bC[\underline b]$. For instance, it computes $\pi_*(\frakZ\otimes_\bbS\bbC)$ for an ${\rm MU}$-module $\frakZ$, and non-degeneracy obstructs $\frakZ$ from being homotopy equivalent to $\frakY\otimes_\bbS{\rm MU}$. Now, Corollary \ref{cor:main} follows by showing that the change-of-coefficients spectral sequence associated to the cochain complex underlying $SH^*(M;{\rm MU})\otimes_\bbZ\bbC$ agrees with that of the cochain complex underlying $SH^{-*}_{\frakU^{\frakc\frakh}}(M;\bbL)\otimes_\bbC\bbC[b_n,b_{n+1},\ldots]$ (from the $E_2$-page onwards). Meanwhile, we can compute enough to see that $\big[-,\ch_{\rho_0}(M)\big]$ contributes non-trivially to the differential on the $E_{2\rho_0-1}$-page.
    
\begin{rem}
We would like to emphasize that Corollary \ref{cor:main} yielding non-trivial consequences is particular to the exact setting since, in the closed setting, the BV operator, and hence $[-,-]$, vanishes.\footnote{An expected, but not constructed, example of a closed symplectic manifold where one can associate a complex-oriented flow category to it (whose morphism spaces do not admit orbifold points) is a symplectic K3 surface; here, one should choose the almost complex structure generically among the integrable ones to prevent sphere bubbling.}
\end{rem}

\subsection{Applications}\label{subsec:applications}
\subsubsection{Non-base change examples}

We provide examples to demonstrate the computability of the criterion of Corollary \ref{cor:main}.

\begin{example}\label{example:nonbasechange}
The hypothesis of Corollary \ref{cor:main} holds in the case of $M=T^*\bbC P^n$, $n\in\bbZ$ is odd and at least 3, cf. Subsection \ref{subsec:nonbasechange} for the computation. 
\end{example}

Note, $T^*\bbC P^n$ (and, more generally, any cotangent bundle) admits a stable $\bbR$-polarization, and hence admits a theory of spectral symplectic cohomology over the sphere spectrum. One concrete consequence of Example \ref{example:nonbasechange} is that, for any choice of stable $\bbR$-polarization of $T^*\bbC P^n$, we have 
    \begin{equation}\label{eqn:nonbasechange}
    \bbF^{\rm U}_{T^*\bbC P^n}\not\simeq\bbF^{\rm fr}_{T^*\bbC P^n}\otimes_\bbS{\rm MU},
    \end{equation}
where $\bbF^{\rm fr}_{T^*\bbC P^n}$ denotes the spectral symplectic cohomology of $T^*\bbC P^n$ over the sphere spectrum with respect to the given choice of stable $\bbR$-polarization, when $n$ is odd and at least 3. 

\begin{rem}
Note, the consequence detailed in \eqref{eqn:nonbasechange} did not require an explicit computation of the left or right hand side. This leads to the following natural question: given a closed smooth manifold $Q$, can one compute $\bbF^{\rm U}_{T^*Q}$ in terms of a Thom spectrum associated to an ${\rm MU}$ local system on $\calL Q$, i.e., what is the statement of the ``spectral ${\rm MU}$ Viterbo isomorphism''? Note, $T^*Q$ admits a ``standard'' stable $\bbR$-polarization via the horizontal distribution, and the ``spectral ${\rm \bbS}$ Viterbo isomorphism'' reads
    \begin{equation}
    \bbF^{{\rm fr},\std}_{T^*Q}\simeq\calL Q^{-TQ},
    \end{equation}
where $\bbF^{{\rm fr},\std}_{T^*\bbC P^n}$ denotes the spectral symplectic cohomology of $T^*\bbC P^n$ over the sphere spectrum associated to this ``standard'' choice of stable $\bbR$-polarization, cf. \cite{Bla1}, which utilizes the (co)homological proof in \cite{Abo14} of the general non-oriented/non-spin case, for a proof of this spectral equivalence in the general non-oriented/non-spin case, and cf. the references therein for an overview of partial spectral results proven beforehand. In particular, Example \ref{example:nonbasechange} shows the ${\rm MU}$ Viterbo isomorphism is, in some sense, orthogonal to the $\bbS$ Viterbo isomorphism, and in fact still open.
\end{rem}

\begin{rem}
Observe, we exclude the case $M=T^*\bbC P^n$, $n\in\bbZ$ is even and at least 2, from Example \ref{example:nonbasechange} since, in this case, $\bbC P^n$ is not spin, hence: 
    \begin{equation}
    SH^*(T^*\bbC P^n;\bbZ)\not\cong H_{2n-*}(\calL\bbC P^n;\bbZ);
    \end{equation}
this anomaly was first noticed by Kragh \cite{Kra18} and verified in an unpublished note by Seidel \cite{Sei}. In fact, the left hand side of the previous equation will be completely 2-torsion while the right hand side is not. (The Viterbo isomorphism in the case of a non-oriented/non-spin base requires twisting by a local coefficient system, cf. \cite{Abo14}.) However, this 2-torsion computation at least tells us that
    \begin{equation}
    \bbF^{\rm U}_{T^*\bbC P^n}\not\simeq\bbF^{\fr,\std}_{T^*\bbC P^n}\otimes_\bbS{\rm MU},
    \end{equation}
when $n\in\bbZ$ is even and at least 2, since the left and right hand side have different integral homology: the left hand side recovers ordinary integral symplectic cohomology, while the right hand side recovers ordinary integral symplectic cohomology with coefficients twisted by the local coefficient system coming from the non-oriented/non-spin Viterbo isomorphism, cf. \cite[Proposition 4.3]{Bla1}. In order to say the left hand side cannot be homotopy equivalent to $\frakY\otimes_\bbS{\rm MU}$, for any spectrum $\frakY$, would require further investigation.
\end{rem}

\begin{rem}
Observe, we exclude the case $M=T^*\bbC P^1$ in Example \ref{example:nonbasechange} since it has no non-trivial Chern classes. In fact, we can explicitly compute
    \begin{equation}
    \bbF^{\rm U}_{T^*\bbC P^1}\simeq\bbF^{\fr,\std}_{T^*\bbC P^1}\otimes_\bbS{\rm MU};
    \end{equation}
indeed, this relation holds for any $T^*Q$ for which $Q$ is stably framed.
\end{rem}

\subsubsection{Ample smooth divisor complements}
Now, we will explain how Example \ref{example:nonbasechange} implies the non-triviality of higher-dimensional complex cobordism classes of relative Gromov-Witten type moduli spaces in the context of a smooth complex projective variety relative to an ample smooth divisor. 

Let $Z\equiv\bbC P^n\times\bbC P^n$ be equipped with its standard K\"ahler form and consider the ample smooth $(1,1)$-hypersurface
    \begin{equation}
    D\equiv\bigg\{\big([x_0:\cdots:x_n],[y_0:\cdots:y_n]\big):\sum_{i=0}^n x_iy_i=0\bigg\};
    \end{equation}
it is straightforward to see that $T^*\bbC P^n$, with its standard Weinstein form, is Weinstein isomorphic to $M\equiv Z-D$. We denote by $S_DZ$ the circle bundle associated to the normal bundle of $D$ in $Z$. Now, in \cite{Bla26}, building on the (co)homological work of Ganatra-Pomerleano \cite{GP20,GP21} (and even earlier work of Diogo \cite{Dio12} and Diogo-Lisi \cite{DL19}, and ideas of Seidel \cite{Sei02a}), it was shown that $\bbF^{\rm U}_{T^*\bbC P^n}$ is a filtered spectrum whose associated graded is given by 
    \begin{align}
    Gr_k\bbF^{\rm U}_{T^*\bbC P^n}&\equiv{\rm cofib}\big(F_{k-1}\bbF^{\rm U}_{T^*\bbC P^n}\to F_k\bbF^{\rm U}_{T^*\bbC P^n}\big) \\
    &\simeq\begin{cases}
    \frakD\bbC P^n_+\otimes_\bbS{\rm MU}& k=1 \\
    \frakD\Sigma^{-2nk}(S_DM)_+\otimes_\bbS{\rm MU}& k>1 \\
    * & {\rm otherwise},
    \end{cases}
    \end{align}
where $\frakD(-)$ denotes taking Spanier-Whitehead dual; moreover, it was shown that the obstruction to degenerating into the associated graded is given by an element
    \begin{equation}
    \calG\calW^{\rm U}_{T^*\bbC P^n}\in\Omega^{\rm U}_0\Big(\Sigma^{-2n+1}(S_DZ)_+\otimes_\bbS\frakD\bbC P^n_+\Big)
    \end{equation}
which is related to the complex cobordism classes of higher-dimensional 1-pointed genus 0 relative Gromov-Witten type moduli spaces of $(Z,D)$ (the moduli spaces are essentially given by considering pseudoholomorphic spheres which are only allowed to intersect $D$ at a $+\infty$, where the intersection is transverse at $+\infty$, modulo $\bbR$-translations of the domain and \emph{not} $S^1$-rotations of the domain, cf. \cite[Section 5]{Bla26}). A simple observation is the following.

\begin{claim}
In the case $n\in\bbZ$ is odd and at least 3, $\calG\calW^{\rm U}_{T^*\bbC P^n}\neq0$.
\end{claim}

\begin{proof}[Proof of claim]
Cf. Subsection \ref{subsec:ampledc}.
\end{proof}

Meanwhile, we also have the following computation.

\begin{claim}
In the case $n\in\bbZ$ is odd and at least 3, $\calG\calW^{\rm U}_{T^*\bbC P^n}\otimes_{\rm MU}\bbZ=0$.
\end{claim}

\begin{proof}[Proof of claim]
Cf. Subsection \ref{subsec:ampledc}.
\end{proof}

Combining the previous two claims, we are led to the following interpretation. $\calG\calW^{\rm U}_{T^*\bbC P^n}\otimes_{\rm MU}\bbZ$ is the obstruction to $\bbF^{\rm U}\otimes_{\rm MU}\bbZ$ degenerating into its associated graded and is only governed by the oriented cobordism classes of the zero-dimensional moduli spaces defining $\calG\calW^{\rm U}_{T^*\bbC P^n}$ (this interpretation, although phrased differently here, was already known to Ganatra-Pomerleano, cf. \cite[Theorem 1.6]{GP20}), i.e., since the former vanishes and the latter does not, there must be a higher-dimensional 1-pointed genus 0 relative Gromov-Witten type moduli space which is non-trivial in complex cobordism; this answers a natural perturbation of \cite[Question 1]{Bla26}.

\begin{rem}
Let ${\rm B}$ denote another cobordism ring spectrum which receives a ring map ${\rm MU}\to{\rm B}$, e.g. ${\rm MSpin}$, ${\rm MSO}$, or ${\rm MO}$. Even though $\calG\calW^{\rm U}_{T^*\bbC P^n}$ cannot be zero, it could also be the case that $\calG\calW^{\rm U}_{T^*\bbC P^n}\otimes_{\rm MU}{\rm B}$ is non-zero, i.e., the non-triviality in complex cobordism of the higher-dimensional moduli spaces could possibly be detected in ${\rm B}$-structured cobordism. A natural question is: what is the ``minimal'' ${\rm B}$ such that $\calG\calW^{\rm U}_{T^*\bbC P^n}\otimes_{\rm MU}{\rm B}$ is non-trivial? However, we would like to emphasize that it must be a \emph{higher-dimensional} moduli space which is non-trivial in ${\rm B}$-structured cobordism due to the fact that $\calG\calW^{\rm U}_{T^*\bbC P^n}\otimes_{\rm MU}\bbZ$ vanishes, i.e., this is a Floer-homotopical phenomenon and \emph{not} a Floer-(co)homological one.
\end{rem}

\begin{rem}
\cite{Bla26} uses Assumption 1 in \emph{loc. cit.} which the pair $(Z,D)$ considered here does not fall under (in particular, it fails part (3)). This assumption was only necessary to work with framed flow categories and obtain splittings of spectra instead of ${\rm MU}$-modules, i.e., it is straightforward to see that, using the canonical complex orientations, one can repeat the entirety of \emph{loc. cit.} to obtain the analogous splitting/obstruction results for complex-oriented flow categories in order to deal with examples such as $(Z,D)$.
\end{rem}

\subsection{Methods}\label{subsec:methods}
Throughout the present article, we use the $\infty$-categorical machinery of complex-oriented flow categories due to Abouzaid-Blumberg \cite{AB24}: this is an abstraction of the structures arising for the moduli spaces of Floer trajectories (i.e., inhomogeneous pseudoholomorphic cylinders) in $X$ along with their canonical complex orientations.

\begin{rem}\label{rem:equivalence}
    It has recently been shown by Hedenlund-Oldervoll \cite{HO} that the stable $\infty$-category of complex-oriented flow categories is equivalent to the stable $\infty$-category of ${\rm MU}$-modules. 
    This builds on ideas of Abouzaid-Blumberg \cite{AB24} which show that the stable $\infty$-category of framed flow categories is equivalent to the stable $\infty$-category of spectra (i.e., $\bbS$-modules). 
\end{rem}

$\bbF^{\rm U}$ is constructed most naturally as a complex-oriented flow category, and our first offering is a concrete chain-level model for the complexification of such objects. There is an explicit formula for \eqref{eqn:complexification} via the Chern character of a stably complex closed smooth manifold $Y$:
    \begin{equation}
    [Y] \mapsto \sum_{q_1,q_2,\ldots\geq 0} \frac{1}{q_1!q_2!\cdots}\Bigg(\int_Y\ch_1(Y)^{q_1}\ch_2(Y)^{q_2}\cdots\Bigg)(b^{q_1}_1b^{q_2}_2\cdots)\in \bC[\underline b].
    \end{equation}
    
Given a complex-oriented flow category $\bbX$, we use a similar formula to produce a chain complex 
    \begin{equation}
    CM_*(\bbX,\frakc\frakh;\bC[\underline b]),
    \end{equation}
linear over $\bC[\underline b]$, by integrating the Chern characters of all the $\bbX_{xy}$'s. Now, the $\bbX_{xy}$'s are not closed, so we must make coherent choices of cochain representatives of each Chern character via Chern-Weil theory. In particular, while the coefficients strictly depend on these choices, the resulting object of the derived category of $\bbC[\underline{b}]$-modules does not; moreover, its homology groups provide a concrete model for $SH^{-*}(M; {\rm MU})\otimes_\bbZ\bC$.

Thus, we desire to compute the Chern character of each moduli space of Floer trajectories. Since the Cauchy-Riemann operator is a Dirac-type operator on a complex vector bundle over a family of cylinders, $\scrE\to\scrC$, parameterized by a base, the Chern character of a single moduli space is computed via Grothendieck-Riemann-Roch as the pushforward (or fiberwise integral) of the Chern character of $\scrE$. Note, there is some precedent for computing Chern classes of moduli spaces of holomorphic curves, e.g. in Gromov-Witten theory, cf. \cite{Coates-Givental}.

Unfortunately, it is not sufficient for our purposes to compute the Chern character of each moduli space $\bbX_{xy}$ individually. As mentioned above, the coefficients in our chain complex are not determined by the Chern characters as cohomology classes, but rather, cochain representatives of all such cohomology classes simultaneously, i.e., we require a computation of these Chern characters in a way compatible with gluing; this is a coherence problem.

To keep track of the needed coherence, we introduce a notion of \emph{Grothendieck-Riemann-Roch (GRR) flow category}; roughly, this consists of a flow category $\bbX$ such that each $\bbX_{xy}$ is equipped with a stable isomorphism of its tangent bundle to the kernel of a family of Cauchy-Riemann operators on a complex vector bundle $\pi_{xy}:\scrE_{xy}\to\bbX_{xy}\times(\bbR\times S^1)$. Such a flow category $\bbX$ naturally yields a complex-oriented flow category $\ind(\bbX)$. In particular, the flow category underlying symplectic cohomology naturally admits a GRR structure with the complex vector bundles given by the pullback of $TX$ under suitable evaluation maps and the families of Cauchy-Riemann operators induced by the linearization of the Floer equation. Finally, given a GRR flow category $\bbX$, we may again define a chain complex 
    \begin{equation}
    CM_*\big(\bbX,\pi_!\frakc\frakh(\scrE);\bbC[\underline{b}]\big),
    \end{equation}
linear over $\bbC[\underline{b}]$, by integrating all the pushforwards of all the Chern characters of all the $\scrE_{xy}$'s along all the projections $\bbX_{xy}\times(\bbR\times S^1)\to\bbX_{xy}$.

Our main technical result, from which we deduce Theorem \ref{thm:main}, is a `` homotopy coherent'' version of the Grothendieck-Riemann-Roch theorem in the context of GRR flow categories.

\begin{thm}\label{thm:tech}
Let $\bbX$ be a GRR flow category which is a directed colimit of finite GRR flow categories. We have the following isomorphism of $\bC[\underline b]$-modules:
    \begin{equation}\label{eq: eiongwiepngreping}
    HM_*\big(\bbX,\pi_!\frakc\frakh(\scrE);\bbC[\underline{b}]\big)\cong HM_*\big(\ind(\bbX),\frakc\frakh;\bC[\underline b]\big).
    \end{equation}
\end{thm}

Our proof of Theorem \ref{thm:tech} necessitates the entirety of the $\infty$-categorical machinery of \cite{AB24}; in particular, we upgrade all of the constructions discussed above to $\infty$-categories and functors between them. Using this, we prove that both sides of \eqref{eq: eiongwiepngreping}, now viewed as functors acting on $\bbX$, agree up to an autoequivalence of the category of differential graded (DG) $\bC[\underline b]$-modules. Using To\"en's classification of DG-functors \cite{Toen}, together with the classical Grothendieck-Riemann-Roch theorem, we constrain this autoequivalence enough to prove Theorem \ref{thm:tech}.

\begin{rem}
There are many proofs/versions of Grothedieck-Riemann-Roch which are index-theoretic in nature, cf. \cite{BGV04} and the references therein; however, to the knowledge of the present authors, there is no statement in the literature which incorporates the ``homotopy coherence'' required in the present article. The difficulty for carrying this out in other frameworks (e.g. using heat kernels) is analytic and lies in the compatibility with gluing; here, we bypass this with purely homotopy-theoretic methods.
\end{rem}

\begin{rem}\label{rem:eopjfpwnwGN}
We do not prove that the chain complexes underlying both sides of \eqref{eq: eiongwiepngreping}, and analogously \eqref{eq: rwigpiweg}, are quasi-isomorphic. However, we indicate in Subsection \ref{sec: opwjgfoebgjorbtgojrbt} a conjectural approach to doing so using more homotopical machinery. Conjecturally, flow categories admit a symmetric monoidal structure and the equivalences of Remark \ref{rem:equivalence} can be upgraded to be symmetric monoidal; the sketch in Subsection \ref{sec: opwjgfoebgjorbtgojrbt} requires this extra structure. In terms of $\mathrm{Spec}(\Omega^{\rm U}_* \otimes_\bbZ \bC)$, this means that Theorem \ref{thm:main} computes the cohomology of the corresponding sheaf, rather than the object of the derived category itself.
\end{rem}

\begin{rem}
If we consider the map ${\rm MU}\to{\rm KU}$, where ${\rm KU}$ is the topological complex $K$-theory spectrum, then the complexification of
    \begin{equation}
    SH^{-*}(X;{\rm KU})\equiv\pi_*\big(\bbF^{\rm U}\otimes_{\rm MU}{\rm KU}\big)
    \end{equation}
should be computable via symplectic cohomology bulk-deformed by the Todd class; this does not immediately follow from the results in the present article, but would require the chain-level lift of our homotopy coherent Grothendieck-Riemann-Roch result as described in the previous remark.
\end{rem}

\subsection*{Acknowledgments}
The authors are grateful to Shaoyun Bai, Rune Haugseng, Liam Keenan, Yonghwan Kim, 
John Pardon, Paul Seidel, Nick Sheridan, Ivan Smith, and Hiro Lee Tanaka for helpful comments and discussions. The first author was partially supported by an NSF Graduate Research Fellowship award during this work. The second author is grateful to the Max Planck Institute for Mathematics in Bonn for its hospitality and financial support, and to the EPSRC for its financial support through EPSRC standard grant EP/W015889/1. This work was partly completed while the first author was a visitor at the Max Planck Institute for Mathematics in Bonn, and the first author would like to thank the Institute for their hospitality.

\section{Complex-oriented flow categories}
\subsection{Flow categories}
\subsubsection{Stratified smooth manifolds with corners}
Given any manifold with corners $X$, there is a poset $\cP_X$ whose objects are connected corner strata such that there is a morphism $F\to F'$ if $F'$ lies in the closure of $F$. Moreover, there is a functor $\codim: \cP_X \to \bZ_{\geq 0}$ sending each object to its codimension. A morphism $X \to X'$ of manifolds with corners is a smooth map which is the inclusion of a union of disjoint boundary faces; this induces a codimension-preserving functor $\cP_X \to \cP_{X'}$.

\begin{defin}
    A category $\cP$ equipped with a functor $\codim:\cP \to \bZ_{\geq 0}$ is a \emph{model (for manifolds with corners)} if, for each $x \in \cP$, the overcategory $\cP_{/x}$ of $x$ is isomorphic to the cube $2^{\{1, \ldots, \codim(x)\}}$ with its usual poset structure.
\end{defin}

If $X$ has is a manifold with corners, $\cP_X$ is not always a model, e.g. consider the case of a teardrop. However, all manifolds considered in this paper will satisfy this condition.

\begin{defin}
    A \emph{stratified smooth manifold with corners} is a pair $(X, \cP)$, where $X$ is a manifold with corners and $\cP$ is a model equipped with a functor $F:\cP_X \to \cP$ which (1) preserves the codimension and (2) for each $x\in\cP$, induces an equivalence $\cP_{/x}\xrightarrow{\sim}\cP_{/Fx}$.
\end{defin}

\begin{defin}
    A \emph{morphism} of stratified smooth manifold with corners $(X, \cP) \to (X',\cP')$ consists of a morphism $X \to X'$ and a codimension-preserving functor $\cP \to \cP'$ which is compatible with the functors from $\cP_X$ and $\cP_{X'}$.
\end{defin}

We think of $\cP$ as providing a stratification of the corner structure of $X$. For $p \in \cP$, we write $(\partial^p X,\partial^p \cP)$ for the stratified smooth manifold with corners consisting of the union of faces in $X$ living over the object $p$ together with the stratifying category given by the overcategory $\partial^p \cP \equiv \cP_{/p}$.
\subsubsection{Flow simplices}
We recall a construction from \cite[Section 4]{AB24}. Let $\vec{\calP}\equiv(\calP_0,\cdots,\calP_n)$ be any $(n+1)$-tuple of sets, $n\geq0$; we call such a tuple an \emph{object $(n+1)$-tuple}. For any two $x\in\calP_j$ and $y\in\calP_{j'}$, where $0\leq j\leq j'\leq n$, we may construct a model $\vec{\calP}(x,y)$ as follows. 

\begin{itemize}
\item The objects are trees with only bivalent vertices whose
    \begin{enumerate}
    \item edges are labeled by elements of $\calP_k$, $j\leq k\leq j'$, such that elements on edges appear in \emph{ascending order from left to right} and the incoming resp. outgoing leaf is labeled by $x$ resp. $y$;
    \item and vertices are labeled by a subset of $\{k+1,\ldots,k'-1\}$, where the edge to the left resp. right of a vertex is labeled by an element of $\calP_k$ resp. $\calP_{k'}$. (This condition is vacuous if $n=0,1$.)
    \end{enumerate}
\item There is a morphism $T\to T'$ if 
    \begin{enumerate}
    \item $T$ is obtained from $T'$ by collapsing internal edges;
    \item the labels of the uncollapsed edges agree with the labels of the corresponding edges of $T'$;
    \item and, given any vertex $v\in T$, $v$ contains the union of the labels of the vertices $v'\in T'$ which are collapsed to it together with any $k$, $j<k<j'$, with the property that (i) $T'$ contains an edge labeled by $\calP_k$ and (ii) all edges of $T'$ labeled by an element of $\calP_k$ are collapsed to $v$. (This condition is vacuous if $n=0,1$.)
    \end{enumerate}
\item The functor $\codim:\vec{\calP}(x,y)\to\bbZ_{\geq0}$ is given by counting 
    \begin{enumerate}
    \item the number of internal edges;
    \item and, if a vertex $v\in T$ is labeled by $S\subset\{k+1,\ldots,k'-1\}$, the cardinality of $\{k+1,\ldots,k'-1\}-S$.
    \end{enumerate}
\end{itemize}
For any three $x\in\calP_j$, $y\in\calP_{j'}$, and $z\in\calP_k$, where $0\leq j\leq k\leq j'\leq n$, there is a natural functor 
    \begin{equation}
    \vec{\calP}(x,z)\times\vec{\calP}(z,y)\to\vec{\calP}(x,y)
    \end{equation}
given by concatenation. In particular, the collection
    \begin{equation}
    \big\{\vec{\calP}(x,y)\big\}_{x,y}
    \end{equation}
yields a strict 2-category which we abuse notation and denote by $\vec{\calP}$. There is a natural inclusion of strict 2-categories: 
    \begin{equation}
    \partial^j\vec{\calP}\equiv(\calP_0,\ldots,\widehat{\calP_j},\ldots,\calP_n)\hookrightarrow\vec{\calP};
    \end{equation}
it is straightforward to see that, for any $0\leq j<j'\leq n$, there is an equality of strict 2-categories, 
    \begin{equation}
    \partial^{j'}\partial^j\vec{\calP}=\partial^{j-1}\partial^{j'}\vec{\calP},
    \end{equation}
which is compatible with the inclusions.

\begin{defin}
A \emph{flow $n$-simplex} $\bbX\equiv\bbX_{0\cdots n}$ on an object $(n+1)$-tuple consists of the following data. In the following, we always take $x\in\calP_j$, $y\in\calP_{j'}$, and $z\in\calP_k$, where
    \begin{equation}
    0\leq j\leq k\leq j'\leq n,
    \end{equation}
unless otherwise specified. 
\begin{enumerate}
\item For any two $x$ and $y$, a stratified compact smooth manifold with corners 
    \begin{equation}
    \big(\bbX_{xy},\vec{\calP}(x,y)\big).
    \end{equation}
\item For any three $x$, $y$, and $z$, a morphism of stratified compact smooth manifolds with corners 
    \begin{equation}
    \bbX_{xz}\times\bbX_{zy}\hookrightarrow\bbX_{xy}
    \end{equation}
which is an embedding of a codimension 1 boundary stratum that lifts 
    \begin{equation}
    \vec{\calP}(x,z)\times\vec{\calP}(z,y)\to\vec{\calP}(x,y).
    \end{equation}
\item For any two $x$ and $y$, the codimension 1 boundary strata of $\bbX_{xy}$ are enumerated by (i) the morphisms of the above form and (2) the stratified compact submanifolds with corners 
    \begin{equation}
    \big(\partial^k\bbX_{xy},\partial^k\vec{\calP}(x,y)\big),\;\;j<k<j';
    \end{equation}
moreover, the natural associativity diagrams intertwining these commutes.
\end{enumerate}
\end{defin}

\begin{rem}
A flow $0$-simplex is simply referred to as a \emph{flow category}.
\end{rem}

\begin{notation}
In the remainder of the present article, we will frequently consider a flow $n$-simplex $\bbX\equiv\bbX_{0\cdots n}$ together with $x\in\calP_j$, $y\in\calP_{j'}$, and $z\in\calP_k$, where
    \begin{equation}
    0\leq j\leq k\leq j'\leq n.
    \end{equation}
We fix this once and for all as standing notation so that we do not have to repeat it \emph{ad infinitum}. Often, we will actually require $j<k<j'$; we will note this explicitly when it occurs.
\end{notation}

\subsubsection{Aside: virtual bundles}
In the present article, we take the following as our definition of virtual bundle.

\begin{defin}
Let $Y$ be a space. A \emph{real/complex virtual bundle} on $Y$ is a pair $E\equiv(E^+,E^-)$ such that $E^\pm\to Y$ is a real/complex vector bundle. An \emph{equivalence} of real/complex virtual bundles, written 
    \begin{equation}
    E\equiv(E^+,E^-)\cong(F^+,F^-)\equiv F,
    \end{equation}
is an isomorphism of real/complex vector bundles 
    \begin{equation}
    E^+\oplus F^-\cong E^-\oplus F^+.
    \end{equation}
\end{defin}

\begin{rem}
If $Y=*$, we refer to a virtual bundle over $*$ as a \emph{virtual vector space}.
\end{rem}

\begin{defin}\label{defin:connection2}
A connection on a virtual bundle $E\equiv(E^+,E^-)$ is a connection on both $E^\pm$.
\end{defin}

\begin{rem}
Analogously to Definition \ref{defin:connection2}, we may define a metric, or a Hermitian connection (on a complex virtual bundle), or a...
\end{rem}

\begin{defin}
Let $E\equiv(E^+,E^-)\to Y$ be a real/complex virtual bundle. We say $E$ is \emph{orientable} if 
    \begin{equation}
    \det E^+\otimes_\bbR(\det E^-)^{-1},
    \end{equation}
where $\det(E^\pm)$ is the real determinant line bundle of $E^\pm$, is trivializable.
\end{defin}

A choice of orientation of a virtual bundle is a choice of trivialization of the aforementioned real tensor product of real determinant line bundles.

\subsubsection{Oriented flow simplices}
In the sequel, we will use the following convention.

\begin{convention}
We endow $\bbR^{\lvert\{j+1,\ldots,j'\}\rvert}$ with the standard orientation if and only if $\binom{j'-j}{2}$ is even. Observe, this forces the map
    \begin{equation}
    \bbR^{\lvert\{j+1,\ldots,k\}\rvert}+\bbR^{\lvert\{k+1,\ldots,j'\}\rvert}\to\bbR^{\lvert\{j+1,\ldots,j'\}\rvert},
    \end{equation} 
which simply concatenates, to have a Koszul sign of $(-1)^{(j'-k)(k-j)}$. Moreover, this forces the map 
    \begin{equation}
    \bbR^{\lvert\{j+1,\ldots,\widehat{k},\ldots,j'\}\rvert}+\bbR^{\lvert\{k\}\rvert}\to\bbR^{\lvert\{j+1,\ldots,j'\}\rvert},
    \end{equation}
which simply puts the $k$-th coordinate into the correct position and then concatenates, to have a Koszul sign of $(-1)^{k-j-1}$.
\end{convention}

\begin{defin}\label{defin:orientedsimplex}
An \emph{oriented flow $n$-simplex} is a flow $n$-simplex $\bbX\equiv\bbX_{0\cdots n}$ together with the following data.
\begin{enumerate}
\item A real virtual vector space $V_x$ for any $x$.
\item A real vector bundle $W(x,y)\to\bbX_{xy}$, for any two $x$ and $y$, satisfying 
    \begin{equation}
    W(x,y)\cong W(x,z)\oplus W(z,y)\;\;\textrm{over}\;\;\bbX_{xz}\times\bbX_{zy}
    \end{equation}
such that the natural diagram encoding associativity commutes.
\item An oriented real virtual bundle $I(x,y)\to\bbX_{xy}$, for any two $x$ and $y$, satisfying 
    \begin{equation}
    I(x,y)\cong I(x,z)\oplus I(z,y)\;\;\textrm{over}\;\;\bbX_{xz}\times\bbX_{zy},
    \end{equation}
where the equivalence is an an equivalence of underlying real virtual bundles which has the Koszul sign $(-1)^{\abs{z}}$ when the right hand side has the product orientation, such that the natural diagram encoding associativity commutes.\footnote{Cf. Remark \ref{rem:koszulsignoriented} for a discussion about this Koszul sign.}
\item An equivalence
    \begin{equation}\label{eqn:wxy1oriented}
    T\bbX_{xy}+\underline{V}_y+\underline{\bbR}+W(x,y)\cong I(x,y)+W(x,y)+\underline{\bbR}^{\lvert\{j+1,\ldots,j'\}\rvert}+\underline{V}_x,
    \end{equation}
for any two $x$ and $y$, such that the natural diagram associated to any codimension 1 boundary stratum of $\bbX_{xy}$ commutes.
\end{enumerate}
\end{defin}

\cite[Theorem 1.6]{AB24} shows there is a stable $\infty$-category $\flow^{\rm SO}$ whose $n$-simplices consist of oriented flow $n$-simplices.\footnote{In fact, $\flow^{\rm SO}$ is equivalent to the stable $\infty$-category of ${\rm MSO}$-modules, cf. \cite{HO}.} However, we will require a slightly more refined definition.

\begin{defin}\label{defin:cohorientedsimplex}
A \emph{decorated oriented flow $n$-simplex} is an oriented flow $n$-simplex $\bbX\equiv\bbX_{0\cdots n}$ together with the following additional data.
\begin{enumerate}
\item An orientation of $V_x$ for any $x$.
\item An equivalence of real virtual bundles
    \begin{equation}\label{eqn:wxy2oriented}
    T\bbX_{xy}+\underline{V}_y+\underline{\bbR}\cong I(x,y)+\underline{\bbR}^{\lvert\{j+1,\ldots,j'\}\rvert}+\underline{V}_x,
    \end{equation}
for any two $x$ and $y$, such that the natural diagram associated to any codimension 1 boundary stratum of $\bbX_{xy}$ commutes. Moreover, the sum of \eqref{eqn:wxy2oriented} with the identity on $W(x,y)$ recovers \eqref{eqn:wxy1oriented}.
\end{enumerate}
\end{defin}

Again, \cite[Theorem 1.6]{AB24} shows there is a stable $\infty$-category $\flow^{\rm decSO}$ whose $n$-simplices consist of decorated oriented flow $n$-simplices. Fortunately, $\flow^{\rm decSO}$ is really only cosmetically different than $\flow^{\rm SO}$.

\begin{lem}\label{lem:obviousforgetful}
The obvious forgetful functor 
	\begin{equation}
	\flow^{\rm decSO}\to\flow^{\rm SO}
	\end{equation}
is an equivalence of stable $\infty$-categories.
\end{lem}
\begin{proof}[Proof sketch]
Consider the unit flow category $\unit$ with a single object $*$ and no morphisms. We may upgrade $\unit$ to an oriented flow category $\Sigma^d\unit$, $d\in\bbZ$, via defining $V_*\equiv\bbR^d$. We may go even further and upgrade $\Sigma^d\unit$ to a decorated oriented flow category by fixing the standard orientation on $\bbR^d$. Now, the results of \cite[Section 8]{AB24} imply $\flow^{\rm SO}(\unit,\unit)$ resp. $\flow^{\rm decSO}(\unit,\unit)$ is an associative ring spectrum; in fact, as stable $\infty$-categories, we have
    \begin{equation}
    \flow^{\rm SO}\simeq\Mod_{\flow^{\rm SO}(\unit,\unit)}\;\;\textrm{and}\;\;\flow^{\rm decSO}\simeq\Mod_{\flow^{\rm decSO}(\unit,\unit)}.
    \end{equation}
In particular, $\flow^{\rm SO}$ resp. $\flow^{\rm decSO}$ is generated under homotopy colimits by $\unit$.  Finally, the obvious forgetful functor (1) sends $\unit$ to $\unit$, (2) preserves homotopy colimits (this may be checked explicitly since homotopy cofibers resp. coproducts in both the domain and codomain are given by a mapping cone construction resp. disjoint union), and (3) induces an equivalence of associative ring spectra: 
    \begin{equation}
    \flow^{\rm decSO}(\unit,\unit)\simeq\flow^{\rm SO}(\unit,\unit).
    \end{equation}
The aforementioned equivalence may be checked by (1) taking the associative ring map induced by the forgetful functor and (2) observing the induced map on homotopy groups, which are both the oriented bordism groups, is an equivalence.
\end{proof}

\begin{rem}\label{rem:koszulsignoriented}
In part (3) of Definition \ref{defin:orientedsimplex}, we require the Koszul sign $(-1)^{\abs{z}}$ for the following reason. Recall, the homotopy groups of the mapping spectrum $\flow^{\rm SO}(\unit,\unit)$ resp. $\flow^{\rm decSO}(\unit,\unit)$ are 
    \begin{align}
    \pi_d\flow^{\rm SO}(\unit,\unit)\equiv{\rm HoFlow}^{\rm SO}(\Sigma^d\unit,\unit)&\;\;\textrm{resp.} \\
    \pi_d\flow^{\rm decSO}(\unit,\unit)\equiv{\rm HoFlow}^{\rm cohSO}(\Sigma^d\unit,\unit)&,
    \end{align}
where ${\rm HoFlow}^{\rm SO}$ resp. ${\rm HoFlow}^{\rm cohSO}$  denotes the homotopy category. It is straightforward to see that our choice of Koszul sign ensures 
    \begin{equation}
    \pi_d\flow^{\rm SO}(\unit,\unit)\cong\pi_d\flow^{\rm decSO}(\unit,\unit)\cong\Omega^{\rm SO}_d,
    \end{equation}
where the latter is the degree $d$ oriented bordism group. Meanwhile, if we had instead asked that
    \begin{equation}
    I(x,z)\oplus I(z,y)\cong I(x,y)
    \end{equation}
was an equivalence of oriented real virtual bundles, we would obtain the wrong homotopy groups (for instance, each homotopy group would be 2-torsion).
\end{rem}

Consider a decorated oriented flow $n$-simplex $\bbX\equiv\bbX_{0\cdots n}$. Of course, the equivalence 
    \begin{equation}
    T\bbX_{xy}+\underline{V}_y+\underline{\bbR}\cong I(x,y)+\underline{\bbR}^{\lvert\{j+1,\ldots,j'\}\rvert}+\underline{V}_x
    \end{equation}
induces a choice of orientation on $\bbX_{xy}$; therefore, the product 
    \begin{equation}
    \bbX_{xz}\times\bbX_{zy}
    \end{equation}
may be given the product orientation. Meanwhile, by thinking of the aforementioned product as a codimension 1 boundary stratum, 
    \begin{equation}
    \bbX_{xz}\times\bbX_{zy}\hookrightarrow\bbX_{xy},
    \end{equation}
we may induce an \emph{a priori} different orientation on the product via the Stoke's orientation (i.e., via the outward pointing normal).

\begin{lem}\label{lem:koszulsign1}
The Koszul sign associated to comparing the product orientation to the boundary orientation of $\bbX_{xz}\times\bbX_{zy}$ has exponent
    \begin{multline}
    \abs{z}\big(1+\dim I(z,y)+j'-k\big)+\abs{x}\big(\dim I(z,y)+j'-k\big)+ \\
    \dim I(z,y)(k-j)+\abs{z}+(j'-k)(k-j)-\dim\bbX_{zy}-1.
    \end{multline}
\end{lem}

\begin{proof}
We consider the associativity diagram
    \begin{equation}
    \begin{tikzcd}
    \begin{gathered}T\bbX_{xz}+\underline{\bbR}+\\T\bbX_{zy}+\underline{V}_y+\underline{\bbR}\end{gathered} \arrow[r]\arrow[d] & \begin{gathered}I(x,z)+\underline{\bbR}^{\lvert\{j+1,\ldots,k\}\rvert}+\\\underline{V}_x+I(z,y)+\underline{\bbR}^{\lvert\{j+1,\ldots,k\}\rvert}\end{gathered} \arrow[d] \\
    T\bbX_{xy}+\underline{V}_y+\underline{\bbR} \arrow[r] & I(x,y)+\underline{\bbR}^{\lvert\{j+1,\ldots,j'\}\rvert}+\underline{V}_x.
    \end{tikzcd}
    \end{equation}
If we keep track of the Koszul sign, then, starting in the top left corner, going down and then right has the Koszul sign
    \begin{equation}
    (-1)^{1+\Theta+\dim\bbX_{zy}},
    \end{equation}
where $\Theta\in\bbZ$ is the unknown Kozsul sign comparing the product and boundary orientation. Meanwhile, going right and then down has the Koszul sign
    \begin{equation}
    (-1)^{\abs{z}(1+\dim I(z,y)+j'-k)+\abs{x}(\dim I(z,y)+j'-k)+\dim I(z,y)(k-j)+\abs{z}+(j'-k)(k-j)}.
    \end{equation}
\end{proof}

\begin{cor}\label{cor:koszulsign1}
Suppose $\dim\bbX_{xz}$, $\dim\bbX_{zy}$, $\dim I(x,z)$, and $\dim I(z,y)$ are all even, then the Koszul sign associated to comparing the product orientation to the boundary orientation of $\bbX_{xz}\times\bbX_{zy}$ is
    \begin{equation}
    (-1)^{j'-k+1}.
    \end{equation}
\end{cor}

\begin{proof}
Starting from the conclusion of the previous lemma, use the fact that, by assumption, 
    \begin{equation}
    \abs{x}-\abs{z}+k-j-1\equiv0\mod 2.
    \end{equation}
\end{proof}

Now, consider a face $\partial^{\langle 0\cdots\widehat{k}\cdots n\rangle}\bbX$ of $\bbX$, $0<k<n$. Of course, the equivalence 
    \begin{equation}
    T\partial^{\langle 0\cdots\widehat{k}\cdots n\rangle}\bbX_{xy}+\underline{V}_y+\underline{\bbR}\cong I(x,y)+\underline{\bbR}^{\lvert\{j+1,\ldots,\widehat{k},\ldots,j'\}\lvert}+\underline{V}_x
    \end{equation}
induces a choice of orientation on $\partial^{\langle 0\cdots\widehat{k}\cdots n\rangle}\bbX_{xy}$; we refer to this as the ``face orientation''. Meanwhile, by thinking of \emph{idem} as a codimension 1 boundary stratum,
    \begin{equation}
    \partial^{\langle 0\cdots\widehat{k}\cdots n\rangle}\bbX_{xy}\hookrightarrow\bbX_{xy},
    \end{equation}
we may induce an \emph{a priori} different orientation on $\partial^{\langle 0\cdots\widehat{k}\cdots n\rangle}\bbX_{xy}$.

\begin{lem}\label{lem:koszulsign2}
The Koszul sign associated to comparing the face orientation to the boundary orientation of $\partial^{\langle 0\cdots\widehat{k}\cdots n\rangle}\bbX_{xy}$ has exponent
    \begin{equation}
    \abs{y}+1+\abs{x}+k-j.
    \end{equation}
\end{lem}

\begin{proof}
The proof follows by arguing analogously to the proof of Lemma \ref{lem:koszulsign1} using the associativity diagram
    \begin{equation}
    \begin{tikzcd}
    \partial^{\langle 0\cdots\widehat{k}\cdots n\rangle}\bbX_{xy}+\underline{\bbR}^{\lvert\{k\}\rvert}+\underline{V}_y+\underline{\bbR} \arrow[d]\arrow[r] & I(x,y)+\underline{\bbR}^{\lvert\{j+1,\ldots,\widehat{k},\ldots,j'\}\lvert}+\underline{V}_x+\underline{\bbR}^{\lvert\{k\}\rvert} \arrow[d] \\
    T\bbX_{xy}+\underline{V}_y+\underline{\bbR} \arrow[r] & \underline{\bbR}^{\lvert\{j+1,\ldots,j'\}\lvert}+\underline{V}_x.
    \end{tikzcd}
    \end{equation}
\end{proof}

\begin{cor}\label{cor:koszulsign2}
Suppose $\dim\bbX_{xy}$ and $\dim I(x,y)$ are both even, then the Koszul sign associated to comparing the face orientation to the boundary orientation of $\partial^{\langle 0\cdots\widehat{k}\cdots n\rangle}\bbX_{xy}$ is
    \begin{equation}
    (-1)^{j'-k}.
    \end{equation}
\end{cor}

\begin{proof}
Starting from the conclusion of the previous lemma, use the fact that, by assumption, 
    \begin{equation}
    \abs{x}-\abs{y}+j'-j-1\equiv0\mod 2.
    \end{equation}
\end{proof}

\subsubsection{Complex-oriented flow simplices}
\begin{defin}\label{def: cx flow cat}
A \emph{complex-oriented flow $n$-simplex} is a flow $n$-simplex $\bbX\equiv\bbX_{0\cdots n}$ together with the following data.
\begin{enumerate}
\item A real virtual vector space $V_x$ for any $x$.
\item A real vector bundle $W(x,y)\to\bbX_{xy}$, for any two $x$ and $y$, satisfying 
    \begin{equation}
    W(x,y)\cong W(x,z)\oplus W(z,y)\;\;\textrm{over}\;\;\bbX_{xz}\times\bbX_{zy},
    \end{equation}
such that the natural diagram encoding associativity commutes.
\item A complex virtual bundle $I(x,y)\to\bbX_{xy}$, for any two $x$ and $y$, satisfying 
    \begin{equation}
    I(x,y)\cong I(x,z)\oplus I(z,y)\;\;\textrm{over}\;\;\bbX_{xz}\times\bbX_{zy},
    \end{equation}
such that the natural diagram encoding associativity commutes.
\item An equivalence
    \begin{equation}\label{eqn:wxy1complex}
    T\bbX_{xy}+\underline{V}_y+\underline{\bbR}+W(x,y)\cong I(x,y)+W(x,y)+\underline{\bbR}^{\lvert\{j+1,\ldots,j'\}\rvert}+\underline{V}_x,
    \end{equation}
for any two $x$ and $y$, such that the natural diagram associated to any codimension 1 boundary stratum of $\bbX_{xy}$ commutes.
\end{enumerate}
\end{defin}

\cite[Theorem 1.6]{AB24} shows there is a stable $\infty$-category $\flow^{\rm U}$ whose $n$-simplices consist of complex-oriented flow $n$-simplices.\footnote{In fact, $\flow^{\rm U}$ is equivalent to the stable $\infty$-category of ${\rm MU}$-modules, cf. \cite{HO}.} Again, we require a slightly more refined definition.

\begin{defin}\label{def: dec cx flow simp}
A \emph{decorated complex-oriented flow $n$-simplex} is a complex-oriented flow $n$-simplex $\bbX\equiv\bbX_{0\cdots n}$ together with the following additional data.
\begin{enumerate}
\item An orientation of $V_x$ for any $x$.
\item An equivalence of real virtual bundles
    \begin{equation}\label{eqn:wxy2complex}
    T\bbX_{xy}+\underline{V}_y+\underline{\bbR}\cong I(x,y)+\underline{\bbR}^{\lvert\{j+1,\ldots,j'\}\rvert}+\underline{V}_x,
    \end{equation}
for any two $x$ and $y$, such that the natural diagram associated to any codimension 1 boundary stratum of $\bbX_{xy}$ commutes. Moreover, the sum of \eqref{eqn:wxy2complex} with the identity on $W(x,y)$ recovers \eqref{eqn:wxy1complex}.
\end{enumerate}
\end{defin}

Again, \cite[Theorem 1.6]{AB24} shows there is a stable $\infty$-category $\flow^{\rm decU}$ whose $n$-simplices consist of decorated complex-oriented flow $n$-simplices. 

\begin{lem}\label{L219}
The obvious forgetful functor 
	\begin{equation}
	\flow^{\rm decU}\to\flow^{\rm U}
	\end{equation}
is an equivalence of stable $\infty$-categories.
\end{lem}

\begin{proof}
The proof is completely analogous to the proof of Lemma \ref{lem:obviousforgetful}. One difference to note is that 
    \begin{equation}
    \pi_d\flow^{\rm decU}(\unit,\unit)\cong\pi_d\flow^{\rm U}(\unit,\unit)\cong\Omega^{\rm U}_d,
    \end{equation}
where the latter is the degree $d$ complex-oriented bordism group.
\end{proof}

Consider a (decorated) complex-oriented flow $n$-simplex $\bbX\equiv\bbX_{0\cdots n}$; since $I(x,y)$ is a complex virtual bundle, it has an underlying orientable real virtual bundle which we denote by $\widetilde{I}(x,y)$.

\begin{lem}
Let $\widetilde{I}(x,y)$ have the standard orientation induced from the complex structure if and only if $\abs{y}$ is even. The equivalence of real virtual bundles
    \begin{equation}\label{eqn:koszulconventioncomplex}
    \widetilde{I}(x,y)\cong\widetilde{I}(x,z)\oplus\widetilde{I}(z,y)
    \end{equation}
which is induced from the corresponding equivalence of complex virtual bundles has the Koszul sign $(-1)^{\abs{z}}$ when the right hand side has the product orientation.
\end{lem}

\begin{proof}
Let $\widetilde{I}^\pm(x,y)$ be $\widetilde{I}(x,y)$ with the standard resp. opposite orientation induced from the complex structure. We have four possibilities depending on the parity of $\abs{z}$ and $\abs{y}$.
\begin{itemize}
\item If $\abs{z}$ is even and $\abs{y}$ is even, then
    \begin{equation}
    \widetilde{I}^+(x,y)\cong\widetilde{I}^+(x,z)\oplus\widetilde{I}^+(z,y)
    \end{equation}
is orientation-preserving. 
\item If $\abs{z}$ is even and $\abs{y}$ is odd, then
    \begin{equation}
    \widetilde{I}^-(x,y)\cong\widetilde{I}^+(x,z)\oplus\widetilde{I}^-(z,y)
    \end{equation}
is orientation-preserving. 
\item If $\abs{z}$ is odd and $\abs{y}$ is even, then
    \begin{equation}
    \widetilde{I}^+(x,y)\cong\widetilde{I}^-(x,z)\oplus\widetilde{I}^+(z,y)
    \end{equation}
is orientation-reversing. 
\item If $\abs{z}$ is odd and $\abs{y}$ is odd, then
    \begin{equation}
    \widetilde{I}^-(x,y)\cong\widetilde{I}^-(x,z)\oplus\widetilde{I}^-(z,y)
    \end{equation}
is orientation-reversing. 
\end{itemize}
In particular, whether or not \eqref{eqn:koszulconventioncomplex} is orientation-preserving or orientation-reversing with respect to the orientation rule in the statement of this lemma is  completely determined by the parity of $\abs{z}$, as desired.
\end{proof}

Now, the following result is immediate. 

\begin{cor}
There are obvious base change functors
    \begin{equation}
    \flow^{\rm U}\to\flow^{\rm SO}\;\;\textrm{and}\;\;\flow^{\rm decU}\to\flow^{\rm decSO},
    \end{equation}
induced by the previous lemma, which commute with the forgetful functors. 
\end{cor}

\subsection{de Rham theory for flow categories}

In this section, we will work with de Rham cochains (i.e., differential forms) and de Rham cohomology unless otherwise specified. We work over $\bC[\underline b] \equiv \bC[b_1,\ldots,b_n,\ldots]$, with $\deg(b_n)=2n$; this is a model for $H_*(BU; \bC)$.

\begin{rem}
Let $M$ be a smooth manifold. De Rham cochains $\Omega^*(M;\bC[\underline b])$ over $\bC[\underline b]$ are naturally bigraded via 
    \begin{equation}
    \Omega^*(M;\bC[\underline b])^{p,q}\equiv\Omega^p(M;\bC[\underline b]_{-q}),\;\;p,q\in\bbZ
    \end{equation}
and, unless otherwise specified, we use the total grading: 
    \begin{equation}
    \Omega^r(M;\bbR)\equiv\bigoplus_{p+q=r}\Omega^p(M;\bC[\underline b]_{-q}).
    \end{equation}
\end{rem}

\begin{defin}
Let $\bbX\equiv\bbX_{0\cdots n}$ be a flow $n$-simplex.
\begin{enumerate}
\item An \emph{additively coherent cochain} on $\bbX$, written $\fraka$, is a choice of differential forms 
    \begin{equation}
    a_{xy}\in\Omega^*(\bbX_{xy};\bC[\underline b]),
    \end{equation}
for any two $x$ and $y$, satisfying the coherence relation 
    \begin{equation}
    a_{xy}\vert_{\bbX_{xz}\times\bbX_{zy}}=\pi_{\bbX_{xz}}^*a_{xz}+\pi_{\bbX_{zy}}^*a_{zy},
    \end{equation}
for any three $x$, $y$ and $z$.
\item Given $\fraka$, we may define another additively coherent cochain $d\fraka$ via 
    \begin{equation}
    (da)_{xy}\equiv da_{xy}.
    \end{equation}
\item We say $\fraka$ is an \emph{additively coherent cocycle} if each $a_{xy}$ is closed.
\item We say $\fraka$ is \emph{degree} $r$ if each $a_{xy}$ is degree $r$.
\end{enumerate}
\end{defin}

\begin{defin}
Let $\bbX\equiv\bbX_{0\cdots n}$ be a flow $n$-simplex.
\begin{enumerate}
\item A \emph{multiplicatively coherent cochain} on $\bbX$, written $\frakm$, is a choice of differential forms 
    \begin{equation}
    m_{xy}\in\Omega^*(\bbX_{xy};\bC[\underline b]),
    \end{equation}
for any two $x$ and $y$, satisfying the coherence relation 
    \begin{equation}
    m_{xy}\vert_{\bbX_{xz}\times\bbX_{zy}}=\pi_{\bbX_{xz}}^*m_{xz}\wedge\pi_{\bbX_{zy}}^*m_{zy},
    \end{equation}
for any three $x$, $y$, and $z$.
\item We say $\frakm$ is a \emph{multiplicatively coherent cocycle} if each $m_{xy}$ is closed.
\end{enumerate}
\end{defin}

We may turn additively coherent cocycles into multiplicatively coherent ones.

\begin{defin}
Let $\fraka$ be an additively coherent cochain on $\bbX$ of degree 0; we define a multiplicatively coherent cochain $\exp\fraka$ on $\bbX$ via 
    \begin{equation}\label{eqn:exponentialsum}
    (\exp a)_{xy}\equiv\exp a_{xy}\equiv\sum_{n\geq0}\dfrac{1}{n!}a_{xy}^i.
    \end{equation}
We refer to $\exp\fraka$ as the \emph{exponential} of $\fraka$.
\end{defin}

Note, when defining $\exp\fraka$, since we assumed $\fraka$ was degree 0, \eqref{eqn:exponentialsum} is finite for degree reasons (except for the de Rham degree 0 part, but here it converges because it is the usual exponential map applied to a smooth function on $\bbX_{xy}$). Also, an elementary argument shows that, if $\fraka$ was a cocycle, so is $\exp\fraka$.

\subsection{Vector bundles on flow categories}
For any $r\in\bbZ_{\geq0}$, we write 
    \begin{equation}
    [r]\equiv\{1,\ldots,r\}.
    \end{equation}
    
\begin{defin}
Let $\bbX\equiv\bbX_{0\cdots n}$ be a flow $n$-simplex. \emph{Stabilization data} on $\bbX$, written $\vec{r}$, consists of the following data. 
\begin{enumerate}
\item For any two $x$ and $y$, an integer 
    \begin{equation}
    r_{xy}\in\bbZ_{\geq0}\;\;\textrm{resp.}\;\;r_{xy,k}\in\bbZ_{\geq0}.
    \end{equation}
\item For any three $x$, $y$, and $z$, an injection 
    \begin{equation}\label{eqn:stabilizationdatabreak}
    [r_{xz}]\amalg[r_{zy}]\hookrightarrow[r_{xy}]
    \end{equation}
such that the natural associativity diagram commutes.
\item For any two $x$ and $y$, an inclusion 
    \begin{equation}\label{eqn:stabilizationdatarestrict}
    [r_{xy,k}]\hookrightarrow[r_{xy}]
    \end{equation}
such that the natural associativity diagram commutes.
\item We require the natural associativity diagram intertwining \eqref{eqn:stabilizationdatabreak} and \eqref{eqn:stabilizationdatarestrict} commutes.
\end{enumerate}
\end{defin}

\begin{notation}
In the sequel, expressions of the form $\underline\bC^{r-r'}$ are shorthand for $\underline\bC^{[r]-[r']}$ with respect to the relevant injection $[r']\hookrightarrow[r]$.
\end{notation}

\begin{defin}
Let $\bbX\equiv\bbX_{0\cdots n}$ be a flow $n$-simplex. A \emph{real/complex vector bundle} on $\bbX$ with stabilization data $\vec{r}$, written $E\to\bbX$, consists of the following data.
\begin{enumerate}
\item For any two $x$ and $y$, a real/complex vector bundle
    \begin{equation}
    E_{xy}\to\bbX_{xy}\;\;\textrm{resp.}\;\;E_{xy,k}\to\partial^k\bbX_{xy}.
    \end{equation}
\item For any $x$, $z$ and $y$, an isomorphism
    \begin{equation}\label{eqn:vbbreak1}
    E_{xz}\oplus E_{zy}\oplus \underline \bC^{r_{xy}-r_{xz}-r_{zy}}\cong E_{xy} \;\;\textrm{over}\;\;\bbX_{xz}\times\bbX_{zy}
    \end{equation}
such that the natural associativity diagram commutes.
\item For any two $x$ and $y$, an isomorphism
    \begin{equation}\label{eqn:vbrestrict1}
    E_{xy,k} \oplus \underline\bC^{r_{xy}-r_{xy,k}}\cong E_{xy}\;\;\textrm{over}\;\;\partial^k\bbX_{xy}
    \end{equation}
such that the natural associativity diagram commutes.
\item We require the natural associativity diagram intertwining \eqref{eqn:vbbreak1} and \eqref{eqn:vbrestrict1} commutes.
\end{enumerate}
\end{defin}

\begin{rem}
An isomorphism of real/complex vector bundles over a flow $n$-simplex is defined in the natural way. 
\end{rem}

\begin{defin}
Let $\bbX\equiv\bbX_{0\cdots n}$ be a flow $n$-simplex. A \emph{real/complex virtual bundle} on $\bbX$, written $E\equiv(E^+,E^-)\to\bbX$, consists of the following data. 
\begin{enumerate}
\item For any two $x$ and $y$, a real/complex virtual bundle 
    \begin{equation}
    E_{xy}\equiv (E^+_{xy},E^-_{xy})\to\bbX_{xy}\;\;\textrm{resp.}\;\;E_{xy,k}\equiv(E^+_{xy,k},E^-_{xy,k})\to\partial^k\bbX_{xy}.
    \end{equation}
\item For any two $x$ and $y$, an equivalence
    \begin{equation}\label{eqn:vbbreak}
    E_{xz}\oplus E_{zy}\cong E_{xy}\;\;\textrm{over}\;\;\bbX_{xz}\times\bbX_{zy}
    \end{equation}
such that the natural associativity diagram commutes.
\item For any two $x$ and $y$, an equivalence
    \begin{equation}\label{eqn:vbrestrict}
    E_{xy,k}\cong E_{xy}\;\;\textrm{over}\;\;\partial^k\bbX_{xy}
    \end{equation}
such that the natural associativity diagram commutes.
\item We require the natural associativity diagram intertwining \eqref{eqn:vbbreak} and \eqref{eqn:vbrestrict} commutes.
\end{enumerate}
\end{defin}

\begin{rem}
An equivalence of real/complex virtual bundles over a flow $n$-simplex is defined in the natural way. 
\end{rem}

\begin{rem}
Any vector bundle $E$ with stabilization data $\vec r$ determines a virtual bundle $F$ by setting $F^+_{xy} \equiv E_{xy}$ and $F^-_{xy}\equiv \underline\bC^{r_{xy}}$.
\end{rem}

\begin{defin}\label{defin:connection}
Let $E\to\bbX$ be a vector bundle on a flow $n$-simplex with stabilization data $\vec{r}$ or a virtual bundle. A \emph{connection} on $E$ consists of the following data.
\begin{enumerate}
\item For any two $x$ and $y$, a connection $\nabla_{xy}$ resp. $\nabla_{xy,k}$ on $E_{xy}$ resp. $E_{xy,k}$.
\item For any two $x$ and $y$, an equality
    \begin{equation}\label{eqn:connectionbreak}
    \nabla_{xz}\oplus \nabla_{zy}=\varphi_{xzy}^*\nabla_{xy}
    \end{equation}
such that the natural associativity diagram commutes.
\item For any two $x$ and $y$, an equality
    \begin{equation}\label{eqn:connectionrestrict}
    \nabla_{xy,k}=\varphi_{xy,k}^*\nabla_{xy}
    \end{equation}
such that the natural associativity diagram commutes.
\item We require the natural associativity diagram intertwining \eqref{eqn:connectionbreak} and \eqref{eqn:connectionrestrict} commutes.
\end{enumerate}
In \eqref{eqn:connectionbreak} resp. \eqref{eqn:connectionrestrict} (1) in the non-virtual case, we implicitly sum the left hand side with the standard connection on the trivial bundle and (2) in the virtual case, these are equalities of connections on virtual bundles.
\end{defin}

\begin{rem}
Analogously to Definition \ref{defin:connection}, we may define a metric, or a Hermitian connection (on a complex vector bundle/virtual bundle), or a...\footnote{Moreover, a connection on a vector bundle/virtual bundle over a flow $n$-simplex always exists via induction on the dimension of $\bbX_{xy}$ since the space of connections is contractible. Analogously, we may always choose a metric, or a Hermitian connection (on a complex vector bundle/virtual bundle), or a...}
\end{rem}

\begin{defin}
Let $E$ and $E'$ be two (virtual) vector bundles over a flow $n$-simplex. A \emph{concordance} between $E$ and $E'$ is a 1-parameter family of (virtual) vector bundles interpolating between $E$ and $E'$. 
\end{defin}

\begin{lem}
Concordant vector bundles resp. virtual bundles are isomorphic resp. equivalent.
\end{lem}

\begin{proof}
Choose a connection on the concordance and parallel transport in the parameter direction.
\end{proof}

\subsection{Chern characters}
\subsubsection{The ordinary Chern character}
We provide a brief recap of the Chern-Weil construction of characteristic classes, cf. \cite[Appendix C]{MS74} for a more comprehensive account of the subject.

Let $(E,\nabla)\to M$ be a complex vector bundle over a smooth manifold equipped with a Hermitian connection; we denote by $F^\nabla\in\Omega^2\big(M; \mathrm{End}(E)\big)$ its curvature.

    \begin{defin}
    The \emph{(cochain-level) Chern character} associated to $\nabla$ is
        \begin{equation}
            \ch(E; \nabla) \equiv \mathrm{Tr}\Bigg(\exp\Big(\frac{-1}{2\pi i}F^\nabla\Big)\Bigg) \in \Omega^*(M; \bC).
        \end{equation}
    We normally drop $\nabla$ from the notation.
    \end{defin}

    \begin{rem}
    For a complex virtual bundle $E\equiv(E^+,E^-)\to M$ equipped with a Hermitian connection $\nabla\equiv(\nabla^+,\nabla^-)$, we define 
        \begin{equation}
        \ch(E; \nabla) \equiv \ch(E^+;\nabla^+)-\ch(E^-;\nabla^-)\in\Omega^*(M; \bC).
        \end{equation}
    \end{rem}
    
    Observe, the Chern character has the following properties.
    \begin{enumerate}
        \item $\ch(E)$ is a closed differential form whose cohomology class represents the ordinary Chern character:
        \begin{equation}
            \mathrm{Rank}(E) + c_1(E) + \frac12\big(c_1(E)^2-2c_2(E)\big) + 
            \cdots.
        \end{equation}
        \item For a smooth map $f: M' \to M$, we have:
        \begin{equation}
            \ch(f^*E) = f^*\ch(E).
        \end{equation}
        \item If $(E',\nabla')\to M$ is another complex vector bundle over $M$ equipped with a Hermitian connection, then:
        \begin{align}
        \ch(E\oplus E') &= \ch(E) + \ch(E'), \\
        \ch(E \otimes_\bbC E')&=\ch(E)\wedge\ch(E').
        \end{align}
    \end{enumerate}
    These properties also hold for a complex virtual bundle as well.

\subsubsection{The polynomial-valued Chern character}
Now, let $E\to M$ be a complex vector bundle/virtual bundle over a smooth manifold $M$. We denote by $\ch(E;\bC[\underline b])\in\Omega^0(M;\bC[\underline b])$ the \emph{polynomial-valued Chern character}: 
    \begin{equation}
    \ch(E;\bC[\underline b])\equiv\sum_{\rho\geq0}\ch_\rho(E)b_\rho,
    \end{equation}
where $\ch_\rho(E)$ is defined to be the $2\rho$-th graded piece of the usual Chern character: 
    \begin{equation}
    \ch(E)\equiv\sum_{\rho\geq0}\ch_\rho(E)\in\Omega^{2*}(M;\bbC).
    \end{equation}
We would like to emphasize that we are viewing $\ch(E; \bC[\underline b])$ as a cocycle as opposed to just a cohomology class.

Observe,
    \begin{equation}
    \ch(E\oplus E';\bC[\underline b])=\ch(E;\bC[\underline b])+\ch(E';\bC[\underline b]).
    \end{equation}
In particular, the ingredients in the present section can be cooked into a stew of examples of additively coherent cocycles on flow $n$-simplices. 

\begin{example}
Let $(E,\nabla)\to\bbX$ be a complex vector bundle/virtual bundle, equipped with a Hermitian connection, over a flow $n$-simplex. For any two $x$ and $y$, we define an additively coherent cocycle $\frakc\frakh(E;\bC[\underline b])$ via
    \begin{align}
    \ch(E;{\bC[\underline b]})_{xy}\equiv\ch(E_{xy};\bC[\underline b])\in\Omega^0(\bbX_{xy};\bC[\underline b])&\;\;\textrm{resp.} \\
    \ch(E;{\bC[\underline b]})_{xy,k}\equiv\ch(E_{xy,k};\bC[\underline b])\in\Omega^0(\bbX_{xy,k};\bC[\underline b])&,
    \end{align}
where $\ch(E_{xy};\bC[\underline b])$ resp. $\ch(E_{xy};\bC[\underline b])$ is computed with respect to the Hermitian connection $\nabla_{xy}$ resp. $\nabla_{xy,k}$.    
\end{example}

\section{Universal curves on flow simplices}
\subsection{Universal curves}
\begin{defin}\label{defin:abstractcurve}
An \emph{abstract curve} consists of the following data:
\begin{enumerate}
\item a genus 0 Riemann surface $\Sigma$ with two non-compact ends, where one is labeled \emph{positive} and the other is labeled \emph{negative}; 
\item and holomorphic embeddings $\epsilon^\pm:\bbR_\pm\times S^1\to\Sigma$, with disjoint images, onto neighborhoods of the positive and negative ends, respectively. We refer to these as \emph{tube-like ends} (or \emph{cylindrical ends}).
\end{enumerate}
Moreover, we require that there exists a biholomorphism $\Sigma\cong\bbR\times S^1$ such that both $\epsilon^\pm$ are given by translation in the $\bbR$-direction.
\end{defin}

We may also consider the parametric version.

\begin{defin}
Let $M$ be a smooth manifold. A \emph{(smooth) family of abstract curves} on $M$ consists of a smoothly varying family $\scrC$ of abstract curves parameterized by $M$. Note, such data includes a smooth $\bbR\times S^1$-bundle over $M$; therefore, by an abuse of notation, we simply denote all the aforementioned data by $\pi:\scrC\to M$.
\end{defin}

\begin{construction}
Let $\Sigma_1$ and $\Sigma_2$ be abstract curves and consider $\tau\in\bbR_{>0}$. We define an abstract curve 
    \begin{equation}
    \Sigma_1\#_\tau\Sigma_2
    \end{equation}
by gluing together 
    \begin{equation}
    \Sigma_1-\epsilon^{+,\circ}_1,\;\;[-1/\tau,1/\tau]\times S^1,\;\;\textrm{and}\;\;\Sigma_2-\epsilon^{-,\circ}_2
    \end{equation}
along the appropriate boundaries (here, $\epsilon^{\pm,\circ}_i$ denotes the interior of the image of $\epsilon^\pm_i$). Observe, this construction is associative: 
    \begin{equation}
    (\Sigma_1\#_\tau\Sigma_2)\#_\sigma\Sigma_3=\Sigma_1\#_\tau(\Sigma_2\#_\sigma\Sigma_3).
    \end{equation}
Moreover, this construction works in families: given $\scrC$ and $\scrC'$ over $M$, we obtain a family of abstract curves 
    \begin{equation}
    \scrC\#_\tau\scrC'\to M\times(0,+\infty)
    \end{equation}
with fiber $\scrC_x\#_\tau\scrC_x'$ over $(x,\tau)$; again, this construction is associative.
\end{construction}

Let $\bbX\equiv\bbX_{0\cdots n}$ be a flow $n$-simplex. Recall, we may choose a system of collars on $\bbX$:
    \begin{align}
    \big\{\varphi_{xzy}:\bbX_{xz}\times\bbX_{zy}\times[0,\epsilon)\to\bbX_{xy}\big\}\amalg\big\{\varphi_{xy,k}:\partial^k\bbX_{xy}\times[0,\epsilon)\to\bbX_{xy}\big\},
    \end{align}
where the first index set is over all $x\in\calP_j$, $y\in\calP_{j'}$, and $z\in\calP_k$ satisfying $0\leq j\leq k\leq j'\leq n$ and the second index set is over all $x\in\calP_j$, $y\in\calP_{j'}$, and $k$ satisfying $0\leq j<k<j'\leq n$; we denote by $\bbX^\circ_{xy}\equiv\bbX_{xy}-\partial\bbX_{xy}$ the interior of $\bbX_{xy}$.

\begin{defin}
A \emph{universal curve} over a flow $n$-simplex $\bbX$, written $\pi:\scrC\to\bbX$, consists of the following data.
\begin{enumerate}
\item For any two $x$ and $y$, a smooth family of abstract curves $\pi_{xy}:\scrC_{xy}\to\bbX^\circ_{xy}$ resp. $\pi_{xy,k}:\scrC_{xy,k}\to\big(\partial^k\bbX_{xy}\big)^\circ$.
\item For any three $x$, $y$, and $z$, an isomorphism of abstract curves over $\bbX^\circ_{xz}\times\bbX^\circ_{zy}\times(0,\epsilon)_\tau$:
    \begin{equation}\label{eqn:assocuniversalcurve1}
    \scrC_{xz}\#_{1/\tau}\scrC_{zy}\xrightarrow{\sim}\varphi_{xzy}^*\scrC_{xy}.
    \end{equation}   
We require the natural associativity condition associated to four elements commutes.
\item For any two $x$ and $y$, an isomorphism of abstract curves over $\big(\partial^k\bbX_{xy}\big)^\circ\times(0,\epsilon)$:
    \begin{equation}\label{eqn:assocuniversalcurve2}
    \varphi_{xy,k}^*\scrC_{xy,k}\xrightarrow{\sim}\scrC_{xy}.
    \end{equation}
We require the natural associativity condition associated to two comparable strata commutes.
\item We require the natural associativity condition intertwining \eqref{eqn:assocuniversalcurve1} and \eqref{eqn:assocuniversalcurve2} commutes.
\end{enumerate}
\end{defin}

\begin{rem}
It is tempting to glue all the strata together into a single object mapping to each $\bbX_{xy}$ -- this would perhaps be more deserving of the name ``universal curve''. Though this can provide useful intuition, it seems more convenient to instead work with the above data directly.
\end{rem}

\begin{defin}
Let $\pi:\scrC\to\bbX\equiv\bbX_{0\cdots n}$ be a universal curve over a flow $n$-simplex and $Y$ some space. An \emph{evaluation map} to $Y$ consists of the following data.
\begin{enumerate}
\item For any $x$, a map $f_x:S^1\to Y$.
\item For any two $x$ and $y$, a map $f_{xy}:\scrC_{xy}\to Y$ resp. $f_{xy,k}:\scrC_{xy,k}\to Y$.
\end{enumerate}
We require this data to satisfy the following conditions. 
\begin{enumerate}
\item Over the tube-like ends, $f_{xy}$ resp. $f_{xy,k}$ is independent of the $\bbR_-$-coordinate resp. $\bbR_+$-coordinate and equals $f_x$ resp. $f_y$.
\item For any three $x$, $y$, and $z$, the following diagram commutes: 
    \begin{equation}\label{eqn:universalcurvebreaking}
    \begin{tikzcd}
    \scrC_{xz}\#_{1/\tau}\scrC_{zy}\arrow[r]\arrow[dr,"f_{xz}\#_{1/\tau} f_{zy}",swap] & \scrC_{xy}\arrow[d,"f_{xy}"] \\
    & Y.
    \end{tikzcd}
    \end{equation}
\item For any two $x$ and $y$, the following diagram commutes:
    \begin{equation}\label{eqn:universalcurverestricting}
    \begin{tikzcd}
    \scrC_{xy,k}\arrow[r]\arrow[dr,"f_{xy,k}",swap] & \scrC_{xy}\arrow[d,"f_{xy}"] \\
    & Y.
    \end{tikzcd}
    \end{equation}
\item The natural associativity diagram intertwining \eqref{eqn:universalcurvebreaking} and \eqref{eqn:universalcurverestricting} commutes.
\end{enumerate}
\end{defin}

\begin{rem}
This is similar to the notion of a ``flow category living over the free loop space of $Y$'' defined in \cite[Section 4.9]{PS25b}.
\end{rem}

\subsection{Vector bundles}
\begin{defin}
Let $\bbX\equiv\bbX_{0\cdots n}$ be a flow $n$-simplex. \emph{Embedding data} on $\bbX$, written $\vec{\nu}$, consists of the following data. 
\begin{enumerate}
\item For any two $x$ and $y$, an integer 
    \begin{equation}
    \nu_{xy}\in\bbZ_{\geq0}\;\;\textrm{resp.}\;\;\nu_{xy,k}\in\bbZ_{\geq0}.
    \end{equation}
\item For any three $x$, $y$, and $z$, an injection 
    \begin{equation}\label{eqn:embeddingdatabreak}
    [\nu_{xz}],[\nu_{zy}]\hookrightarrow[\nu_{xy}]
    \end{equation}
such that the natural associativity diagram commutes.
\item For any two $x$ and $y$, an inclusion 
    \begin{equation}\label{eqn:embeddingdatarestrict}
    [\nu_{xy,k}]\hookrightarrow[\nu_{xy}]
    \end{equation}
such that then natural associativity diagram commutes.
\item We require the natural associativity diagram intertwining \eqref{eqn:embeddingdatabreak} and \eqref{eqn:embeddingdatarestrict} commutes.
\end{enumerate}
\end{defin}

\begin{defin}
Let $\pi:\scrC\to\bbX\equiv\bbX_{0\cdots n}$ be a universal curve over a flow $n$-simplex. A \emph{real/complex vector bundle (of rank $\vec{\nu}$, where $\vec{\nu}$ is embedding data on $\bbX$)} on $\scrC$, written $\scrE\to\scrC$, consists of the following data. 
\begin{enumerate}
\item For any $x$ and $y$, a real/complex vector bundle $\scrE_{xy}\to\scrC_{xy}$ resp. $\scrE_{xy,k}\to\scrC_{xy,k}$ such that we have a fixed trivialization $\scrE_{xy}\cong\underline{\bbC}^{\nu_{xy}}$ resp. $\scrE_{xy,k}\cong\underline{\bbC}^{\nu_{xy,k}}$ over the tube-like ends.\footnote{Note, in both the real and complex case, we require the trivialization over the tube-like ends to be to $\underline{\bbC}^{\nu_{xy}}$ resp. $\underline{\bbC}^{\nu_{xy,k}}$.}
\item For any three $x$, $y$, and $z$, an isomorphism 
    \begin{equation}\label{eqn:vectorbundlebreaking}
    (\scrE_{xz}\oplus\underline{\bbC}^{\nu_{xy}-\nu_{xz}})\#_{1/\tau}(\scrE_{zy}\oplus\underline{\bbC}^{\nu_{xy}-\nu_{zy}})\cong\scrE_{xy}
    \end{equation}
covering the abstract gluing map, such that the natural associativity diagram commutes.
\item For any $x$ and $y$, an isomorphism 
    \begin{equation}\label{eqn:vectorbundlerestricting}
    \scrE_{xy,k} \oplus\underline\bC^{\nu_{xy}-\nu_{xy,k}}\cong\scrE_{xy}
    \end{equation}
covering the abstract restricting map, such that the natural associativity diagram commutes.
\item The natural associativity diagram intertwining \eqref{eqn:vectorbundlebreaking} and \eqref{eqn:vectorbundlerestricting} commutes.
\end{enumerate}
\end{defin}

The following two examples are of importance to us.

\begin{example}
Let $f:\scrC\to Y$ be an evaluation map and $E\to Y$ a real/complex vector bundle equipped with a trivialization over each $f_x$: 
    \begin{equation}
    f_x^*E\cong\underline{\bbC}^\nu;
    \end{equation}
there is a real/complex vector bundle $\scrE\equiv f^*E\to\scrC$ given by 
    \begin{equation}
    \scrE_{xy}\equiv f_{xy}^*E\to\scrC_{xy}\;\;\textrm{resp.}\;\;\scrE_{xy,k}\equiv f_{xy,k}^*E\to\scrC_{xy,k}.
    \end{equation}
\end{example}

\begin{example}
The vertical (tangent) bundles of the smooth families of abstract curves defines a vector bundle $\scrE\equiv V\scrC\to\scrC$; explicitly, $\scrE$ is given by
    \begin{equation}
    \scrE_{xy}\equiv V\scrC_{xy}\to\scrC_{xy}\;\;\textrm{resp.}\;\;\scrE_{xy,k}\equiv V\scrC_{xy,k}\to\scrC_{xy,k},
    \end{equation}
where $V\scrC_{xy}$ resp. $V\scrC_{xy,k}$ is the vertical bundle of $\scrC_{xy}\to\bbX_{xy}^\circ$ resp. $\scrC_{xy,k}\to\big(\partial^k\bbX_{xy})^\circ$.
\end{example}

\subsection{de Rham theory}
Unlike for flow $n$-simplices, universal curves will only have one kind of coherent cocycle.

\begin{defin}
Let $\pi:\scrC\to\bbX\equiv\bbX_{0\cdots n}$ be a universal curve over a flow $n$-simplex. A \emph{coherent cochain} on $\scrC$, written $\frakb$, is a choice of differential forms 
    \begin{equation}
    b_{xy}\in\Omega^*(\scrC_{xy};\bC[\underline b])\;\;\textrm{resp.}\;\;b_{xy,k}\in\Omega^*(\scrC_{xy,k};\bC[\underline b]),
    \end{equation}
for any two $x$ and $y$, which vanish over the tube-like ends and satisfy the coherence relations
    \begin{equation}
    b_{xy}\vert_{\scrC_{xz}\#_{1/\tau}\scrC_{zy}}=b_{xz}+b_{zy}\;\;\textrm{and}\;\;b_{xy}\vert_{\scrC_{xy,k}}=b_{xy,k}.\footnote{Note, we may add together the right hand side of the first equality since everything is supported away from the tube-like ends.}
    \end{equation}
We say $\frakb$ is a \emph{coherent cocycle} if each $b_{xy}$ resp. $b_{xy,k}$ is closed.
\end{defin}

\begin{rem}
For a vector bundle $\scrE\to\scrC$, we may define $\scrE$-valued coherent cochains/cocycles in a completely analogous fashion.
\end{rem}

\begin{example}
Let $f:\scrC\to Y$ be an evaluation map to a smooth manifold $Y$ and $b\in\Omega^*(Y;\bC[\underline b])$; there is a coherent cocycle $\frakb\equiv f^*b$ given by 
    \begin{equation}
    b_{xy}\equiv f_{xy}^*b\in\Omega^*(\scrC_{xy};\bC[\underline b])\;\;\textrm{resp.}\;\;b_{xy,k}\equiv f_{xy,k}^*b\in\Omega^*(\scrC_{xy,k};\bC[\underline b]).
    \end{equation}
\end{example}

The previous example shows we may  
``pull''; fortunately, we may also ``push''.

\begin{lem}\label{lem:push}
Let $\pi:\scrC\to\bbX\equiv\bbX_{0\cdots n}$ be a universal curve over a decorated oriented flow $n$-simplex and $\frakb$ a coherent cocycle on $\scrC$. We have that 
    \begin{align}\label{eqn:pushforward}
    (\pi_{xy})_!b_{xy}\equiv\int_{\scrC_{xy}/\bbX_{xy}^\circ}b_{xy}\in\Omega^*(\bbX_{xy}^\circ;\bC[\underline b])&\;\;\textrm{resp.} \\
    (\pi_{xy,k})_!b_{xy,k}\equiv\int_{\scrC_{xy,k}/(\partial^k\bbX_{xy})^\circ}b_{xy,k}\in\Omega^*\Big(\big(\partial^k\bbX_{xy}\big)^\circ;\bC[\underline b]\Big)&
    \end{align}
uniquely extends to a cocycle on $\bbX_{xy}$. In fact, $\fraka\equiv \pi_!\frakb$ given by the previous equation is an additively coherent cocycle on $\bbX$. Its degree is two less than that of $\fraka$.
\end{lem}

\begin{proof}
First, we observe that integration over the fiber makes sense even though $\scrC_{xy}\to\bbX_{xy}^\circ$ has non-compact fibers since $b_{xy}$ is compactly supported. Moreover, uniqueness follows by denseness of $\bbX_{xy}^\circ$ in $\bbX_{xy}$. Finally, the existence of an extension follows by defining the extension over the image of any collar,
    \begin{equation}
    \bbX_{xz_1}^\circ\times\cdots\times\bbX_{z_\mu y}\times[0,\epsilon)^\mu\hookrightarrow\bbX_{xy}\;\;\textrm{resp.}\;\;\big(\partial^{k_\mu}\cdots\partial^{k_1}\bbX_{xy}\big)^\circ\times[0,\epsilon)^\mu\hookrightarrow\bbX_{xy},
    \end{equation}
to simply be 
    \begin{equation}
    (\pi_{xz_1})_!b_{xz_1}+\cdots+(\pi_{z_\mu y})_!b_{z_\mu y}\;\;\textrm{resp.}\;\;(\pi_{xy,k})_!b_{xy,k};
    \end{equation}
clearly, these patch together to give an additively coherent cocycle $\fraka\equiv\pi_!\frakb$.
\end{proof}

\subsection{(Floer-)Cauchy-Riemann operators}
\begin{defin}
Let $\Sigma\cong\bbR\times S^1$ be an abstract curve.
\begin{enumerate}
\item\emph{Puncture datum (of rank $r$)} consists of a translation-invariant $(0,1)$-form $e\in\Omega^{0,1}\big(\bbR\times S^1,\operatorname{End}_\bbC(\bbC^r)\big)$ satisfying the following non-degeneracy condition: 
    \begin{equation}
    \overline{\partial}+e:\Gamma(\bbR\times S^1,\underline\bbC^r)\to\Omega^{0,1}\big(\bbR\times S^1,\operatorname{End}_\bbC(\underline\bbC^r)\big)
    \end{equation}
is an isomorphism.
\item Associated to a puncture datum $e$ is its \emph{Conley-Zehnder index}, denoted ${\rm CZ}(e)$. Moreover, we may define its \emph{degree}, denoted $\deg(e)\equiv r-{\rm CZ}(e)$.\footnote{Note, any two puncture datums of the same rank and degree can be connected through a path of puncture datums of the same rank and degree.}
\item The \emph{reflection} of $e$, written $\overline{e}$, is the pullback of $e$ along the complex-conjugation involution of $S^1$.
\end{enumerate}
\end{defin}

\begin{defin}\label{defin:floercauchyriemann}
Let $\Sigma$ be an abstract curve and $(e,e')$ a pair of puncture datums. A \emph{Floer-Cauchy-Riemann} operator on $\Sigma$ with asymptotes $(e,e')$ consists of the following data.
\begin{enumerate}
\item A complex vector bundle $E\to\Sigma$ which is trivialized over the tube-like ends.
\item A Riemannian metric $g$ on $\Sigma$ and a Hermitian metric $h$ on $E$ which are both trivial over the tube-like ends. 
\item A Hermitian connection $\nabla$ on $E$ which agrees with the standard connection $d$ over the tube-like ends with respect to the given trivializations of $E$ over the tube-like ends. We write $\overline{\partial}_E$ for the $(0,1)$-part of $\nabla$. 
\item A $(0,1)$-form $Y\in\Omega^{0,1}(\Sigma,E)$ which agrees with $e$ resp. $e'$ over the negative resp. positive tube-like end.
\end{enumerate}
From this data, we may define a real (or complex if $Y$ is trivial) Fredholm operator
    \begin{equation}
    D\equiv\overline{\partial}_E+Y:W^{\mu,2}(\Sigma;E)\to W^{\mu-1,2}\big(\Sigma;\Lambda^{0,1}_\Sigma\otimes_\bbC E\big),\;\;\mu\geq2
    \end{equation}
of index $\deg(e)-\deg(e')$.
\end{defin}
\begin{rem}
We refer to a Floer-Cauchy-Riemann operator such that $e$, $e'$, and $Y$ are all trivial simply as a \emph{Cauchy-Riemann} operator.
\end{rem}
Once we fix $E\to\Sigma$ and the puncture datum, every other choice lives in a convex space; hence, we have the following result.

\begin{lem}\label{lem:contractiblefcr}
The space of (Floer-)Cauchy-Riemann operators on $E$ with fixed puncture datum is contractible.
\end{lem}

\begin{defin}
A \emph{stabilization (of rank $r$)} of a Floer-Cauchy-Riemann operator $D$ is a real-linear map 
    \begin{equation}
    f:\bbC^r\to W^{\mu-1,2}\big(\Sigma;\Lambda^{0,1}_\Sigma\otimes_\bbC E\big),
    \end{equation}
whose image is contained in $W^{\mu-1,2}$-sections supported away from the tube-like ends, such that $D+f$ is surjective. For a Cauchy-Riemann operator, we require such a stabilization is complex-linear.
\end{defin}

\begin{defin}
Let $M$ be a smooth manifold, $\pi:\scrC\to M$ a family of abstract curves, $E\to\scrC$ a complex vector bundle which is trivialized over the tube-like ends, and $(e,e')$ a pair of puncture datums. A \emph{(smooth) family of (Floer-)Cauchy-Riemann operators} on $\scrC$, with respect to $E$ and asymptotes $(e,e')$, consists of a smoothly-varying family of tuples,
    \begin{equation}
    \big\{(E_p\equiv E\lvert_{\pi^{-1}(p)},g_p,h_p,\nabla_p,Y_p)\big\}_{p\in M},
    \end{equation}
as in Definition \ref{defin:floercauchyriemann} together with a Hermitian connection $\nabla$ on $E$ whose restriction to $E_p$ is $\nabla_p$.
\end{defin}

The space of stabilizations for a fixed (Floer-)Cauchy-Riemann operator is not contractible. However, the following result will allow us to make inductive choices; this is the appropriate analogue of Lemma \ref{lem:contractiblefcr} for stabilizations.

\begin{lem}
Let $D$ be a family of (Floer-)Cauchy-Riemann operator on a complex vector bundle on a family of abstract curves parameterized by the disk $D^{\ell+1}$. For any (smooth) family $\{f_\phi\}_{\phi\in S^\ell}$ of stabilizations (of rank $r$) of $D\vert_{S^\ell}$, there is $r'\geq r$ and a family of stabilizations (of rank $r'$) $\{g_\phi\}_{\phi\in D^{\ell+1}}$ of $D$ such that 
    \begin{equation}
    g_\phi=f_\phi\oplus0,\;\;\forall\phi\in S^\ell.
    \end{equation}
\end{lem}

\begin{proof}
For any family of linear maps of finite-dimensional vector bundles $\{h_\phi:F\to G\}_{\phi\in D^{\ell+1}}$ such that $h_\phi$ is surjective for all $\phi\in S^\ell$, there exists $r'\gg0$ and a family of linear maps of finite-dimensional vector bundles $\{k_\phi:\underline\bbR^{r'}\to G\}_{\phi\in D^{\ell+1}}$ such that $h_\phi\oplus k_\phi$ is surjective for all $\phi\in D^{\ell+1}$ and $k_\phi=0$ for all $\phi \in S^\ell$. The lemma follows since (Floer-)Cauchy-Riemann operators are Fredholm -- thus, have finite-dimensional cokernel.
\end{proof}

The kernel of a stabilized (Floer-)Cauchy-Riemann operator $D+f$ is a vector space of real rank $\deg(e)-\deg(e')+2r$. Given a family of stabilized (Floer-)Cauchy-Riemann operators $D+f$ over a base $M$, taking the kernel of each fiberwise operator yields a real rank $\deg(e)-\deg(e')+2r$ vector bundle
    \begin{equation}
    \ind D\equiv\ker(D+f)\to M;
    \end{equation}
we refer to this as the \emph{index bundle} of $D$. In the case that $D+f$ is a family of stabilized Cauchy-Riemann operators, $\ind D$ is a complex rank $r$ vector bundle. 

\begin{construction}
Let $D_1$ resp. $D_2$ be a (Floer-)Cauchy-Riemann operator on $\Sigma_1$ resp. $\Sigma_2$ such that the positive puncture datum of $D_1$ agrees with the negative puncture datum of $D_2$. Moreover, let $\tau\in\bbR_{>0}$. Recall, we may naturally glue all pieces of the data determining $D_1$ and $D_2$ to obtain a (Floer-)Cauchy-Riemann operator $D_1\#_\tau D_2$ on $\Sigma_1\#_\tau\Sigma_2$. Analogously, we may glue stabilizations $f_1\#_\tau f_2$ since the image of each stabilization lands in compactly supported sections; the glued map, 
    \begin{equation}
    f_1\#_\tau f_2:\bbC^{r_1+r_2}\to W^{\mu-1,2}(\Sigma_1\#_\tau\Sigma_2;E_1\#E_2),
    \end{equation}
is a stabilization for $\tau\gg0$, cf. \cite[Lemma 3.13]{PS24a}. Moreover, we may construct a real/complex-linear map 
    \begin{equation}
    \ker(D_1+f_1)\oplus\ker(D_2+f_2)\to\ker(D_1\#_\tau D_2+f_1\#_\tau f_2),\;\;\tau\gg0
    \end{equation}
using a standard gluing construction, cf. \cite[Section 3.4]{PS24a}; we refer to this construction as a \emph{linear gluing isomorphism}. 
\end{construction}

\begin{defin}
Let $\scrE\to\scrC\xrightarrow{\pi}\bbX$ be a complex vector bundle on a universal curve over a flow $n$-simplex. Assume we are given puncture datum for all objects: 
    \begin{equation}
    \frake\equiv\{e_x\}_{x\in\calP_j,0\leq j\leq n}.
    \end{equation}
A \emph{(smooth) family of Floer-Cauchy-Riemann operators} on $\scrE$ with puncture datum $\frake$, written $D$, consists of the following data.
\begin{enumerate}
\item For any two $x$ and $y$, a family of Floer-Cauchy-Riemann operators $D_{xy}$ resp. $D_{xy,k}$ on $\scrE_{xy}\to\scrC_{xy}$ resp. $\scrE_{xy,k}\to\scrC_{xy,k}$ with puncture datum $e_x$ resp. $e_y$ at the negative resp. positive tube-like end.
\item For any three $x$, $y$, and $z$, an equality 
    \begin{equation}\label{eqn:fcrbreaking}
    D_{xz}\#_{1/\tau} D_{zy}=D_{xy}
    \end{equation}
of families of Floer-Cauchy-Riemann operators on $\scrE_{xz}\#_{1/\tau}\scrE_{zy}\cong\scrE_{xy}$, such that the natural associativity diagram commutes.
\item For any two $x$ and $y$, an equality 
    \begin{equation}\label{eqn:fcrrestrict}
    D_{xy,k}=D_{xy}
    \end{equation}
of families of Floer-Cauchy-Riemann operators on $\scrE_{xy,k}\cong\scrE_{xy}$, such that the natural associativity diagram commutes.
\item The natural associativity diagram intertwining \eqref{eqn:fcrbreaking} and \eqref{eqn:fcrrestrict} commutes.
\end{enumerate}
In \eqref{eqn:fcrbreaking} resp. \eqref{eqn:fcrrestrict}, we implicitly sum the left hand side with standard operators (in the case of \eqref{eqn:fcrbreaking}, this sum occurs before gluing).
\end{defin}

\begin{defin}
Let $\scrE\to\scrC\xrightarrow{\pi}\bbX$ be a complex vector bundle on a universal curve over a flow $n$-simplex. Assume we are given a family of Floer-Cauchy-Riemann operator $D$ on $\scrE$. A \emph{(smooth) family of stabilizations (of rank $\vec{r}$, where $\vec{r}$ is stabilization data)} of $D$, written $f$, consists of the following data.
\begin{enumerate}
\item For any two $x$ and $y$, a family of stabilizations (of rank $r_{xy}$ resp. $r_{xy,k}$) $f_{xy}$ resp. $f_{xy,k}$ of $D_{xy}$ resp. $D_{xy,k}$.
\item For any three $x$, $y$, and $z$, an equality 
    \begin{equation}\label{eqn:stabilizationbreaking}
    f_{xz}\#_{1/\tau} f_{zy}=f_{xy}
    \end{equation}
of families of stabilizations of $D_{xz}\#_{1/\tau} D_{zy}=D_{xy}$, such that the natural associativity diagram commutes.\footnote{Since the gluing of families of stabilizations is not necessarily a family of stabilizations, we are implicitly requiring that the gluing actually is a family of stabilizations.}
\item For any two $x$ and $y$, an equality 
    \begin{equation}\label{eqn:stabilizationrestrict}
    f_{xy,k}=f_{xy}
    \end{equation}
of families of stabilizations of $D_{xy,k}=D_{xy}$, such that the natural associativity diagram commutes.
\item The natural associativity diagram intertwining \eqref{eqn:stabilizationbreaking} and \eqref{eqn:stabilizationrestrict} commutes.
\end{enumerate}
In \eqref{eqn:stabilizationbreaking} resp. \eqref{eqn:stabilizationrestrict}, we implicitly sum the left hand side with the zero stabilization (in the case of \eqref{eqn:stabilizationbreaking}, this sum occurs before gluing).
\end{defin}

\begin{rem}
As before, we refer to a family of Floer-Cauchy-Riemann operators such that $\frake$ and $\{Y_{xy}\}$ are all simply trivial as a family of Cauchy-Riemann operators. Recall, this implies all the corresponding differential operators are complex-linear and all stabilizations are complex-linear.
\end{rem}

\begin{lem}\label{lem:stabilization}
Let $D$ be a family of (Floer-)Cauchy-Riemann operators on a complex vector bundle on a universal curve over a flow $n$-simplex. Suppose $\vec{r}$ is sufficiently large, then there exists a family of stabilizations (of rank $\vec{r}$) $f$ of $D$.
\end{lem}

\begin{proof}
For each $D_{xy}$ resp. $D_{xy,k}$, it is straightforward that we have the existence of a family of stabilizations; the problem is associativity. However, by choosing stabilizations inductively over the dimension of $\bbX_{xy}$, we may achieve the desired associativity needed to construct a family of stabilizations $f$ of $D$.
\end{proof}

\begin{lem}\label{lem:indexbundle}
If $D$ is a family of (Floer-)Cauchy-Riemann operators on a complex vector bundle on a universal curve over a flow $n$-simplex and $f$ is a family of stabilizations of $D$, then there exists a virtual bundle $\ind(D,f)$ which is well-defined up to contractible choice. If $D$ is actually a family of Cauchy-Riemann operators, then $\ind(D,f)$ is a complex virtual bundle.
\end{lem}

\begin{proof}
Over each $\bbX_{xy}^\circ$ resp. $\big(\partial^k\bbX_{xy}\big)^\circ$, we define 
    \begin{equation}
    \ind^\circ_{xy}D\equiv\ker(D_{xy}+f_{xy})\;\;\textrm{resp.}\;\;\ind^\circ_{xy,k}D\equiv\ker(D_{xy,k}+f_{xy,k});
    \end{equation}
we must glue these together over the collar neighborhoods to form a vector bundle over the entirety of $\bbX_{xy}$. Over a collar,
    \begin{equation}
    \bbX_{xz}^\circ\times\bbX_{zy}^\circ\times[0,\epsilon)_\tau\hookrightarrow\bbX_{xy}^\circ\;\;\textrm{resp.}\;\;\big(\partial^k\bbX_{xy}\big)^\circ\times[0,\epsilon)_\tau\hookrightarrow\bbX_{xy},
    \end{equation}
we define 
    \begin{equation}
    \ind_{xy}D\equiv\ind^\circ_{xz}D\oplus\ind^\circ_{zy}D\oplus\bbC^{r_{xy}-r_{xz}-r_{zy}}\;\;\textrm{resp.}\;\;\ind_{xy}D\equiv\ind^\circ_{xy,k}D_{xy,k};
    \end{equation}
moreover, we identify $\ind_{xy}D$ to $\ind^\circ_{xy}D$ over a collar neighborhood of the first form via a linear gluing isomorphism. Now, this is not strictly associative on overlaps of different collar neighborhoods; however, the failure of this is fortunately quite small. By performing small homotopies of the linear gluing isomorphisms inductively over the dimension of $\bbX_{xy}$, we can make this construction strictly associative, cf. \cite[Lemma 3.15]{PS24a}.
\end{proof}

\subsection{Concordances and (Floer-)Cauchy-Riemann operators}
\begin{defin}
Let $\scrE\to\scrC\xrightarrow{\pi}\bbX$ be a complex vector bundle on a universal curve over a flow $n$-simplex.
\begin{itemize}
\item Let $D$ and $D'$ be families of (Floer-)Cauchy-Riemann operators on $\scrE$. A \emph{concordance} between $D$ and $D'$ is a 1-parameter family of (Floer-)Cauchy-Riemann operators interpolating between $D$ and $D'$; here, the puncture data is also allowed to vary.\footnote{As mentioned before, the degree of a puncture datum remains constant in a family.} 
\item Let $D$ be a family of (Floer-)Cauchy-Riemann operators on $\scrE$ together with two families of stabilizations $f$ and $f'$ (of rank $\vec{r}$ and $\vec{r}'$, respectively). A \emph{concordance} between $f$ and $f'$ is a 1-parameter family of stabilizations (through stabilizations of constant rank $\vec{r}''\geq\vec{r},\vec{r}'$) interpolating between $f\oplus0$ and $f'\oplus0$.

\item Finally, let $D+f$ and $D'+f'$ be families of stabilized (Floer-)Cauchy-Riemann operators on $\scrE$. A \emph{concordance} between $D+f$ and $D'+f'$ is: a concordance between $D$ and $D'$; and a concordance between $f$ and $f'$. 
\end{itemize}
\end{defin}

In this subsection, we will show, using concordances, that the index bundle of a family of (Floer-)Cauchy-Riemann operators is (1) independent of the stabilization data and (2) always stably complex. We fix a complex vector bundle on a universal curve over a flow $n$-simplex: $\scrE\to\scrC\xrightarrow{\pi}\bbX$.

\begin{lem}
Let $D$ and $D'$ be families of (Floer-)Cauchy-Riemann operators on $\scrE$ with puncture data of the same rank and degree. There exists a concordance between $D$ and $D'$.
\end{lem}

\begin{proof}
Recall, any two puncture datum of the same rank and degree can be connected through a 1-parameter family. Thus, Lemma \ref{lem:contractiblefcr} implies that there exists a concordance between each $D_{xy}$ and $D_{xy}'$ separately; by choosing these concordances inductively over the dimension of $\bbX_{xy}$, we may achieve the desired associativity for a concordance between $D$ and $D'$. 
\end{proof}

The following two lemmata follow completely analogously to Lemma \ref{lem:stabilization}.

\begin{lem}
Let $f$ and $f'$ be stabilizations (of ranks $\vec{r}$ and $\vec{r}'$, respectively) of a family of (Floer-)Cauchy-Riemann operators $D$ on $\scrE$. There exists a concordance between $f$ and $f'$.
\end{lem}

\begin{lem}
Let $D$ and $D'$ be two families of (Floer-)Cauchy-Riemann on $\scrE$ and let $f$ and $f'$ be stabilizations (of ranks $\vec{r}$ and $\vec{r}'$, respectively) of $D$ and $D'$, respectively. Suppose there is a concordance between $D$ and $D'$, then there exists a stabilization of the aforementioned concordance which is a concordance between $f$ and $f'$.
\end{lem}

The following lemma follows completely analogously to Lemma \ref{lem:indexbundle}.

\begin{lem}
Let $D+f$ and $D'+f'$ be concordant families of stabilized (Floer-)Cauchy-Riemann operators on $\scrE$. There exists a concordance between $\ind(D,f)$ and $\ind(D',f')$.
\end{lem}

\begin{cor}
The index bundle of a family of (Floer-)Cauchy-Riemann operators on $\scrE$ only depends on $\scrE$ and the degrees of the puncture datums. In particular, if all degrees of all puncture datums vanish, then the index bundle is stably complex.
\end{cor}

Finally, we see that the condition on the degree of the puncture datums in the previous corollary is not necessary. 

\begin{prop}\label{prop:stablecomplex}
The index bundle of a family of (Floer-)Cauchy-Riemann operators on $\scrE$ is stably complex.
\end{prop}

\begin{proof}
Let $D$ be the given family of (Floer-)Cauchy-Riemann operators on $\scrE$ with puncture datums $\{e_x\}_{x\in\calP_j,0\leq j\leq n}$. Let $\vec\nu$ be the embedding data of $\scrE$ and $\overline{D}$ a family of (Floer-)Cauchy-Riemann operators on the trivial bundle $\underline{\bbC}^{\vec{\nu}}\to\scrC$ with puncture datums $\{\overline{e}_x\}_{x\in\calP_j,0\leq j\leq n}$. Since $\underline{\bbC}^{\vec{\nu}}$ is trivial, $\ind\overline{D}$ is stably trivial; in fact, it is straightforward to see that
    \begin{equation}
    \underline{\bbR}^{\delta_{xy}+\deg(e_y)}\oplus\ind\overline{D}_{xy}\cong\underline{\bbR}^{\delta_{xy}+\deg(e_x)}
    \end{equation}
for an appropriately chosen collection of sufficiently large integers $\{\delta_{xy}\}$.\footnote{In particular, $\delta_{xy}$ is the rank of the stabilization used to define $\ind\overline{D}_{xy}$; since we may induct on the dimension of $\bbX_{xy}$ to define stabilizations, we have that $\{\delta_{xy}\}$ satisfies the natural associativity requirement.} Now, $D\oplus\overline{D}$ is a family of (Floer-)Cauchy-Riemann operators on $\scrE$ with puncture datums of degree 0; hence, by the previous corollary, $\ind(D\oplus\overline{D})$ is stably complex. But, 
    \begin{equation}
    \ind (D\oplus\overline{D})\cong\ind D\oplus\ind\overline{D};
    \end{equation}
the proposition follows.
\end{proof}

\subsection{Grothendieck-Riemann-Roch flow categories}\label{sec: GRR flow cats}
Recall, a part of the data of an abstract curve $\Sigma$ is the choice of a biholomorphism which is well-defined up to translation in the $\bR$-direction:
    \begin{equation}
    \Sigma\cong\bbR\times S^1.
    \end{equation}
We denote by 
    \begin{equation}
    \ell_\Sigma\subset\Sigma
    \end{equation}
the line which is the image of $\bbR\times\{0\}$ under the aforementioned biholomorphism. Now, let $\scrC\to M$ be a family of abstract curves over a smooth manifold; this is a fiber bundle whose fiber is an abstract curve $\Sigma$. We denote by 
    \begin{equation}
    \ell_\scrC\subset\scrC
    \end{equation}
the subbundle whose fiber is the line $\ell\subset\Sigma$.

\begin{defin}
Let $\pi:\scrC\to\bbX$ be a universal curve over a flow $n$-simplex. The \emph{universal line} of $\scrC$, written $\pi:\ell_\scrC\to\bbX$, consists of the following data.
\begin{enumerate}
\item For any two $x$ and $y$, the fiber bundle $\pi_{xy}:\ell_{\scrC_{xy}}\to\bbX^\circ_{xy}$ resp. $\pi_{xy,k}:\ell_{\scrC_{xy,k}}\to\big(\partial^k\bbX_{xy}\big)^\circ$.
\item For any three $x$, $y$, and $z$, the isomorphism
    \begin{equation}\label{eqn:line1}
    \ell_{\scrC_{xz}}\#_{1/\tau}\ell_{\scrC_{zy}}\cong\varphi_{xzy}^*\ell_{\scrC_{xy}}.
    \end{equation}
By construction, the natural associativity condition already commutes.
\item For any two $x$ and $y$, the isomorphism
    \begin{equation}\label{eqn:line2}
    \ell_{\scrC_{xy,k}}\cong\varphi_{xy,k}^*\ell_{\scrC_{xy}}.
    \end{equation}
By construction, the natural associativity condition already commutes.
\item By construction, the natural associativity condition intertwining \eqref{eqn:line1} and \eqref{eqn:line2} already commutes.
\end{enumerate}
\end{defin}

\begin{rem}
We continue to work in the setup of the previous definition. Observe, there is straightforward notion of a map $\sigma:\scrC\to\ell_\scrC$ given by defining 
    \begin{equation}
    \sigma_{xy}:\scrC_{xy}\to\ell_{\scrC_{xy}}\;\;\textrm{resp.}\;\;\sigma_{xy,k}:\scrC_{xy,k}\to\ell_{\scrC_{xy,k}}
    \end{equation}
to be the map which exhibits the domain as an $S^1$-bundle over the codomain; these maps, essentially by definition, will satisfy the natural associativity diagrams.
\end{rem}

\begin{defin}\label{defin:grrvb}
Let $\scrE\to\scrC\xrightarrow{\pi}\bbX$ be a complex vector bundle on a universal curve over a flow $n$-simplex. We say $\scrE$ is a \emph{Grothendieck-Riemann-Roch (GRR) complex vector bundle} if a complex trivialization of $\scrE\vert_{\ell_\scrC}$ exists and we have fixed such a complex trivialization which is compatible with the already required complex trivialization over the tube-like ends.
\end{defin}

\begin{defin}
A \emph{GRR flow $n$-simplex} is a flow $n$-simplex $\bbX\equiv\bbX_{0\cdots n}$ together with the following data.
\begin{enumerate}
\item A universal curve $\scrC\to\bbX$.
\item A GRR complex vector bundle $\scrE\to\bbX$.
\item A family of (Floer-)Cauchy-Riemann operators $D$ on $\scrE$.
\item For any two $x$ and $y$, an equivalence of virtual bundles
    \begin{equation}\label{eqn:grrbreak}
    T\bbX_{xy}+\underline{\bbR}\cong\ind D_{xy}+\underline{\bbR}^{\lvert\{j+1,\ldots,j'\}\rvert}
    \end{equation}
such that the natural associativity diagram commutes.
\item For any two $x$ and $y$, an equivalence of virtual bundles 
    \begin{equation}\label{eqn:grrrestrict}
    T\partial^k\bbX_{xy}+\underline{\bbR}\cong\ind D_{xy,k}+\underline{\bbR}^{\lvert\{j+1,\ldots,\widehat{k},\ldots,j'\}\rvert}
    \end{equation}
such that the natural associativity diagram commutes. 
\item We require that the natural associativity diagram intertwining \eqref{eqn:grrbreak} and \eqref{eqn:grrrestrict} commutes.
\end{enumerate}
\end{defin}

The proof of \cite[Theorem 1.6]{AB24} implies that there is a stable $\infty$-category $\flow^{\rm GRR}$ whose $n$-simplices consist of GRR flow $n$-simplices. 

The following result produces a ``flow-categorification'' of the index bundle.

\begin{prop}\label{prop:grrtocomplex}
    There is a functor
    \begin{equation}\label{eqn:indfunctor}
    \ind(\bbX):\flow^{\rm GRR}\to\flow^{\rm decU}
    \end{equation}
which is an equivalence of stable $\infty$-categories.
\end{prop}

\begin{proof}
First, we will construct \eqref{eqn:indfunctor}. Let $\bbX\equiv\bbX_{0\cdots n}$ be a GRR flow $n$-simplex; the goal is to build a decorated complex-oriented structure on $\bbX$, and we proceed as follows. Recall, the proof of Proposition \ref{prop:stablecomplex} showed $\ind D\oplus\overline{D}$ is stably complex, i.e., there exists a trivial vector bundle $\underline{\bbC}^{\vec{r}}\to\bbX$, where $\vec{r}$ is stabilization data on $\bbX$, such that $\ind D\oplus\ind\overline{D}\oplus\underline{\bbC}^{\vec{r}}$ is complex. We define the vector bundle $I\to\bbX$ via
    \begin{equation}
    I_{xy}\equiv I(x,y)\equiv\ind D_{xy}\oplus\ind \overline{D}_{xy}\oplus\underline{\bbC}^{r_{xy}}\to\bbX_{xy}.
    \end{equation}
Moreover, the proof of Proposition \ref{prop:stablecomplex} showed we have an isomorphism
    \begin{equation}
    \underline{\bbR}^{\delta_{xy}+\deg(e_y)}\oplus\ind\overline{D}_{xy}\cong\underline{\bbR}^{\delta_{xy}+\deg(e_x)}
    \end{equation}
for an appropriately chosen collection of sufficiently large integers $\{\delta_{xy}\}$.\footnote{The scrupulous reader may now reasonably object that the choice of $\{\delta_{xy}\}$ may not be compatible with face maps; we may immediately remedy this as follows. First, define $\ind(\bbX)$ on all 0-simplices. Second, define $\ind(\bbX)$ on all 1-simplices making sure to choose $\{\delta_{xy}\}_{x\in\calP_0,y\in\calP_1}$ in a way that is compatible with the already established choice of $\{\delta_{xy}\}_{x,y\in\calP_0}$ and $\{\delta_{xy}\}_{x,y\in\calP_1}$; we may do this because, as already mentioned, $\delta_{xy}$ is the rank of the stabilization used to define $\ind\overline{D}_{xy}$ and we may induct on the dimension of $\bbX_{xy}$, $x\in\calP_0$ and $y\in\calP_1$, to define stabilizations. We proceed in the same fashion for higher-dimensional simplices.} We can rewrite this as an equivalence of virtual bundles: 
    \begin{equation}
    \Big(\ind\overline{D}_{xy}\oplus W(x,y),\underline{\bbR}^{\deg(e_x)}\Big)\cong\Big(W(x,y),\underline{\bbR}^{\deg(e_y)}\Big),\;\;W(x,y)\equiv\underline{\bbR}^{\delta_{xy}}.
    \end{equation}
Finally, we define the oriented virtual vector space
    \begin{equation}
    V_x\equiv\Big(0,\bbR^{\deg(e_x)}\Big)
    \end{equation}
with the standard orientation. In particular, we have the following equivalence of virtual bundles:
    \begin{equation}
    T\bbX_{xy}+\underline{V}_y+\underline{\bbR}+W(x,y)\cong I(x,y)+W(x,y)+\underline{\bbR}^{\lvert\{j+1,\ldots,\widehat{k},\ldots,j'\}\rvert}+\underline{V}_x.
    \end{equation}
Moreover, it is straightforward to see the aforementioned equivalence satisfies all of the requirements to make $\bbX$ into a decorated complex-oriented flow category, and this construction is compatible with face maps. Thus, we have constructed $\ind(\bbX)$ as a map of semi-simplicial sets.

Second, we will show $\ind(\bbX)$ is an equivalence of semi-simplicial sets by showing the fiber over any decorated complex-oriented flow $n$-simplex is non-empty and contractible.

\begin{claim}
The space of universal curves over a decorated complex-oriented flow $n$-simplex is non-empty and contractible. 
\end{claim}
\begin{proof}[Proof of claim]
This follows because, in Definition \ref{defin:abstractcurve}, we required that $\Sigma$ admits (1) tube-like ends and (2) a global biholomorphism to $\bbR\times S^1$ such that the tube-like ends are given by $\bbR$-translation. In particular, a universal curve over any base is a fiber bundle with fiber equivalent to $\bbR\times S^1$ and structure group biholomorphisms of $\bbR\times S^1$ which fix the standard tube-like ends up to $\bbR$-translation, i.e., the structure group is contractible.
\end{proof}

\begin{claim}
Let $\pi:\scrC\to\bbX$ be a universal curve over a decorated complex-oriented flow $n$-simplex.
The space of GRR complex vector bundles $\scrE\to\scrC$ which satisfy the condition succeeding this sentence is non-empty and contractible. Suppose $D$ is a family of (Floer-)Cauchy-Riemann operators on $\scrE$ and $\vec{r}$ is stabilization data on $\bbX$, then, for any two $x$ and $y$, we have an isomorphism
    \begin{equation}\label{eqn:grrclaim}
    I(x,y)\cong\ind D_{xy}\oplus\ind\overline{D}_{xy}\oplus\underline{\bbC}^{r_{xy}}
    \end{equation}
such that the natural associativity diagram associated to any codimension 1 boundary stratum of $\bbX_{xy}$ commutes.
\end{claim}

\begin{proof}[Proof of claim]
Observe, \eqref{eqn:grrclaim} is equivalent to asking for an isomorphism 
    \begin{equation}
    I(x,y)\cong\ind (D_{xy}\oplus\overline{D}_{xy})\oplus\underline{\bbC}^{r_{xy}}
    \end{equation}
satisfying the appropriate natural associativity diagrams. In particular, it suffices to assume $D$ is a Cauchy-Riemann operator. Therefore, we are asking for $I(x,y)$ to be given by the index bundle of a family of Cauchy-Riemann operators on $\scrC_{xy}$, at least stably, and that such a choice lives in a contractible space of choices; this simply amounts to Bott periodicity and the fact that the homotopy fiber of the Bott map $\Omega^2{\rm BU}\xrightarrow{\sim}{\rm BU}$ is contractible (explicitly, using a model for Bott periodicity based on Cauchy-Riemann operators, analogous to that of \cite{Atiyah}). Meanwhile, the condition that the natural associativity diagrams commute amounts to asking if we can choose lifts compatibly with breaking and restricting; but this follows by inducting over the dimension of $\bbX_{xy}$.
\end{proof}

Finally, we may upgrade $\ind(\bbX)$ to an equivalence of stable $\infty$-categories as follows. \cite[Theorem 1.4]{Tan18} shows we may upgrade $\ind(\bbX)$ to a functor of stable $\infty$-categories (the condition in \emph{loc. cit.}, which requires that degenerate 1-simplices are sent to equivalences, is easily verified since the diagonal decorated complex-oriented flow bimodule lifts to the diagonal GRR flow bimodule, and both represent the identity in their respective category, cf. \cite[Subsection 6.2]{AB24}). Therefore, $\ind(\bbX)$ is a functor between stable $\infty$-categories which is an equivalence of semi-simplicial sets, hence an equivalence of stable $\infty$-categories.
\end{proof}

\subsection{Pre-Grothendieck-Riemann-Roch flow categories}
Unfortunately, spectral symplectic cohomology does not immediately yield a GRR flow category since the natural complex vector bundle over the natural universal curve does not come with a fixed trivialization over the associated universal line; this leads us to the following definition.
\begin{defin}
A \emph{pre-GRR flow $n$-simplex} is a flow $n$-simplex $\bbX\equiv\bbX_{0\cdots n}$ together with the following data.
\begin{enumerate}
\item A universal curve $\scrC\to\bbX$.
\item A complex vector bundle $\scrE\to\bbX$.
\item A family of (Floer-)Cauchy-Riemann operators $D$ on $\scrE$.
\item For any two $x$ and $y$, an equivalence of virtual bundles
    \begin{equation}\label{eqn:pregrrbreak}
    T\bbX_{xy}+\underline{\bbR}\cong\ind D_{xy}+\underline{\bbR}^{\lvert\{j+1,\ldots,j'\}\rvert}
    \end{equation}
such that the natural associativity diagram commutes.
\item For any two $x$ and $y$, an equivalence of virtual bundles 
    \begin{equation}\label{eqn:pregrrrestrict}
    T\partial^k\bbX_{xy}+\underline{\bbR}\cong\ind D_{xy,k}+\underline{\bbR}^{\lvert\{j+1,\ldots,\widehat{k},\ldots,j'\}\rvert}
    \end{equation}
such that the natural associativity diagram commutes. 
\item We require that the natural associativity diagram intertwining \eqref{eqn:pregrrbreak} and \eqref{eqn:pregrrrestrict} commutes.
\end{enumerate}
\end{defin}

The proof of \cite[Theorem 1.6]{AB24} implies that there is a stable $\infty$-category $\flow^{\rm preGRR}$ whose $n$-simplices consist of pre-GRR flow $n$-simplices. 

\begin{construction}
We will now produce a functor 
    \begin{equation}
    \calG:\flow^{\rm preGRR}\to\flow^{\rm GRR}.
    \end{equation}
First, let $\scrE\to\scrC\xrightarrow{\pi}M$ be a complex vector bundle over a family of abstract curves over a smooth manifold. Consider the complex vector bundle $\scrE\vert_{\ell_\scrC}\to\ell_\scrC$. We may choose an embedding 
    \begin{equation}
    \scrE\vert_{\ell_\scrC}\hookrightarrow\bbC^N,\;\;N\gg0
    \end{equation}
to find a complex vector bundle $\scrF\to\ell_\scrC$ which satisfies 
    \begin{equation}
    \scrE\vert_{\ell_\scrC}\oplus\scrF\cong\underline{\bbC}^K,\;\;K\gg0.
    \end{equation}
Second, let $M$ have non-empty boundary $\partial M$ and consider $\partial\scrC\equiv\scrC\vert_{\partial M}\to\partial M$. Suppose we already have a choice of embedding 
    \begin{equation}
    \scrE\vert_{\ell_{\partial\scrC}}\hookrightarrow\underline{\bbC}^N,\;\;N\gg0
    \end{equation}
in order to construct $\scrF\to\partial\scrC$ satisfying 
    \begin{equation}
    \scrE\vert_{\ell_{\partial\scrC}}\oplus\scrF\cong\underline{\bbC}^K,\;K\gg0,
    \end{equation}
then, by stabilizing and extending the embedding to an embedding 
    \begin{equation}
    \scrE\vert_{\ell_\scrC}\hookrightarrow\underline{\bbC}^{N+K'},\;\;N,K'\gg0,
    \end{equation}
we may find a complex vector bundle $\scrF'\to\ell_\scrC$ which satisfies 
    \begin{equation}
    \scrF\oplus\underline{\bbC}^{K'}\cong\scrF'\vert_{\ell_{\partial\scrC}}\;\;\textrm{and}\;\;\scrE\vert_{\ell_{\partial\scrC}}\oplus\scrF'\cong\underline{\bbC}^{K+K'},\;K,K'\gg0.
    \end{equation}
Finally, we may build $\calG$ as follows. For any pre-GRR flow $n$-simplex $\bbX\equiv\bbX_{0\cdots n}$, we may build a complex vector bundle $\scrF\to\ell_\scrC$ such that $\scrE\vert_{\ell_\scrC}\oplus\scrF\to\ell_\scrC$ is trivial by inducting over both the dimension of $\bbX_{xy}$ and $n$. In particular, we define $\calG(\bbX)$ to be the GRR flow simplex with all the same data as $\bbX$, except we replace $\scrE\to\scrC$ by $\scrE\oplus\sigma^*\scrF\to\scrC$, where $\sigma:\scrC\to\ell_\scrC$; by construction, this is compatible with face maps and hence defines $\calG$ as a map of simplicial sets. We may upgrade $\calG$ to a functor between stable $\infty$-categories by appealing to \cite[Theorem 1.4]{Tan18} (the condition in \emph{loc. cit.}, which requires that degenerate 1-simplices are sent to equivalences, is easily verified since the diagonal pre-GRR flow bimodule is sent to the diagonal GRR flow bimodule, and both represent the identity in their respective category).
\end{construction}

We will return to properties of this functor in the sequel.

\section{Homotopy coherent Grothendieck-Riemann-Roch}
\subsection{Algebraic preliminaries}\label{sec:alg}

    We briefly recap some parts of the theory of DG-categories, cf. \cite{keller,Tabuada:Une} for more detailed background. We assume all our DG-categories are small and contained in some fixed universe, cf. \cite[Section 2]{Toen} for further discussion on this point. By convention, all of our chain complexes are assumed to be $\bC$-linear and (homologically) $\bZ$-graded.
    \begin{defin}
        A \emph{DG-category} $\cC$ is a small category enriched in chain complexes (over $\bC$). Explicitly, this means each morphism set $\cC(x,y)$ is a chain complex and each composition $\cC(y,z) \otimes \cC(x,y) \to \cC(x,z)$ is a chain map.
        
        \begin{itemize} 
            \item A \emph{DG-functor} $f: \cC \to \cD$ is a functor of enriched categories, i.e., each map $f(x,y):\cC(x,y) \to \cD\big(f(x),f(y)\big)$ is a chain map. DG-natural transformations are defined similarly.
            \item A \emph{left module} over $\cC$ is a DG-functor from $\cC$ to chain complexes, and a \emph{right module} is a left module over $\cC^{\rm op}$. Equivalently, a left module $\cL$ over $\cC$ consists of chain complexes $\cL(x)$, for all $x \in \cC$, along with chain maps $\cC(x,y) \otimes \cL(x) \to \cL(y)$ satisfying a suitable associativity condition. Given $x \in \cC$, the \emph{(left) Yoneda module} of $x$ is given by $\cC(-, x)$.

            \item The \emph{homotopy category} $\Ho(\cC)$ of $\cC$ is the category with the same objects as $\cC$ and morphisms $\Ho(\cC)(x,y) \equiv H_0\big(\cC(x,y)\big)$; this is naturally enriched in $\bC$-vector spaces. 

            \item A DG-functor is an \emph{equivalence} if it induces (1) an equivalence on homotopy categories and (2) quasi-isomorphisms on all mapping complexes.
            \item $\cC$ is \emph{pretriangulated} if the image of the Yoneda embedding, from $\cC$ to its category of modules, is closed (up to quasi-isomorphism) under taking shifts and cones.
        \end{itemize}
    \end{defin}
    The homotopy category of a pretriangulated DG-category is naturally triangulated since its image under the (fully faithful) Yoneda embedding is. The underlying category $\Ho(\cC)$ only recovers $0$-th homology directly, but the shift functor from the triangulated structure allows us to recover all homology groups of the mapping complexes.
    \begin{example}
        Let $R$ be a DG-algebra. The category of $R$-modules, denoted $\Mod_R$, is naturally a DG-category; for $A, B \in \Mod_R$, $\Mod_R(A, B)_i$ consists of the $R$-linear maps $\Sigma^i A \to B$ equipped with the usual differential on mapping complexes:
            \begin{equation}
            df \equiv d_B \circ f - (-1)^i f \circ d_A,\;\;f\in\Mod_R(A,B)_i.
            \end{equation}
        Observe, $\Mod_R$ is pretriangulated. More generally, the module category of any DG-category $\cC$, denoted $\Mod_\cC$, is itself a DG-category.
    \end{example}
    \begin{example}
    Let $R$ be a DG-algebra. The full subcategory of \emph{perfect complexes}, denoted $\perf_R$, is the smallest subcategory of $\Mod_R$ containing the left module $R$ which is closed under shifts, cones, and quasi-isomorphisms. We call the object $R\in\perf_R$ the \emph{unit} module; we denote it by $\unit_R$. Note, its DG-algebra of endomorphisms is exactly $R$.
    \end{example}
    \begin{rem}
    An important special case is when $R$ is a graded ring, e.g. $R={\bC[\underline b]}$. In particular, may view $R$ as a DG-algebra with the trivial differential.
    \end{rem}
    Essential to relating DG-categories and $\infty$-categories is the following construction from \cite[Construction 1.3.1.6]{HA} (although, we follow the indexing conventions of \cite[Section 2.5]{HLS}).
    \begin{construction}
        Let $\cC$ be a DG-category. Its \emph{DG-nerve} $N^{\rm dg}(\cC)$ is the simiplicial set $N^{\rm dg}(\cC)_\bullet$ defined as follows. Objects (i.e., $0$-simplices) are defined to be the objects of $\cC$. For $n \geq 0$, an $n$-simplex $\sigma$ consists of a sequence of objects $x_0, \ldots, x_n$ along with the data of, for each subset $I=\{j_0, \ldots, j_\mu\}\subset\{0,\ldots,n\}$, an element $\sigma_I \in \cC(x_{j_0},x_{j_\mu})_{\mu-1}$ such that the following identity holds:
        \begin{equation}
            d\sigma_I = \sum\limits_{0<l<\mu} (-1)^{\mu-l} \left(\sigma_{I-\{j_l\}} - \sigma_{I_{\geq j_l}} \circ \sigma_{I_{\leq j_l}}\right),
        \end{equation}
        where $I_{\leq j_l} = \{j_0, \ldots, j_l\}$ and $I_{\geq j_l} = \{j_l, \ldots, j_\mu\}$.
    
    \end{construction}
    \begin{rem}
    $N^{\rm dg}(\cC)$ is always an $\infty$-category; moreover, $N^{\rm dg}(-)$ can be made to be functorial in $\cC$. 
    \end{rem}
    \begin{defin}\label{defin:generatedg}
        Let $\cC$ be a DG-category and $A$  a collection of objects in $\cC$. The subcategory \emph{generated by $A$} is the smallest full subcategory $\cA\subset\cC$ containing the objects in $A$ such that the image of $\cA$ under the Yoneda embedding of $\cC$ is closed under shifts, cones, and quasi-isomorphisms.
    \end{defin}

    \begin{rem}
    If a full subcategory $\cA$ generates $\cC$, then the restriction along the inclusion $\cA \to \cC$ defines an equivalence $\Mod_\cC \to \Mod_\cA$. 
    \end{rem}

    \begin{example}
    By construction, $\perf_R$ is generated by $\unit_R$.
    \end{example}

    \begin{warning}
    Note, generation does not allow infinite direct sums; because of this, $\Mod_R$ is \emph{not} generated by $\unit_R$.
    \end{warning}

    Observe, the homotopy category (in the $\infty$-categorical sense) of $N^{\rm dg}(\cC)$ agrees with the homotopy category (in the DG-categorical sense) of $\cC$, cf. \cite[Remark 1.3.1.11]{HA}. From this, we see that if $\cA$ generates $\cC$ (in the DG-categorical sense), then the objects of $N^{\rm dg}(\cA)$ generate $N^{\rm dg}(\cC)$ under finite colimits. 

    \begin{rem}
    Although $\unit_R$ does not generate $\Mod_R$ in the sense of Definition \ref{defin:generatedg}, $\unit_R$ does generate $N^{\rm dg}(\Mod_R)$ under all (not-necessarily-finite) colimits.
    \end{rem}

\subsection{Twisted Morse complexes}
\subsubsection{For flow categories}
Let $\frakm=\exp\fraka$ be a multiplicatively coherent cocycle on a flow category $\bbX\equiv\bbX_0$. Observe, since $\bC[\underline b]$ is evenly graded and $\fraka$ is degree 0 by assumption, it follows that the differential form grading of each $m_{xy}$ is even:
    \begin{equation}
    m_{xy}\in\bigoplus_p\Omega^{2p}(\bbX_{xy};\bC[\underline b]_{-2p}).
    \end{equation}
We will assume that $\bbX$ is decorated oriented -- this implies integration of differential forms over each $\bbX_{xy}$ makes sense. Note, by the evenness of the differential form grading, $\int_{\bbX_{xy}}m_{xy}$ can only ever possibly be non-trivial when $\dim\bbX_{xy}$ is even. Finally, as our ultimate application will be to complex-oriented flow categories, we will furthermore assume each $\dim I(x,y)$ is even. 

We may define 
    \begin{equation}
    CM_*(\bbX,\frakm;\bC[\underline b])\equiv\bigoplus_{x\in\calP}\bC[\underline b]\langle x\rangle,
    \end{equation}
where we use the total grading given by summing the grading on $\bbX$ and minus the grading on $\bC[\underline b]$, together with
    \begin{equation}
    \partial^\frakm:CM_*(\bbX,\frakm;\bC[\underline b])\to CM_*(\bbX,\frakm;\bC[\underline b])
    \end{equation}
defined via the $\bC[\underline b]$-linear extension of 
    \begin{equation}\label{eqn:twistedcodiff}
    x\mapsto\sum_{y\in\calP}\Bigg(\int_{\bbX_{xy}}m_{xy}\Bigg)y.
    \end{equation}

\begin{lem}
$(\partial^\frakm)^2=0$.
\end{lem}

\begin{proof}
We see
    \begin{align}
    (\partial^\frakm)^2x&=\sum_{z,y\in\calP}\Bigg(\int_{\bbX_{xz}}m_{xz}\Bigg)\Bigg(\int_{\bbX_{zy}}m_{zy}\Bigg)y \\
    &=\sum_{z,y\in\calP}\Bigg(\int_{\bbX_{xz}\times\bbX_{zy}}\pi_{\bbX_{xy}}^*m_{xz}\wedge\pi_{\bbX_{zy}}^*m_{zy}\Bigg)y.
    \end{align}
Using Corollary \ref{cor:koszulsign1}, we have that, for a fixed $y$-component, the right hand side becomes
    \begin{equation}
    \sum_{z\in\calP}\Bigg(-\int_{\bbX_{xz}\times\bbX_{zy}}m_{xy}\vert_{\bbX_{xz}\times\bbX_{zy}}\Bigg)=-\int_{\partial\bbX_{xy}}m_{xy}\vert_{\partial\bbX_{xy}}=-\int_{\bbX_{xy}}dm_{xy}=0,
    \end{equation}
where the last equality holds since $m_{xy}$ is closed.
\end{proof}

\begin{defin}
We refer to the chain complex
    \begin{equation}
    \big(CM_*(\bbX,\frakm;\bC[\underline b]),\partial^\frakm\big)
    \end{equation}
as the \emph{$\frakm$-twisted chain complex} of $\bbX$.
\end{defin}

\subsubsection{For flow simplices}\label{subsubsection:twistedmorsesimplex}
Now, let $\frakm=\exp\fraka$ be a multiplicatively coherent cocycle on a flow $n$-simplex $\bbX\equiv\bbX_{0\cdots n}$. Again, the differential form grading on each $m_{xy}$ is even. Observe, $\frakm$ induces a multiplicatively coherent cocycle $\frakm_\sigma=\exp\fraka_\sigma$ on $\partial^\sigma\bbX$ for each $\partial^\sigma\Delta^n\subset\Delta^n$. We will assume $\bbX$ (and, by extension, each $\partial^\sigma\bbX$) is decorated oriented; hence, $\int_{\bbX_{xy}}m_{xy}$ is defined and can only ever be possibly non-trivial when $\dim\bbX_{xy}$ is even. Finally, as our ultimate application will be to complex-oriented flow simplices, we will furthermore assume each $\dim I(x,y)$ is even.

For $\mu\geq1$, consider
    \begin{equation}
    \partial^{\langle j_0\cdots j_\mu\rangle}\Delta^n\subset\Delta^n,\;\;0\leq j_0<\cdots< j_\rho<\cdots<j_\mu\leq n
    \end{equation}
together with
    \begin{equation}
    H^{\frakm_{\langle j_0\cdots j_\mu\rangle}}:CM_*(\bbX_{j_0},\frakm_{\langle j_0\rangle};\bC[\underline b])\to CM_*(\bbX_{j_\mu},\frakm_{\langle j_\mu\rangle};\bC[\underline b]),\;\;\bbX_{j_\rho}\equiv\partial^{\langle j_\rho\rangle}\bbX
    \end{equation}
defined via the $\bbC[\underline{b}]$-linear extension of 
    \begin{equation}
    x\mapsto\sum_{y\in\calP_{j_\mu}}\Bigg(\int_{\partial^{\langle j_0\cdots j_\mu\rangle}\bbX_{xy}}m_{xy}\Bigg)y.
    \end{equation}
For $\mu=0$, we take 
    \begin{equation}
    H^{\frakm_{j_0}}\equiv\partial^{\frakm_{j_0}}.
    \end{equation}

\begin{lem}\label{lem:dgnerveequation}
    \begin{multline}\label{eqn:dgnerveequation}
    \partial^{\frakm_{\langle j_\mu\rangle}}H^{\frakm_{\langle j_0\cdots j_\mu\rangle}}-(-1)^\mu H^{\frakm_{\langle j_0\cdots j_\mu\rangle}}\partial^{\frakm_{\langle j_0\rangle}}= \\
    \sum_{0<\rho<\mu}(-1)^{\mu-\rho}\big(H^{\frakm_{\langle j_0\cdots\widehat{j_\rho}\cdots j_\mu\rangle}}-H^{\frakm_{\langle j_\rho\cdots j_\mu\rangle}}\circ H^{\frakm_{\langle j_0\cdots j_\rho\rangle}}\big).
    \end{multline}
\end{lem}

\begin{proof}
Since $m_{xy}$ is closed, we have that 
    \begin{align}
    0&=\int_{\partial^{\langle j_0\cdots j_\mu\rangle}\bbX_{xy}}dm_{xy} \\
    &=\int_{\partial(\partial^{\langle j_0\cdots j_\mu\rangle}\bbX_{xy})}m_{xy}\vert_{\partial\bbX_{xy}} \\
    &=\sum_{\substack{z\in\calP_{j_\rho} \\ 0\leq\rho\leq\mu}}\int_{\partial^{\langle j_0\cdots j_\mu\rangle}\bbX_{xz}\times\partial^{\langle j_0\cdots j_\mu\rangle}\bbX_{zy}}m_{xy}\vert_{\partial^{\langle j_0\cdots j_\mu\rangle}\bbX_{xz}\times\partial^{\langle j_0\cdots j_\mu\rangle}\bbX_{zy}}+ \nonumber \\
    &\hspace{30ex}\sum_{0<\rho<\mu}\int_{\partial^{\langle j_0\cdots\widehat{j_\rho}\cdots j_\mu\rangle}\bbX_{xy}}m_{xy}\vert_{\partial^{\langle j_0\cdots\widehat{j_\rho}\cdots j_\mu\rangle}\bbX_{xy}}.
    \end{align}
Observe, by Corollaries \ref{cor:koszulsign1} and \ref{cor:koszulsign2}, the last line becomes 
    \begin{multline}
    \sum_{\substack{z\in\calP_{j_\rho} \\ 0\leq\rho\leq\mu}}(-1)^{\mu-\rho+1}\Bigg(\int_{\partial^{\langle j_0\cdots j_\mu\rangle}\bbX_{xz}}m_{xz}\Bigg)\Bigg(\int_{\partial^{\langle j_0\cdots j_\mu\rangle}\bbX_{zy}}m_{zy}\Bigg)+ \\
    \sum_{0<\rho<\mu}(-1)^{\mu-\rho}\int_{\partial^{\langle j_0\cdots\widehat{j_\rho}\cdots j_\mu\rangle}\bbX_{xy}}m_{xy}\vert_{\partial^{\langle j_0\cdots\widehat{j_\rho}\cdots j_\mu\rangle}\bbX_{xy}}
    \end{multline}
(in particular, we would like to emphasize that, in the second summand, we have swapped from the boundary orientation to the face orientation, hence the Koszul sign). But this is precisely a single $y$-component of \eqref{eqn:dgnerveequation}; the lemma follows.
\end{proof}

The following corollary is a straightforward consequence of the previous lemma.

\begin{cor}
The assignment 
    \begin{equation}
    (\partial^{\langle j_0\cdots j_\mu\rangle}\bbX,\frakm_{\langle j_0\cdots j_\mu\rangle})\rightsquigarrow H^{\frakm_{\langle j_0\cdots j_\mu\rangle}},
    \end{equation}
with the obviously induced face maps, yields an $n$-simplex in $N^{\rm dg}(\Mod_{\bC[\underline b]})$; moreover, this assignment commutes with the face maps.
\end{cor}

\subsection{Chern characters, revisited}\label{sec: chern revisited}
In this subsection, we will produce a ``flow-categorification'' of the Chern character, i.e., we will build the functor 
    \begin{equation}\label{eqn:cherncharacterfunctor}
    \ch(\bbX):\flow^{\rm decU}\to\dgnerve(\Mod_{\bC[\underline b]}).
    \end{equation}

\begin{construction}\label{construction:chernfunctor}
First, recall that every decorated complex-oriented flow $n$-simplex $\bbX\equiv\bbX_{0\cdots n}$ has a fixed complex vector bundle over it, denoted $I\to\bbX$, defined by 
    \begin{equation}
    I_{xy}\equiv I(x,y)\to\bbX_{xy}.
    \end{equation}
Now, we may fix a choice of Hermitian connection on any $I\to\bbX$ which is compatible with face maps by inducting over the dimension of $\bbX_{xy}$ and $n$. Consider the multiplicatively coherent cocycle $\exp\big(\frakc\frakh(I)\big)$ on $\bbX$; by performing the construction in Part \ref{subsubsection:twistedmorsesimplex}, we obtain an $n$-simplex in $\dgnerve(\Mod_{\bC[\underline b]})$, i.e., we have the assignment
    \begin{equation}
    \Big(\partial^{\langle j_0\cdots j_\mu\rangle}\bbX,\exp\big(\frakc\frakh(I)\big)_{\langle j_0\cdots j_\mu\rangle}\Big)\rightsquigarrow H^{\exp(\frakc\frakh(I))_{\langle j_0\cdots j_\mu\rangle}},\;\;0\leq j_0<\cdots<j_\mu\leq n.
    \end{equation}
Since our choice of Hermitian connections, and hence our polynomial-valued Chern characters, are compatible with face maps, we have constructed \eqref{eqn:cherncharacterfunctor} as a map of semi-simplicial sets. We may upgrade $\ch(\bbX)$ to a functor between stable $\infty$-categories by appealing to \cite[Theorem 1.4]{Tan18}. We must check the condition that degenerate 1-simplices are sent to equivalences; the argument is as follows. So, consider a decorated complex-oriented flow category $\bbX$ together with its diagonal decorated complex-oriented flow bimodule $s\bbX:\bbX\to\bbX$ which represents the identity; topologically, we have that 
    \begin{equation}
    s\bbX_{xy}\cong\begin{cases}
    \bbX_{xy}\times[0,1], & x\neq y \\
    *, & \textrm{otherwise};
    \end{cases}
    \end{equation}
however, as a stratified smooth manifold with corners, $s\bbX_{xy}$ is considerably more complicated. Fortunately, the topological type of $s\bbX_{xy}$ is enough to deduce that the map of twisted Morse complexes induced by $s\bbX_{xy}$ under $\ch(\bbX)$,
    \begin{equation}
    CM_*\Big(\bbX,\exp\big(\frakc\frakh(I)\big);{\bC[\underline b]}\Big)\to CM_*\Big(\bbX,\exp\big(\frakc\frakh(I)\big);{\bC[\underline b]}\Big),
    \end{equation}
is an upper triangular matrix.\footnote{This follows because it is a formal consequence of the definition of a flow category that $\bbX_{xy}$ and $\bbX_{yx}$ cannot both be non-empty, cf. \cite[Remark 1.4]{AB24}.} Therefore, the aforementioned map is a quasi-isomorphism of chain complexes, as desired.
\end{construction}

\begin{prop}\label{prop: rpoejgwpierngore}
The map induced by $\ch(\bbX)$ (on generators, after base change),
    \begin{equation}
        \pi_*\flow^{\rm decU}(\unit, \unit)\otimes_\bbZ \bC \to \pi_*N^{\rm dg}(\Mod_{\bC[\underline b]})(\unit_{\bC[\underline b]}, \unit_{\bC[\underline b]}) \cong H_*({\bC[\underline b]}) \cong {\bC[\underline b]},
    \end{equation}
    is an isomorphism. 
\end{prop}
\begin{proof}
Using the isomorphism $\pi_d \flow^{\rm decU}(\unit, \unit) \cong \Omega^U_d$, we see that $\ch(\bbX)$, after base change to $\bbC$, sends the class of a stably complex closed smooth manifold $Y$ to
    \begin{equation}
    \sum_{q_1,q_2,\ldots\geq 0} \frac{1}{q_1!q_2!\cdots}\Bigg(\int_Y\ch_1(Y)^{q_1}\ch_2(Y)^{q_2}\cdots\Bigg)(b^{q_1}_1b^{q_2}_2\cdots)\in \bC[\underline b].
    \end{equation}
It is standard that this is an isomorphism; for instance, it follows from the Thom isomorphism combined with \cite[Theorem 21.4.3]{MP12}.
\end{proof}
\begin{cor}\label{cor:firpg-iegn}
    For any $\bX \in \flow^{\rm decU}$, we have the following isomorphism:
    \begin{equation}
        \pi_*\bX \otimes_\bZ \bC \cong H_*\big(\ch(\bX)\big)=\Big(\bbX,\exp\big(\frakc\frakh(I)\big);{\bC[\underline b]}\Big),
    \end{equation}
where the second equality is by definition.
\end{cor}
\subsection{Bulk-deformations, a first pass}\label{subsec:derhamfirstpass}
In this subsection, we will produce a ``flow-categorification'' of bulk-deformations (using a de Rham model), i.e., we will build the functor 
    \begin{equation}\label{eqn:ahat}
    \smallint{}_{\scrC/\bbX}\ch(\scrE):\flow^{\rm GRR}\to\dgnerve(\Mod_{\bC[\underline b]}).
    \end{equation}

\begin{construction}\label{construction:ahatfunctor}
First, recall that every GRR flow $n$-simplex $\bbX\equiv\bbX_{0\cdots n}$ has a fixed complex vector bundle over its universal curve $\pi:\scrC\to\bbX$, denoted $\scrE\to\scrC$. Now, we may fix a choice of Riemann metric on any $\scrC\to\bbX$ which is compatible with face maps by inducting over the dimension of $\bbX_{xy}$ and $n$. Moreover, recall $\scrE$ already comes with a choice of Hermitian connection which is compatible with face maps. The upshot is as follows. 
\begin{enumerate}
\item Consider $\scrE_{xy}\to\scrC_{xy}$; we define 
    \begin{equation}
    \ch(\scrE_{xy};{\bC[\underline b]})\equiv\sum_{\rho\geq0}\ch_{\rho+2}(\scrE_{xy})b_{\rho+1}\in\Omega^2(\scrC_{xy};{\bC[\underline b]}),
    \end{equation}
where $\ch_{\rho+2}(\scrE_{xy})$ is the $(2\rho+4)$-nd graded piece of the usual Chern character. We may make the analogous considerations for $\scrE_{xy,k}\to\scrC_{xy,k}$. The reason for the shift by 1 between the indexing of $\ch$ and the indexing of $b$, in this case, will soon be apparent.

\item We have a coherent cocycle $\frakc\frakh(\scrE)$ on $\scrC$ given by 
    \begin{align}
    \ch(\scrE;{\bC[\underline b]})_{xy}\equiv\ch(\scrE_{xy};{\bC[\underline b]})\in\Omega^2(\scrC_{xy};{\bC[\underline b]})&\;\;\textrm{resp.} \\
    \ch(\scrE;{\bC[\underline b]})_{xy,k}\equiv\ch(\scrE_{xy,k};{\bC[\underline b]})\in\Omega^2(\scrC_{xy,k};{\bC[\underline b]})&.
    \end{align}
Note, since $\scrE$ is trivialized, as a complex vector bundle, over the tube-like ends, $\ch(\scrE_{xy};{\bC[\underline b]})$ resp. $\ch(\scrE_{xy,k};{\bC[\underline b]})$ vanishes over the tube-like ends. 
\item Lemma \ref{lem:push} shows that $\pi_!\frakc\frakh(\scrE)$ is an additively coherent cocycle on $\bbX$ of degree 0. (This is the reason for the shift by 1 in (1), i.e., in this case, fiber integration decreases the differential form grading by 2.)
\end{enumerate}
We now have the multiplicatively coherent cocycle $\exp\big(\pi_!\frakc\frakh(\scrE)\big)$ on $\bbX$; by proceeding in the exact same way as Construction \ref{construction:chernfunctor} (using $\exp\big(\pi_!\frakc\frakh(\scrE)\big)$ instead of $\frakc\frakh(I)$), we have constructed \eqref{eqn:ahat} as a functor of stable $\infty$-categories.
\end{construction}

\begin{rem}
Observe, Construction \ref{construction:ahatfunctor} implicitly incorporates the relative $\widehat{A}$-genus of the various fiber bundles $\scrC_{xy}\to\bbX_{xy}$ resp. $\scrC_{xy,k}\to\big(\partial^k\bbX_{xy}\big)^\circ$; fortunately, the relative $\widehat{A}$-genus of a topologically trivial $(\bbR\times S^1)$-bundle is simply 1.
\end{rem}

The maps on endomorphisms of units induced by $\ch(\bbX)$ and $\smallint{}_{\scrC/\bbX}\ch(\scrE)$ are compatible in the following sense; this should be viewed as a form of the classical Grothendieck-Riemann-Roch theorem.

\begin{lem}\label{lem:inc grr}
    The following diagram commutes:
    \begin{equation}
        \begin{tikzcd}
            \pi_*\flow^{\rm GRR}(\unit,\unit) 
            \arrow[rr]
            \arrow[dr,bend right]
            &&
            \pi_*\flow^{\rm decU}(\unit,\unit)
            \arrow[dl,bend left]
            \\
            &
            H_*({\bC[\underline b]})
            &
        \end{tikzcd}
    \end{equation}
\end{lem}
\begin{proof}
We deduce this from an existing form of (not necessarily homotopy coherent) Bott periodicity.

Let $\bX$ represent a class in $\pi_d\flow^{\rm GRR}(\unit,\unit) $. Explicitly, this consists of the following data: a closed smooth $d$-manifold $X$, a universal curve $\scrC \to X$, a complex vector bundle $\scrE\to\scrC$ compatibly trivialized over $\ell_\scrC$ and the tube-like ends (together with appropriate extra data), and a stable isomorphism between the index bundle of a family of Cauchy-Riemann operators $D$ on $\scrE$ and $TX+\underline{\bbR}$. Since $X$ has no boundary, we may choose a trivialization $\scrC \cong X \times (\bR \times S^1)$.

Let $H \to \bC\bP^1$ be the tautological line bundle. Choosing a trivialization over $[0,+\infty]$ and then removing $\{0,+\infty\}$, we obtain a complex line bundle $\cH$ over $\bR \times S^1$ with (compatible) trivializations over $\ell\equiv\bbR\times\{1\}$ and the tube-like ends.

The proof of Bott periodicity in \cite[Sections 1-3]{Atiyah} shows that there is a stable isomorphism of vector bundles over $X \times (\bR \times S^1)$,
    \begin{equation}
    (\cH - \underline\bC) \boxtimes_\bbC \ind(D)\cong\scrE,
    \end{equation}
which is trivialized over $\ell$ and the tube-like ends. Since $\ch(\cH-\underline\bC)$ represents the fundamental class $[\bR \times S^1]$ of $\bR \times S^1$ in compactly supported cohomology, we see
    \begin{equation}
    [\bR \times S^1] \smile \ch\big(\ind(D)\big)=\ch(\scrE),
    \end{equation}
and hence that 
    \begin{equation}
    \ch_i\big(\ind(D)\big) = \pi_! \ch_{i+1}(\scrE),
    \end{equation}
whence the claim
\end{proof}

\begin{prop}\label{prop: rpoejgwpierngoreeerger}
The map induced by $\smallint{}_{\scrC/\bbX}\ch(\scrE)$ (on generators, after base change),
    \begin{equation}
        \pi_*\flow^{\rm GRR}(\unit, \unit)\otimes_\bbZ \bC \to \pi_*N^{\rm dg}(\Mod_{\bC[\underline b]})(\unit_{\bC[\underline b]}, \unit_{\bC[\underline b]}) \cong H_*({\bC[\underline b]})\cong\bC[\underline b],
    \end{equation}
    is an isomorphism. 
\end{prop}
\begin{proof}
Combine Propositions \ref{prop:grrtocomplex} and \ref{prop: rpoejgwpierngore} with Lemma \ref{lem:inc grr}.
\end{proof}

\subsection{Pre-Grothendieck-Riemann-Roch and bulk-deformations}
Construction \ref{construction:ahatfunctor} did not require that $\scrE$ was a GRR complex vector bundle, i.e., by performing the exact same construction, we obtain a functor
    \begin{equation}\label{eqn:ahat2}
    \smallint{}_{\scrC/\bbX}\ch(\scrE):\flow^{\rm GRR}\to\dgnerve(\Mod_{\bC[\underline b]}).
    \end{equation}
This subsection is dedicated to collecting the following result.

\begin{lem}
We have a commutative diagram: 
    \begin{equation}
    \begin{tikzcd}
    & \flow^{\rm GRR}\arrow[ddr,"\smallint{}_{\scrC/\bbX}\ch(\scrE)",bend left] & \\
    & & \\
    \flow^{\rm preGRR}\arrow[uur,"\calG",bend left]\arrow[rr,"\smallint{}_{\scrC/\bbX}\ch(\scrE)"] & & \dgnerve(\Mod_{\bC[\underline b]}).
    \end{tikzcd}
    \end{equation}
\end{lem}

\begin{proof}
The proof relies on the following elementary result.

\begin{claim}
Let $\pi:C\to M$ be a fiber bundle over a smooth manifold with compact oriented fibers. For any $\beta\in\Omega^*(M)$, we have that
    \begin{equation}
    \int_{C/M}\pi^*\beta=0.
    \end{equation}
\end{claim}

\begin{proof}[Proof of claim]
Recall, we have the projection formula:
    \begin{equation}
    \int_{C/M}\alpha\wedge\pi^*\beta=\Bigg(\int_{C/M}\alpha\Bigg)\wedge\beta,
    \end{equation}
where $\alpha\in\Omega^*(C)$ has vertical compact support and $\beta\in\Omega^*(M)$; the claim follows by using the case $\alpha=1$ of \emph{idem}.
\end{proof}

Consider a pre-GRR flow $n$-simplex $\bbX\equiv\bbX_{0\cdots n}$ Now, the lemma follows by observing
    \begin{equation}
    \pi_!\frakc\frakh(\scrE\oplus\sigma^*\scrF;{\bC[\underline b]})=\pi_!\frakc\frakh(\scrE;{\bC[\underline b]}),
    \end{equation}
where $\scrF\to\ell_\scrC$ is the complex vector bundle used in the construction of $\calG(\bbX)$, as additively coherent cocycles on $\calG(\bbX)$; this utilizes the aforementioned claim.
\end{proof}

Moreover, we may define a functor 
    \begin{equation}
    \ind(\bbX):\flow^{\rm preGRR}\to\flow^{\rm decU}
    \end{equation}
as the composition 
    \begin{equation}
    \flow^{\rm preGRR}\xrightarrow{\calG}\flow^{\rm GRR}\xrightarrow{\ind(\bbX)}\flow^{\rm decU}.
    \end{equation}

\begin{rem}
Consider the diagram:
\begin{equation}\label{eq:diag}
    \begin{tikzcd}
        \flow^{\rm GRR}
        \arrow[rr, "\ind(\bX)"]
        \arrow[ddr, swap, "\smallint{}_{\scrC/\bbX}\ch(\scrE)",bend right]
        &&
        \flow^{\rm decU}
        \arrow[ddl, "\ch(\bX)",bend left]
        \\&&\\
        &N^{\rm dg}(\Mod_{\bC[\underline b]}).
    \end{tikzcd}
\end{equation}
If this diagram were to commute, Theorem \ref{thm:tech} would follow immediately and, in fact, we would prove a stronger chain-level result, cf. Remark \ref{rem:eopjfpwnwGN}.
\end{rem}

\subsection{Grothendieck-Riemann-Roch, revisited}
Putting all of the previous parts together, we are ready to prove Theorem \ref{thm:tech}. Since the details of the proof are rather long-winded, we first provide a sketch.

    \begin{proof}[Sketch proof of Theorem \ref{thm:tech}]

        Since both sides of (\ref{eq: eiongwiepngreping}) commute with directed colimits, it suffices to restrict our discussion to finite flow categories (we indicate restriction to the full subcategory of finite flow categories by a subscript $(-)_{\rm fin}$); in this case, both functors land in the subcategory $N^{\rm dg}(\perf_{\bC[\underline b]})$.
        
        $\flow^{\rm decU}_{\rm fin}$ and $\flow^{\rm GRR}_{\rm fin}$ are stable $\infty$-categories and should be viewed as $\infty$-categories enriched in $\bS$-modules (i.e., spectra); moreover, $N^{\rm dg}(\perf_{\bC[\underline b]})$ should be viewed as a $\bC$-linear category. All three categories: $\flow^{\rm decU}_{\rm fin}$, $\flow^{\rm GRR}_{\rm fin}$ and $N^{\rm dg}(\perf_{\bC[\underline b]})$, each have a preferred generator, and the functors between them send preferred generators to preferred generators. Combined with Propositions \ref{prop:grrtocomplex}, \ref{prop: rpoejgwpierngore}, and \ref{prop: rpoejgwpierngoreeerger}, all three arrows in \eqref{eq:diag} are equivalences after base changing the top two categories to $\bbC$. I.e., the failure for \eqref{eq:diag} to commute is classified by an autoequivalence of $\perf_{\bC[\underline b]}$.

        By a result of To\"en \cite{Toen}, DG-endofunctors of $\perf_{\bC[\underline b]}$ are classified by ${\bC[\underline b]}$-${\bC[\underline b]}$ bimodules. In fact, autoequivalences of $\perf_{\bC[\underline b]}$ preserving the preferred generator correspond to bimodules isomorphic to ${\bC[\underline b]}$ as a left ${\bC[\underline b]}$-module, but with right module action conjugated by an autoequivalence $\psi$ of ${\bC[\underline b]}$; it follows from Lemma \ref{lem:inc grr} that $\psi$ acts as the identity on homology -- this is sufficient to conclude the theorem. 
    \end{proof}
    The rest of this section provides a more detailed account of this argument, starting with a closer look at how to relate DG-categories and stable $\infty$-categories.

\subsubsection{$\infty$-categorical preliminaries}
    Let $\Cat_\infty$ be the $\infty$-category of $\infty$-categories, and, given any $\cC \in \Cat_\infty$, we denote by $\Ho(\cC)$ its underlying homotopy category (this is an ordinary category). We write $\Cat^\Ex_\infty\subset\Cat_\infty$ for the subcategory of stable $\infty$-categories whose morphisms consist of exact functors; this should be thought of as the $\infty$-category of $\infty$-categories $\cC$, such that $\Ho(\cC)$ is a triangulated category, whose morphisms are functors which induce triangulated functors on homotopy categories. Cf. \cite[Section 2.2]{BGT} for a more detailed discussion. The inclusion $\Cat^\Ex_\infty \to \Cat_\infty$ is conservative and a right adjoint by \cite[Theorem 1.1.4.4]{HA}. 
\subsubsection{DG-nerve, revisited} 
    Recall the construction $N^{\rm dg}(-)$ from Section \ref{sec:alg}. As already mentioned, \cite[Proposition 1.3.1.20]{HA} shows this construction can be upgraded to a functor: 
        \begin{equation}
        N^{\rm dg}: \Cat_{\rm dg} \to \Cat_\infty,
        \end{equation}
    where $\Cat_{\rm dg}$ is the category of DG-categories.

    \begin{rem}
        There is a technical point here: $\Cat_{\rm dg}$ forms a model category rather than an $\infty$-category, so the correct technical statement is that there is a functor to $\Cat_\infty$ from the $\infty$-category $N(\Cat_{\rm dg}^{\rm cof})[W^{-1}]$ induced by the model category structure on $\Cat_{\rm dg}$. Explicitly, this is obtained from $\Cat_{\rm dg}$ by taking the full subcategory of cofibrant objects, taking the nerve, and formally inverting all the weak equivalences. To avoid clutter in the sequel, we will abuse notation and write $\Cat_{\rm dg}$ to also mean $N(\Cat_{\rm dg}^{\rm cof})[W^{-1}]$ where appropriate, and similarly for other model categories arising. We are primarily interested in the homotopy categories of these categories, and applying $N\big((-)^{\rm cof}\big)[W^{-1}]$ does not affect the homotopy category. Cf. \cite[Section 2.2]{BGT} for further discussion.
    \end{rem}

    \begin{lem}
       After restricting to the full subcategory of pretriangulated DG-categories, $N^{\rm dg}$ admits a factorization $N^{\rm dg}_\infty$ through $\Cat^\Ex_\infty$:
        \begin{equation}
            \begin{tikzcd}
                \Cat_{\rm dg}^{\rm pretri}
                \arrow[r]
                \arrow[d, "N^{\rm dg}_\infty",swap]
                &
                \Cat_{\rm dg}
                \arrow[d, "N^{\rm dg}"]
                \\
                \Cat^\Ex_\infty
                \arrow[r]
                &
                \Cat_\infty.
            \end{tikzcd}
        \end{equation}
    \end{lem}
    \begin{proof}
        That there is a factorization at the level of objects is proved in \cite[Theorem 4.3.1]{Faonte:Nerve} (cf. also \cite[Theorem 5.1]{Ornaghi:Nerve}). The same argument as in \cite{Faonte:Nerve} shows $N^{\rm dg}$ sends DG-functors (between pretriangulated DG-categories) to exact functors, thus giving a factorization at the level of morphisms. Note, no coherence needs checking, since $\Cat^\Ex_\infty \to \Cat_\infty$ is an inclusion of simplicial sets which is an inclusion at the level of objects and at the level of morphisms (essentially by definition, cf. \cite[Definition 2.12]{BGT}).
    \end{proof}
    
\subsubsection{Recollections on enriched categories}
    One should think of $\infty$-categories as very close to categories enriched in spaces, and stable $\infty$-categories as very close to categories enriched in spectra. In the sequel, we recall some more precise statements in this direction.

    Let $\mathrm{Spc}$ be the category of simplicial sets (``spaces'') and $\Cat_{\rm Spc}$ the category of categories enriched in simplicial sets. We write $\Cat_\Sp$ for the category of small categories enriched in symmetric spectra equipped with the model category structure of \cite[Section 2.1]{BGT} (cf. also \cite{Tabuada:SpectralCats}). These are closely related to stable $\infty$-categories: there is a ``Yoneda'' functor of $\infty$-categories,
        \begin{equation}
        \bY:\Cat_\infty \to \Cat_{\rm Sp},\footnote{As already mentioned, we drop $N\big((-)^{\rm cof}\big)[W^{-1}]$ from the notation.}
        \end{equation}
    which is a fully faithful right adjoint, cf. \cite[Definition 4.7, Theorem 4.22]{BGT}. Any $\cC \in \Cat_\infty$ (functorially) determines a category enriched in spaces which we denote as $\mathfrak{C}(\cC)$ (cf. \cite[Section 1.1.5]{HTT} and \cite[Section 2.2]{BGT}); this yields an equivalence of $\infty$-categories  
        \begin{equation}
        \mathfrak{C}: \Cat_\infty \to \Cat_{\rm Spc}
        \end{equation}
    which is left adjoint to the simplicial nerve, and $\bY(\cC)$ is defined to be the full spectral subcategory of $\mathrm{Fun}\big(\mathfrak{C}(\cC)^{\rm op}, \Sp\big)$ generated by stably representable functors (on bifibrant objects). From the construction, we find that $\Ho(\cC) \cong \Ho\big(\bY(\cC)\big)$ as triangulated categories. We will later restrict to the full subcategory of pretriangulated spectral categories $\Cat_\Sp^{\rm pretri}$; one could equivalently instead work with a model category on $\Cat_\Sp$ whose equivalences are functors inducing equivalences on pretriangulated closures, cf. \cite[Section 2.1]{BGT} for further discussion. Crucially, $\bY$ restricts to an equivalence $\Cat^\Ex_\infty \to \Cat_\Sp^{\rm pretri}$, cf. \cite[Remark 4.24]{BGT}.

    There is an adjunction $\Sigma^\infty: {\rm Spc} \rightleftharpoons \mathrm{Sp}: \Omega^\infty$ inducing one on enriched categories. Observe, we have the following commutative diagram (essentially by applying Yoneda):

    \begin{equation}\label{eq: weiorgheaouf}
        \begin{tikzcd}
            \Cat^\Ex_\infty 
            \arrow[r, "\bY"]
            \arrow[d]
            &
            \Cat_\Sp^{\rm pretri} 
            \arrow[d, "\Omega^\infty"]
            \\
            \Cat_\infty 
            \arrow[r, "\mathfrak{C}"]
            &
            \Cat_{\rm Spc}.
        \end{tikzcd}
    \end{equation}

    Let $H\bC$ be the $\bbC$-Eilenberg-Maclane spectrum, $\Cat_{H\bC}$ the category of small categories enriched in the monoidal model category of $H\bC$-modules, and $\Cat_{H\bC}^{\rm pretri}$ the full subcategory of pretriangulated ones, cf. \cite[Section 4.2]{Tabuada:THHTC}. By applying $\pi_*$ to the mapping spectra, the homotopy category $\Ho(\cC)$, for any $\cC \in \Cat_{H\bC}$, is naturally enriched in graded $\bC$-modules. There is an adjunction induced by the unit map $\bS \to H\bC$:
    \begin{equation}\label{eq: irehgupoebgouerbg}
        \begin{tikzcd}
            \Cat_\Sp
            \arrow[rr, "-\otimes_\bS H\bC",shift left=2, "{\textrm{\rotatebox[origin=c]{270}{$\dashv$}}}" swap]
            &&
            \arrow[ll, "\bF", shift left=2]
            \Cat_{H\bC};
        \end{tikzcd}
    \end{equation}
    this is a Quillen adjunction and hence induces an adjunction on the corresponding homotopy categories, cf.  \cite[Appendix A]{Tabuada:THHTC}. The left adjoint does not preserve the property of being pretriangulated, however, post-composing with the functor which sends $\cC\in \Cat_{H\bC}$ to its pretriangulated closure in the sense of \cite[Section 5]{BM:LocTHHTC} (note, this pretriangulation functor induces a left adjoint on homotopy categories), we obtain a similar adjunction to \eqref{eq: irehgupoebgouerbg} between the full subcategories of pretriangulated objects on both sides; we denote this adjunction the same way.

    Here, $-\otimes_\bS H\bC$ sends a category $\cC\in\Cat_{\rm Sp}$ to the pretriangulated closure of a category with (1) the same objects as $\cC$ and (2) mapping spaces the mapping spectra in $\cC$ smashed with $H\bC$. Moreover, for any $\cC\in\Cat_{\rm Sp}$, $\cC \otimes_\bS H\bC$ has homotopy category generated (in the triangulated sense) by the linear subcategory $\Ho(\cC) \otimes_\bZ \bC$. Meanwhile, the forgetful functor $\bF$ on $\cC\in\Cat_{H\bbC}$ (1) preserves objects and (2) takes the underlying mapping spectrum of the $H\bC$-module mapping spectrum. Finally, note that, for DG-categories that are already pretriangulated, $(-)^{\rm pretri}$ preserves homotopy categories.

    There is a functor which induces an isomorphism on homotopy categories (in fact, the functor is the right adjoint of a Quillen equivalence):
    \begin{equation}
        \bT:\Cat_{\rm dg}\xrightarrow{\sim} \Cat_{H\bC},
    \end{equation} 
    cf.  \cite[Section 9]{Tabuada:THHTC}, which similarly restricts to an equivalence on full subcategories of pretriangulated objects. Strictly speaking, $\bT$ is a zig-zag consisting of three Quillen adjunctions. The important aspect of this for the present article is that $\bT$ induces an isomorphism on homotopy categories and preserves homotopy categories on objects: 
        \begin{equation}
        \Ho(\cC) \cong \Ho\big(\bT(\cC)\big),\;\;\cC\in\Cat_{\rm dg},
        \end{equation}
    as categories enriched in graded $\bC$-modules.
\subsubsection{Adjoints for the DG-nerve}
    \begin{lem}
        The following diagram commutes at the level of homotopy categories:
        \begin{equation}\label{eq: righesoprbg}
            \begin{tikzcd}
                \Cat^{\rm pretri}_{\rm dg}
                \arrow[r, "\bT"]
                \arrow[dd, "N^{\rm dg}",swap]
                &
                \Cat^{\rm pretri}_{H\bC} 
                \arrow[d, "\bF"]
                \\
                &
                \Cat^{\rm pretri}_\Sp 
                \arrow[d, "\Omega^\infty"]
                \\
                \Cat_\infty 
                \arrow[r, "\mathfrak{C}"]
                &
                \Cat_{\rm Spc}.
            \end{tikzcd}
        \end{equation}
    \end{lem}
    \begin{proof}
        All functors in this diagram are right adjoints. Therefore, it suffices to check that the diagram induced by passing to left adjoints commutes. The adjoint functor obtained by going up and then left takes a category enriched in spaces and sends it to the pretriangulated DG-category with (1) the same objects and (2) morphisms obtained by applying the functor $\mathrm{Spc} \to \mathrm{Ch}(\bC)$ given by: passing to suspension spectra, smashing with $H\bC$, and applying the Dold-Kan correspondence $\Mod_{H\bC} \simeq \mathrm{Ch}(\bC)$. Meanwhile, the adjoint functor going left and then up (1) again preserves objects and (2) applies the functor $C_*(-; \bC)$, given by taking simplicial chains, levelwise to the morphism spaces, cf. \cite[Proposition 4.5]{RiveraZeinalian}. That these two monoidal functors $\mathrm{Spc} \to \mathrm{Ch}(\bC)$ are equivalent follows from the facts that: both can be upgraded to colimit-preserving Quillen functors of monoidal model categories (since all intermediate steps in their construction can), $\mathrm{Spc}$ is generated by a point, and both functors send the point to $\bC$.
    \end{proof}
    \begin{prop}
        The following diagram commutes at the level of homotopy categories:
        \begin{equation}\label{eq: oirewgsoeahg}
            \begin{tikzcd}
                \Cat^{\rm pretri}_{\rm dg}
                \arrow[d, "N^{\rm dg}_\infty",swap]
                \arrow[r, "\bT"]
                &
                \Cat_{H\bC}^{\rm pretri}
                \arrow[d, "\bF"]
                \\
                \Cat^\Ex_\infty 
                \arrow[r, "\bY"]
                &
                \Cat_\Sp^{\rm pretri}.
            \end{tikzcd}
        \end{equation}
    \end{prop}
    \begin{proof}
        Observe, \eqref{eq: righesoprbg} is the concatenation of \eqref{eq: oirewgsoeahg} and \eqref{eq: weiorgheaouf}. The claim follows because the downwards maps in \eqref{eq: weiorgheaouf} are injective on the sets of isomorphism classes of objects and homotopy classes of maps.
    \end{proof}
    \begin{cor}\label{cor: repiwghewigh}
        The functor induced by $N^{\rm dg}_\infty$ on homotopy categories has a left adjoint which we denote by $\bL$:
        \begin{equation}
            \begin{tikzcd}
                \Ho(\Cat^\Ex_\infty)
                \arrow[rr, "\bL",shift left=2, "{\textrm{\rotatebox[origin=c]{270}{$\dashv$}}}" swap]
                &&
                \arrow[ll, "N^{\rm dg}_\infty", shift left=2]
                \Ho(\Cat^{\rm pretri}_{\rm dg})
            \end{tikzcd}
        \end{equation}
    \end{cor}
    \begin{proof}
        Explicitly, the left adjoint is given by $\bT^{-1} \circ(-\otimes_\bS H\bC)\circ \bY$. That this is a left adjoint follows from: the adjunctions stated above, fully faithfulness of $\bY$, and a diagram chase around \eqref{eq: oirewgsoeahg}.
    \end{proof}
    Combining the above discussions on homotopy categories, we we obtain the following result.
    \begin{cor}\label{cor: pirehgpiegnpi}
        For $\cC \in \Cat^\Ex_\infty$, there is a fully faithful embedding 
            \begin{equation}
            \Ho(\cC) \otimes_{\bZ} \bC \to \Ho\big(\bL(\cC)\big)
            \end{equation}
        whose image generates (here, we tensor with $\bC$ on each morphism group).
    \end{cor}
\subsubsection{Strictification} \label{sec: riehgpioejhngpejrg}
In this part, we return to the diagram \eqref{eq:diag} and use the adjunction we have constructed to obtain a diagram of DG-categories.
    
Recall, both $\flow^{\rm GRR}$ and $\flow^{\rm decU}$ are generated under colimits and shifts by $\unit$; in both cases, the subcategory of finite flow categories is generated by $\unit$ under finite colimits. Similarly, recall $\perf_{\bC[\underline b]}$ is generated as a DG-category by $\unit_{\bC[\underline b]}$.
    
The functors $\smallint{}_{\scrC/\bbX}\ch(\scrE)$ and $\ch(\bbX)$ send finite flow categories to perfect ${\bC[\underline b]}$-modules, and so restrict to functors:
    \begin{align}
        \smallint{}_{\scrC/\bbX}\ch(\scrE): \flow^{\rm GRR}_{\rm fin} 
        &\to 
        N^{\rm dg}_\infty(\perf_{\bC[\underline b]}),
        \\
        \ch(\bbX): \flow^{\rm decU}_{\rm fin}
        &\to 
        N^{\rm dg}_\infty(\perf_{\bC[\underline b]}).
    \end{align}
    
    \begin{lem}
    The adjoint functors of DG-categories,
        \begin{align}
            \bL\smallint{}_{\scrC/\bbX}\ch(\scrE): \bL\flow^{\rm GRR}_{\rm fin} 
            &\to 
            \perf_{\bC[\underline b]},
            \\
            \bL\ch(\bbX): \bL\flow^{\rm decU}_{\rm fin}
            &\to 
            \perf_{\bC[\underline b]},
        \end{align}
        are equivalences.
    \end{lem}
    \begin{proof}
    We consider the case of $\bL\ch(\bbX)$; the other case is identical. Observe, $\bL \ch(\bbX)$ sends the generator $\unit$ to the generator $\unit_{\bC[\underline b]}$. It follows from Proposition \ref{prop: rpoejgwpierngore} that $\bL\ch(\bbX)$ induces an isomorphism on homology of the endomorphisms of these objects. Finally, an induction over cones implies $\bL\ch(\bbX)$ is an equivalence.
    \end{proof}
    
    Since $\ind(\bbX)$ is an equivalence, so is $\bL\ind(\bbX)$. In particular, we obtain a zig-zag of equivalences in $\Cat_{\rm dg}$:
    \begin{equation}\label{eq: the zig zag}
        \perf_{\bC[\underline b]} \xleftarrow[\simeq]{\bL\ch(\bbX)} \bL\flow^{\rm decU}_{\rm fin} \xleftarrow[\simeq]{ \ind(\bbX)} \bL\flow^{\rm GRR}_{\rm fin} \xrightarrow[\simeq]{\bL\smallint{}_{\scrC/\bbX}\ch(\scrE)} \perf_{\bC[\underline b]};
    \end{equation}
    thus, we obtain an autoequivalence 
        \begin{equation}
        F: \perf_{\bC[\underline b]} \to \perf_{\bC[\underline b]}
        \end{equation}
    in $\Ho(\Cat_{\rm dg})$.

    By the above discussion on generators, $F(\unit_{\bC[\underline b]}) = \unit_{\bC[\underline b]}$. Identifying endomorphisms of $\unit_{\bC[\underline b]}$ with $\unit_{\bC[\underline b]}$, we obtain a map 
        \begin{equation}
        F_*: {\bC[\underline b]} \to {\bC[\underline b]}
        \end{equation}
    of DG-algebras in the derived category of DG-algebras (note, $F_*$ may be a zig-zag). By Lemma \ref{lem:inc grr}, the induced map on homology, $H_*(F_*)$, is the identity.

\subsubsection{Derived category of DG-categories}

    A result of To\"en (cf. also \cite[Page 738]{HA} for the $\infty$-categorical analogue) describes morphisms in the homotopy category of DG-categories explicitly.
    \begin{thm}[{\cite[Corollary 1.2]{Toen}}]
        Let $\cC$ and $\cD$ be DG-categories. There is a natural bijection between functors $\cC \to \cD$ in the homotopy category of DG-categories and quasi-isomorphism classes of $\cC$-$\cD$ bimodules:
        \begin{equation}
            \Ho(\Cat_{\rm dg})(\cC,\cD) \cong \ob \Ho(\Mod_{\cC \otimes \cD^{\rm op}})^{\rm rqr},
        \end{equation}
        where the superscript $(-)^{\rm rqr}$ denotes the full subcategory of \emph{right quasi-representable} bimodules; these are the bimodules $M$ such that, for each $c \in \cC$, the right $\cD$-module $M(c,-)$ is quasi-isomorphic to a Yoneda module $\cD(-, d)$ for some $d \in \cD$.\footnote{In \textit{loc. cit.}, a cofibrant replacement $Q\cC$ of $\cC$ is taken (this is equivalent to working with $\cC \otimes^\bL \cD^{\rm op}$-modules); this is not necessary in the present article since our ground ring is a field, cf. \cite{Toen:Lectures}.} 
    \end{thm}
    
    In the case of endofunctors of $\perf_{\bC[\underline b]}$, we may see this bijection explicitly following \cite[Pages 632-633]{Toen}. Since $\unit_{\bC[\underline b]}$ generates $\perf_{\bC[\underline b]}$, $\unit_{\bC[\underline b]} \otimes \unit_{\bC[\underline b]}$ generates $\perf_{\bC[\underline b]} \otimes \perf_{\bC[\underline b]}^{\rm op}$. It follows that restriction of bimodules defines an equivalence from $\perf_{\bC[\underline b]}$-$\perf_{\bC[\underline b]}$ bimodules to ${\bC[\underline b]}$-${\bC[\underline b]}$ bimodules. Thus, we have an injection:
    \begin{equation}\label{eq: rowegjuepiwhg}
        p: \Ho(\Cat_{\rm dg})(\perf_{\bC[\underline b]}, \perf_{\bC[\underline b]}) \to \ob\ho(\Mod_{{\bC[\underline b]} \otimes {\bC[\underline b]}^{\rm op}}).
    \end{equation}
    Let $M$ be a ${\bC[\underline b]}$-${\bC[\underline b]}$ bimodule and choose an equivalent bimodule $M'$ which is right flat (i.e., flat as a right ${\bC[\underline b]}$-module); this exists since we may take $M'$ to be a semi-free resolution of $M$ as a ${\bC[\underline b]}\otimes {\bC[\underline b]}^{\rm op}$-module, via \cite[Lemmas 13.3 \& 14.8]{Drinfeld}, and using the fact that ${\bC[\underline b]} \otimes {\bC[\underline b]}^{\rm op}$ is right flat. 

    We may define a DG-functor 
    \begin{equation}
    M \otimes^\bL_{\bC[\underline b]}-: \perf_{\bC[\underline b]} \to \Mod_{\bC[\underline b]}
    \end{equation}
    via the formula $B \mapsto M' \otimes_{\bC[\underline b]} B$, on objects, and $f \mapsto \identity_{M'} \otimes f$, on morphisms. If $M'$ and $M''$ are two right flat bimodules which are both equivalent to $M$, then such an equivalence induces an equivalence of functors as just defined; hence, $M \otimes^\bL_{\bC[\underline b]}-$ only depends on $M$ up to equivalence.
    
    \begin{lem}
        The functor $M \otimes^\bL_{\bC[\underline b]}-$ lands in $\perf_{\bC[\underline b]}$ if and only if $M$ is left perfect (i.e., perfect as a left ${\bC[\underline b]}$-module).
    \end{lem}
    \begin{proof}
        If $M'$ is left perfect, $M' \otimes_{\bC[\underline b]} {\bC[\underline b]}$ is perfect; inducting over the generation length of $B$ implies $M' \otimes_{\bC[\underline b]} B$ is perfect for all perfect $B$.

        Meanwhile, if $M \otimes^\bL_{\bC[\underline b]}-$ lands in $\perf_{\bC[\underline b]}$, then 
            \begin{equation}
            M \otimes^\bL_{\bC[\underline b]} {\bC[\underline b]} \simeq M
            \end{equation}
        must be perfect.
    \end{proof}
    Finally, Inspection of Toen's construction shows the following result.

    \begin{prop}
        Any DG-functor $f: \perf_{\bC[\underline b]} \to \perf_{\bC[\underline b]}$ is equivalent to one of the form $M \otimes^\bL_{\bC[\underline b]}-$, for some left perfect ${\bC[\underline b]}$-${\bC[\underline b]}$ bimodule $M$. Moreover, $M$ is uniquely determined by $f$ up to quasi-isomorphism. Finally, this construction provides an inverse to \eqref{eq: rowegjuepiwhg}
    \end{prop}
    
\subsubsection{Computing the functor}\label{sec: oeirhngouebgoue}
    Let 
        \begin{equation}
        F: \perf_{\bC[\underline b]} \to \perf_{\bC[\underline b]}
        \end{equation}
    be the autoequivalence constructed in Part \ref{sec: riehgpioejhngpejrg}; we denote by $M$ the corresponding ${\bC[\underline b]}$-${\bC[\underline b]}$ bimodule. As already mentioned, $F$ sends $\unit_{\bC[\underline b]}$ to $\unit_{\bC[\underline b]}$. In terms of $M$, $F$ sends $\unit_{\bC[\underline b]}$ to $M$ viewed as a left ${\bC[\underline b]}$-module, i.e., 
        \begin{equation}
        \unit_{\bC[\underline b]}\simeq M
        \end{equation}
    as left ${\bC[\underline b]}$-modules.

    Fix an isomorphism 
        \begin{equation}
        \theta: H_*({\bC[\underline b]})\xrightarrow{\sim} H_*(M)
        \end{equation}
    of left $H_*({\bC[\underline b]})$-modules. The ring of left $H_*({\bC[\underline b]})$-module endomorphisms of $H_*({\bC[\underline b]})$ is isomorphic to $H_*({\bC[\underline b]})^{\rm op}$ (via sending an element to the endomorphism given by multiplication on the right), so the right $H_*({\bC[\underline b]})$-module structure on $H_*(M)$ is determined by a map of rings 
        \begin{equation}
        \psi: H_*({\bC[\underline b]})^{\rm op} \to H_*({\bC[\underline b]})^{\rm op}.
        \end{equation}
    \begin{lem}\label{lem: reiohgeiwoaprp}
        $\psi$ is an isomorphism.
    \end{lem}
    \begin{proof}
        Suppose not. Since $H_i({\bC[\underline b]})$ is finite-dimensional for each $i$, $\psi$ must fail to be injective. Let  $x \neq 0 \in H_i({\bC[\underline b]})$ be an element in $\psi$'s kernel and consider 
            \begin{equation}
            C \equiv \unit_{\bC[\underline b]} \oplus \Sigma \unit_{\bC[\underline b]},\;\; C' \equiv \mathrm{Cone}(\unit_{\bC[\underline b]} \xrightarrow{r_x} \unit_{\bC[\underline b]}) \in \perf_{\bC[\underline b]},
            \end{equation}
            where $r_x$ is multiplication on the right by $x$. Since $C$ and $C'$ have distinct homology, they cannot be quasi-isomorphic; but, $F(C)\simeq F(C')$, contradicting that $F$ is an equivalence.
    \end{proof}

    \begin{cor}\label{cor: reihgeuirbohiubgsepg}
    $M \simeq {\bC[\underline b]}$ as right ${\bC[\underline b]}$-modules.
    \end{cor}
    \begin{proof}
        Let $x \in M$ represent a class in $H_*(M)$ which is a generator as a right $H_*({\bC[\underline b]})$-module. Now, multiplication by $x$ explicitly produces the desired quasi-isomorphism.
    \end{proof}
    \begin{lem}\label{lem: rwoihgoewuhriog}
        $\psi$ is the identity.
    \end{lem}
    \begin{proof}
        Since we may describe $F$ in terms of $M$, we may chase through the construction of $\psi$ to find that 
            \begin{multline}
            H_*(F_*): H_*({\bC[\underline b]}) \cong H_*\big(\perf_{\bC[\underline b]}(\unit_{\bC[\underline b]}, \unit_{\bC[\underline b]})\big) \to \\
            H_*\big(\perf_{\bC[\underline b]}(\unit_{\bC[\underline b]},\unit_{\bC[\underline b]})\big) \cong H_*({\bC[\underline b]})
            \end{multline}
        is exactly $\psi^{-1}$; but recall, via Lemma \ref{lem:inc grr}, we saw that $H_*(F_*)$ is the identity, whence the claim.
    \end{proof}

    \begin{cor}\label{cor: rieowhgesehgsrethouhbg}
        $H_*(M) \cong H_*({\bC[\underline b]})$ as $H_*({\bC[\underline b]})$-$H_*({\bC[\underline b]})$ bimodules.
    \end{cor}

    \begin{proof}
    Follows from the above discussion along with Lemma \ref{lem: rwoihgoewuhriog}.
    \end{proof}
    
    \begin{lem}\label{lem: rioehgeoubgoeshrg}
        Consider $B \in \perf_{\bC[\underline b]}$. We have that 
            \begin{equation}
            H_*(B) \cong H_*(F(B))
            \end{equation}
        as left $H_*({\bC[\underline b]})$-modules.
    \end{lem}
    
    \begin{proof}
        Let $M'$ be a bimodule equivalent to $M$ which is right flat and $x\in M$ represent a class in $H_*(M)$ which is a generator as a left (equivalently right) $H_*({\bC[\underline b]})$-module. The map 
            \begin{equation}
            g_B: B \to F(B) = M' \otimes_{\bC[\underline b]} B
            \end{equation}
         sending $c \mapsto x \otimes c$ is a chain map. In fact, $g_B$ is a quasi-isomorphism for $B=\unit_{\bC[\underline b]}$, so an inductive argument shows $g_B$ is a quasi-isomorphism for all $B \in \perf_{\bC[\underline b]}$.
        
        Now, although $g_B$ might not be a ${\bC[\underline b]}$-module map, $H_*(g_B)$ is an $H_*({\bC[\underline b]})$-module map, as follows: for $c \in B$ and $r \in {\bC[\underline b]}$, 
            \begin{multline}
            H_*(g_B)(rc) = [x \otimes rc] = [xr \otimes c] = \\
            [xr] \otimes [c] = [rx] \otimes [c] = r[x \otimes c] = rH_*(g_B)(c)
            \end{multline}
        (the identity $[rx]=[xr]$ follows from Corollary \ref{cor: rieowhgesehgsrethouhbg}).
    \end{proof}
    
\begin{proof}[Proof of Theorem \ref{thm:tech}]
    Since the constructions in both sides of \eqref{eq: eiongwiepngreping} commute with directed colimits, it suffices to consider the case $\bbX$ is a finite flow category. By construction of $F$, we have that:
    \begin{equation}
        H_*\Big(\smallint{}_{\scrC/\bbX}\ch(\scrE)\Big) = H_*\bigg(F\Big(\ch\big(\ind(\bbX)\big)\Big)\bigg).
    \end{equation}
    The conclusion follows from Lemma \ref{lem: rioehgeoubgoeshrg}.
\end{proof}
\subsection{Chain-level enrichment: a conjectural outline}\label{sec: opwjgfoebgjorbtgojrbt}
We briefly outline an argument that suggests \eqref{eq:diag} does indeed commute (after restricting to finite flow categories and perfect $\bbC[\underline{b}]$-modules, as before); however, the details are beyond the scope of the present article. In particular, this uses, in an essential way, symmetric monoidality of $\flow^{\rm decU}$ and $N^{\rm dg}(\Mod_{\bC[\underline b]})$; recall, symmetric monoidality in the fomer case is, at present, only conjectural.

    \begin{proof}[Conjectural outline of proof of commutativity of \eqref{eq:diag}]
        Now, $\flow^{\rm GRR}$, as defined in the present article, is not obviously symmetric monoidal since it is constructed from the gluing of Riemann surfaces; however, one can modify the construction so that, instead of gluing of Riemann surfaces at tube-like ends, it takes direct sums of the vector bundles over a fixed Riemann surface (cf. \cite[Section 5.4]{PS25b} for an Eckmann-Hilton-type argument relating these two approaches). The upshot is that we may assume \eqref{eq:diag} is a diagram of symmetric monoidal stable $\infty$-categories.
    
        After restricting to finite flow categories and perfect ${\bC[\underline b]}$-modules, we may use the same argument as Part \ref{sec: oeirhngouebgoue} to reduce commutativity of \eqref{eq:diag} to an automorphism $\psi$ of ${\bC[\underline b]}$. Since all of our functors are symmetric monoidal, $\psi$ is an $\bE_\infty$-autoequivalence of ${\bC[\underline b]}$.\footnote{In Part \ref{sec: oeirhngouebgoue}, we constructed $\psi$ as an autoequivalence of $H_*({\bC[\underline b]})$ as a graded ring, but similar methods could have upgraded this to an $A_\infty$-autoequivalence of ${\bC[\underline b]}$ itself; in the presence of symmetric monoidal structures, this can be further upgraded to an $\bE_\infty$-autoequivalence.} Now, ${\bC[\underline b]}$ is a free $\bE_\infty$-algebra on generators $b_1, b_2, \ldots$, and so its autoequivalences are determined uniquely by their action on homology; therefore, the same computation implies $\psi$ is homotopic to the identity as a map of $\bE_\infty$, and hence also $A_\infty$, algebras.
    \end{proof}

\section{Bulk-deformations of symplectic cohomology via Morse theory}
In the remainder of the present article, we will work in the following geometric context. We consider a Liouville domain, i.e., a compact exact symplectic $2n$-manifold $\big(\widehat{M},\widehat{\omega}=d\widehat{\theta}\big)$ with necessarily non-empty boundary $\partial\widehat{M}$ whose Liouville vector field $\widehat{Z}$, defined via
    \begin{equation}
    \widehat{\omega}(\widehat{Z},-)=\widehat{\theta},
    \end{equation}
points outwards along $\partial\widehat{M}$. We may complete $\big(\widehat{M},\widehat{\omega}=d\widehat{\theta}\big)$ to a Liouville manifold $(M,\omega=d\theta)$ by gluing on the positive-half of the symplectization of the contact boundary $\big(\partial\widehat{M},\ker\widehat{\theta}\vert_{\partial\widehat{M}}\big)$:
    \begin{equation}
    M\equiv\widehat{M}\cup_{\partial\widehat{M}}\big([1,+\infty)\times\partial\widehat{M}\big);
    \end{equation}
we denote the radial coordinate on the cylindrical end by $r$. We will assume $M$ is \emph{graded}, i.e., $2c_1(M)=0$; in particular, we fix an $\bbR$-polarization of $\Lambda^n_\bbC TM$, i.e.,
    \begin{equation}
    \Lambda^n_\bbC TM\cong\Lambda\otimes_\bbR\underline{\bbC}.
    \end{equation}

\subsection{Morse theory}\label{sec:51}
Let $f\in C^\infty(\widehat{M})$ be a Morse function and $g$ a (sufficiently generic) Riemannian metric on $M$; we will assume $\grad f$ points strictly outwards on $\partial\widehat{M}$. Given any critical point $a\in\crit(f)$, we denote by $W^u(a)$ resp. $W^s(a)$ the unstable resp. stable manifold of $a$. 

Given any two $a,b\in\crit(f)$, we consider the moduli space $\widetilde{\scrM}_{ab}$ of Morse trajectories connecting $a$ to $b$, i.e., maps 
    \begin{equation}
    \gamma:\bbR_s\to\widehat{M}
    \end{equation}
satisfying 
    \begin{equation}
    \begin{cases}
    \partial_s\gamma+\grad f(\gamma)=0 & \\
    \lim_{s\to-\infty}\gamma(s)=a & \\
    \lim_{s\to+\infty}\gamma(s)=b.
    \end{cases}
    \end{equation}
This is, for generic data, a smooth manifold of dimension $I(a)-I(b)$, where $I(-)$ denotes the Morse index, whose tangent bundle is classified by the index bundle of the family of surjective Fredholm operators 
    \begin{equation}
    D\Xi\equiv\big\{D\Xi_\gamma:W^{1,2}(\bbR;\gamma^*T\widehat{M})\to L^2(\bbR;\gamma^*T\widehat{M})\big\}
    \end{equation}
given by linearizing. When $a\neq b$, there is a free proper $\bbR$-action on $\widetilde{\scrM}_{ab}$ given by time-shift; we denote by $\scrM_{ab}$ the $\bbR$-quotient. We denote by $\bbM_{ab}$ the Gromov-compactification of $\scrM_{ab}$ given by allowing breakings at critical points; also, we define $\bbD\bbM_{ba}\equiv\bbM_{ab}$. Let $\frako_a$ be the orientation line associated to $a$, i.e., the determinant line of the invertible self-adjoint operator $D\Xi_a$ given by linearizing at $a$. Recall, $D\Xi_a$ is the asymptotic operator of $D\Xi_\gamma$ at $-\infty$ for any gradient flow $\gamma$ starting at $a$. We denote by $\abs{\frako_a}$ the free $\bbZ$-module generated by the two possible orientations of $\frako_a$ modulo the relation that the sum of opposite orientations vanishes. 
\begin{enumerate}
\item The $(f,g)$-Morse chain complex 
    \begin{equation}
    CM_j(M;\bbZ)\equiv\bigoplus_{I(a)=j}\abs{\frako_a}
    \end{equation}
has the differential whose $\abs{\frako_a}-\abs{\frako_b}$ component is given by 
    \begin{equation}
    \partial_{a,b}\equiv\sum_{\substack{\gamma\in\bbM_{ab} \\ I(b)=I(a)-1}}\mu_\gamma,
    \end{equation}
where $\mu_\gamma:\frako_a\xrightarrow{\sim}\frako_b$ is the isomorphism induced on orientation lines by $\gamma$; we denote by $HM_*(M;\bbZ)$ the $(f,g)$-Morse homology. Note, we have an identification 
    \begin{equation}
    HM_*(M;\bbZ)\cong H_*(\widehat{M},\partial\widehat{M};\bbZ).
    \end{equation}
\item The $(f,g)$-Morse cochain complex 
    \begin{equation}
    CM^j(M;\bbZ)\equiv\bigoplus_{I(a)=j}\abs{\frako_a}
    \end{equation}
has the codifferential whose $\abs{\frako_b}-\abs{\frako_a}$ component is given by 
    \begin{equation}
    \delta_{b,a}\equiv\sum_{\substack{\gamma\in\bbD\bbM_{ba} \\ I(a)=I(b)+1}}\mu_\gamma,
    \end{equation}
where $\mu_\gamma:\frako_b\xrightarrow{\sim}\frako_a$ is the isomorphism induced on orientation lines by $\gamma$; we denote by $HM^*(M;\bbZ)$ the $(f,g)$-Morse cohomology. Note, we have an identification 
    \begin{equation}
    HM^*(M;\bbZ)\cong H^*(\widehat{M};\bbZ)\cong H^*(M;\bbZ).
    \end{equation}
\end{enumerate}

\subsection{Symplectic cohomology}
\subsubsection{Floer data}\label{sec:521}
Let $H\in C^\infty(M)$ be a Hamiltonian. We will consider Hamiltonians which are linear at infinity, i.e., there exists a constant $\kappa\in\bbR_+$ such that, along the cylindrical end, we have
    \begin{equation}
    H(r,p)=\kappa p,\;\;r\gg1.
    \end{equation}
Also, we will assume $\kappa$ is not the length of a Reeb orbit of $\partial\widehat{M}$. We denote by $X_H$ the associated Hamiltonian vector field: 
    \begin{equation}
    \omega(X_H,-)=-dH.
    \end{equation}
Moreover, we denote by $\phi_H^t$ the associated Hamiltonian flow (i.e., the flow generated by $X_H$). Let $\chi(H)$ be the set of 1-periodic Hamiltonian orbits, i.e., 1-periodic orbits of $\phi_H^t$; generically, each $x\in\chi(H)$ is non-degenerate. 

Meanwhile, let $J$ be an $\omega$-compatible almost complex structure on $M$. We will consider almost complex structures which are of contact type at infinity, i.e., along the cylindrical end, $J$ takes $\partial_r$ to the Reeb field of $\widehat{\theta}\vert_{\partial\widehat{M}}$.

\subsubsection{The construction}
Now, let $H\equiv\{H_t\}_{t\in S^1}$ be a non-degenerate $S^1$-family of linear at infinity Hamiltonians, all of the same slope $\kappa$, and $J\equiv\{J_t\}_{t\in S^1}$ an $S^1$-family of $\omega$-compatible almost complex structures which are of contact type at infinity; we call such a pair $(H,J)$ \emph{admissible}.

Given $x\in\chi(H)$, we may define its degree:
    \begin{equation}
    \deg(x)\equiv n-\operatorname{CZ}(x),
    \end{equation}
where $\operatorname{CZ}(x)$ is the Conley-Zehnder index of the family of symplectic matrices $\Psi_x$ associated to the linearization of the Hamiltonian flow at $x$ using our fixed trivialization of $x^*\Lambda^n_\bbC TM$. (In particular, there is a homotopically unique trivialization of $x^*\Lambda^n_\bbC TM$ such that $x^*\Lambda$ is Maslov 0 resp. 1 if it is orientable resp. non-orientable; this induces a homotopically unique trivialization of $x^*TM$.) Observe, $\deg(x)$ is the index of a Fredholm operator
    \begin{equation}
    D_{\Psi_x}:W^{k,p}(\bbC;\bbC^n)\to W^{k-1,p}(\bbC;\bbC^n),
    \end{equation}
where $\bbC$ has a negative cylindrical end and $D_{\Psi_x}$ has the asymptotic $\Psi_x$ along this negative cylindrical end.

We define $\Theta\equiv\bbR_s\times S^1_t$. Given any two $x,y\in\chi(H)$, we consider the moduli space $\widetilde{\scrF}_{yx}$ of Floer trajectories connecting $x$ to $y$, i.e., cylinders $u:\Theta\to M$ satisfying
    \begin{equation}
    \begin{cases}
    \partial_su+J_t\big(\partial_t-X_H(u)\big)=0 \\
    \lim_{s\to-\infty}u(s,t)=x(t) \\
    \lim_{s\to+\infty}u(s,t)=y(t).
    \end{cases}
    \end{equation}
This is, for generic data, a smooth manifold of dimension $\deg(x)-\deg(y)$ whose tangent bundle is classified by the index bundle of the family of surjective Fredholm operators 
    \begin{equation}
    D\overline{\partial}_{H,J}\equiv\big\{D(\overline{\partial}_{H,J})_u:W^{k,p}(\Theta;u^*TM)\to W^{k-1,p}(\Theta;u^*TM)\big\},\;\; kp>2
    \end{equation}
given by linearizing. When $x\neq y$, there is a free proper $\bbR$-action on $\widetilde{\scrF}_{yx}$ given by time-shift in the $s$-coordinate; we denote by $\scrF_{yx}$ the $\bbR$-quotient. We denote by $\bbF_{yx}$ the Gromov-compactification of $\scrF_{yx}$ given by allowing breakings at Hamiltonian orbits. Let $\frako_x$ be the orientation line associated to $x$, i.e., the determinant line of the invertible self-adjoint operator $D(\overline{\partial}_{H,J})_x$ given by linearizing at $x$. Again, $D(\overline{\partial}_{H,J})_x$ is the asymptotic operator of $D(\overline{\partial}_{H,J})_u$ at $-\infty$ for any Floer trajectory $u$ starting at $x$. The $(H,J)$-Floer cochain complex 
    \begin{equation}
    CF^j(M;H,J;\bbZ)\equiv\bigoplus_{\deg(x)=j}\abs{\frako_x}
    \end{equation}
has the codifferential whose $\abs{\frako_y}-\abs{\frako_x}$ component is given by 
    \begin{equation}
    \delta_{y,x}\equiv\sum_{\substack{u\in\bbF_{yx} \\ \deg(x)=\deg(y)+1}}\mu_u,
    \end{equation}
where $\mu_u:\frako_y\xrightarrow{\sim}\frako_x$ is the isomorphism induced on orientation lines by $u$; we denote by $HF^*(M;H,J;\bbZ)$ the $(H,J)$-Floer cohomology.

Now, consider two pairs of regular Floer data, $(H,J)$ and $(H',J')$, satisfying $\kappa<\kappa'$. We choose monotonic Floer continuation data $(H_s,J_s)$ connecting $(H',J')$ to $(H,J)$, i.e., 
\begin{itemize}
\item a non-degenerate $\bbR$-dependent family $\{H_s\}$ of $S^1$-families of linear at infinity Hamiltonians of slopes $\kappa_s$ which agrees with $(H',\kappa')$ resp. $(H,\kappa)$ for $s\ll0$ resp. $s\gg0$;
\item and an $\bbR$-dependent family $\{J_s\}$ of $S^1$-families of $\omega$-compatible almost complex structures of contact type which agrees with $J'$ resp. $J$ for $s\ll0$ resp. $s\gg0$.
\end{itemize}
Given any two $x\in\chi(H)$ and $y'\in\chi(H')$, we consider the moduli space $\frakc_{\kappa,\kappa'}(x,y')$ of Floer continuation cylinders connecting $x$ to $y'$, i.e., cylinders $u:\Theta\to M$ satisfying 
    \begin{equation}
    \begin{cases}
    \partial_tu+J_{s,t}\big(\partial_t-X_{H_s}(u)\big)=0 \\
    \lim_{s\to-\infty}u(s,t)=y'(t) \\
    \lim_{s\to+\infty}u(s,t)=x(t).
    \end{cases}
    \end{equation}
Again, this is, for generic data, a smooth manifold of dimension $\deg(y')-\deg(x)$. We denote by $\overline{\frakc}_{\kappa,\kappa'}(x,y')$ the Gromov-compactification given by allowing breakings at Hamiltonian orbits. We may define a cochain map 
    \begin{equation}
    \overline{\frakc}_{\kappa,\kappa'}:CF^*(M;H,J;\bbZ)\to CF^*(M;H',J';\bbZ)
    \end{equation}
whose $\abs{\frako_x}$ -- $\abs{\frako_{y'}}$ component is given by
    \begin{equation}
    (\overline{\frakc}_{\kappa,\kappa'})_{x,y'}\equiv\sum_{\substack{u\in\overline{\frakc}_{\kappa,\kappa'}(x,y') \\ \deg(x)=\deg(y')}}\mu_u,
    \end{equation}
where $\mu_u:\frako_x\xrightarrow{\sim}\frako_{y'}$ is the isomorphism induced on orientation lines by $u$.

We define symplectic cohomology via
    \begin{equation}
    SH^*(M;\bbZ)\equiv\varinjlim_{\ell}HF^*(M;H^\ell,J^\ell;\bbZ),
    \end{equation}
where the directed system is over a family of regular Floer data $\big\{(H^\ell,J^\ell)\big\}_{\ell\in\bbZ_{\geq0}}$ satisfying 
    \begin{equation}
    \kappa_\ell<\kappa_{\ell'},\;\;\forall\ell<\ell'.
    \end{equation}
It is known $SH^*(M;\bbZ)$ is independent of all auxiliary choices in the construction.

\subsection{Moduli spaces for domains}
The next two subsections follow \cite[Section 4.2]{Sie21}. However, \emph{loc. cit.} builds bulk-deformed symplectic cohomology using relative cycles of $(\widehat{M},\partial\widehat{M})$, whereas the present article uses Morse cycles of $(\widehat{M},\partial\widehat{M})$. A similar Morse-theoretic treatment of bulk-deformations in the context of relative Fukaya categories appears also in \cite{Sheridan}. The difference is essentially cosmetic. 

In order to actually construct bulk-deformed symplectic cohomology, we must first spend some time constructing the relevant moduli spaces of domains, as follows.

\subsubsection{For differentials}
Let $q\in\bbZ$, $q\geq1$, and consider the moduli space $\calC_{q+2}$ of genus 0 Riemann surfaces with $q+2$ ordered marked points, where the first marked point has an asymptotic marker, modulo biholomorphisms. We declare the first marked point to be negative and call it the \emph{output}; we declare the second marked point to be positive and call it the \emph{input}; and we declare the remaining $q$ marked points to be positive. We may think of $\calC_{q+2}$ as
    \begin{equation}
    \calC_{q+2}\equiv\operatorname{Conf}_q(\bbR\times S^1)/\bbR,
    \end{equation}
where the first marked point corresponds to $-\infty$, the second marked point corresponds to $+\infty$, and the asymptotic marker corresponds to $1\in S^1$. Let $\calC^{univ}_{q+2}\to\calC_{q+2}$ be the universal family: 
    \begin{equation}
    \calC^{univ}_{q+2}\equiv\operatorname{Conf}_q(\bbR\times S^1)\times_\bbR(\bbR\times S^1),
    \end{equation}
where $\bbR$ acts diagonally. Finally, we observe that a choice of asymptotic marker at the first marked point induces a choice of asymptotic marker at the remaining marked points.

Meanwhile, consider the moduli space $T_{\calC_{q+2}}$ of stable trees with $q+2$ ordered external edges.\footnote{Recall, a tree is stable if each vertex has valency at least 3.} We call the first external edge the \emph{output}; we call the second external edge the \emph{input}; we call any internal edge lying on the unique path between the output vertex and the input vertex a \emph{plain} edge; and we call all internal edges which are not plain \emph{round}. We orient all internal edges to point away from the output. We denote by $E_{pl}(T)$ the set of all plain edges of $T$, and $E_{rd}(T)$ the set of all round edges of $T$. Similarly, we denote by $V_{pl}(T)$ the set of all internal vertices contained in the subtree determined by the plain edges, and $V_{rd}(T)$ the set of all internal vertices which are not plain. We adopt the usual shorthand and define 
    \begin{equation}
    E_{pl}(T_\calC)\equiv\coprod_{q\geq1}\coprod_{T\in\calC_{q+2}}E_{pl}(T);
    \end{equation}
we analogously define 
    \begin{equation}
    E_{rd}(T_\calC),\;\;V_{pl}(T_\calC),\;\;\textrm{and}\;\;V_{rd}(T_\calC).
    \end{equation}

Given $v\in V_{pl}(T_\calC)$, let $E(v)$ be the set of all internal edges adjacent to $v$; we consider the moduli space $\calC_{E(v)}$ defined in the same way as $\calC_{\operatorname{val}(v)}$, except instead of ordering the marked points we do the following: index the output by the incoming plain edge at $v$, index the input by the outgoing plain edge at $v$, and index the remaining marked points by the remaining (necessarily round) edges at $v$.

Also, given $v\in V_{rd}(T_\calC)$, we will consider the moduli space $\calN_{E(v)}$ of genus 0 Riemann surfaces with $\operatorname{val}(v)$ marked points, where the marked points are indexed by $E(v)$ (the marked point indexed by the incoming edge is declared negative and the marked point index by the outgoing edge is declared positive), modulo biholomorphisms. We may think of $\calN_{E(v)}$ as 
    \begin{equation}
    \calN_{E(v)}\equiv\operatorname{Conf}_{E(v)}(\bbC P^1)/PSL(2,\bbC),
    \end{equation}
where the notation $\operatorname{Conf}_{E(v)}$ indicates we index the marked points by $E(v)$. Let $\calN_{E(v)}^{univ}\to\calN_{E(v)}$ be the universal family:
    \begin{equation}
    \calN^{univ}_{E(v)}=\operatorname{Conf}_{E(v)}(\bbC P^1)\times_{PSL(2,\bbC)}\bbC P^1,
    \end{equation}
where $PSL(2,\bbC)$ acts diagonally.

For any $v\in V_{rd}(T_\calC)$, we endow the universal family $\calN_{E(v)}^{univ}\to\calN_{E(v)}$ with fiberwise cylindrical ends, i.e., fiberwise disjoint fiberwise holomorphic embeddings 
    \begin{equation}
    \epsilon^\calN_e:\calN_{E(v)}\times\bbR_{\pm}\times S^1\to\calN^{univ}_{E(v)},\;\;e\in E(v),
    \end{equation}
where we use $\bbR_-$ resp. $\bbR_+$ if $e$ is an incoming resp. outgoing edge and the map 
    \begin{equation}
    t\mapsto\lim_{s\to\pm\infty}\epsilon^\calN_e(s,t)
    \end{equation}
is the constant map at the marked point indexed by $e$. 

Similarly, for any $v\in V_{pl}(T_\calC)$, we endow the universal family $\calC_{E(v)}^{univ}\to\calC_{E(v)}$ with fiberwise cylindrical ends:
    \begin{equation}
    \epsilon^\calC_e:\calC_{E(v)}\times\bbR_{\pm}\times S^1\to\calC^{univ}_{E(v)},\;\;e\in E(v),
    \end{equation}
where we use $\bbR_-$ resp. $\bbR_+$ if $e$ is an incoming resp. outgoing edge and the map 
    \begin{equation}
    t\mapsto\lim_{s\to\pm\infty}\epsilon^\calC_e(s,t)
    \end{equation}
is the constant map at the marked point indexed by $e$.

Now, for any $T\in T_{\calC_{q+2}}$ and $\epsilon\in\bbR_{>0}$ sufficiently small, we define 
    \begin{align}
    \overline{\calC}_T&\equiv\Bigg(\prod_{v\in V_{pl}(T)}\calC_{E(v)}\Bigg)\times\Bigg(\prod_{v\in V_{rd}(T)}\calN_{E(v)}\Bigg), \\
    \overline{\calC}^\epsilon_T&\equiv\overline{\calC}_T\times(-\epsilon,0]^{E_{pl}(T)}\times\bbD^{E_{rd}(T)}_\epsilon,
    \end{align}
where $\bbD_\epsilon\subset\bbC$ is the 2-disk of radius $\epsilon$. Observe, our fiberwise cylindrical ends induce gluing maps: 
    \begin{equation}
    \phi_{T,e}:\big\{T'\in\overline{\calC}^\epsilon_T:\rho_e\neq0\big\}\to\overline{\calC}^\epsilon_{T/e},
    \end{equation}
where $\rho_e$ denotes the coordinate in $(-\epsilon,0]^{E_{pl}(T)}$ resp. $\bbD^{E_{rd}(T)}_\epsilon$ if $e$ is plain resp. round and $T/e$ is the tree obtained by collapsing $e$. (When $e$ is plain, we glue by aligning the asymptotic markers; when $e$ is round, the $S^1$-factor corresponds to the angle between the asymptotic markers.) Finally, we define the topological space
    \begin{equation}
    \overline{\calC}_{q+2}\equiv\Bigg(\prod_{T\in\calC_{q+2}}\overline{\calC}^\epsilon_T\Bigg)/\sim,
    \end{equation}
where the equivalence relation is given by $T'\sim\phi_{T,e}(T')$. 

\begin{rem}
Throughout the remainder of the present article, we will assume that all of our choices of fiberwise cylindrical ends are consistent in the usual sense.
\end{rem}

\subsubsection{For continuation maps}
In order to define continuation maps, we need a mechanism to ``break the translational symmetry''; this comes in the form of ``sprinkles''. This notion, in the open-string case, first appeared in \cite[Section 2d]{AS10a}; the following is the closed-string analogue. Let $S\in\calC_{q+2}$ be equipped with its cylindrical ends $\epsilon^\calC_0,\ldots,\epsilon^\calC_{q+1}$, induced by the universal family, and consider any biholomorphism $\Psi:S\xrightarrow{\sim}\bbR\times S^1$ sending the first marked point to $-\infty$, the second marked point to $+\infty$, and the asymptotic marker to $1\in S^1$.

\begin{defin}
A \emph{popsicle stick} for $S$ is a line $L\subset S$, $L\cong\bbR$, such that:%https://en.wikipedia.org/wiki/Fab_(brand)#/media/File:FAB_frozen_dessert.jpg
\begin{itemize}
\item $\big(\epsilon^\calC_0\big)^{-1}(L)$ agrees with $\{1\}\times\bbR_-$ near $s=-\infty$,
\item $\big(\epsilon^\calC_1\big)^{-1}(L)$ agrees with $\{1\}\times\bbR_+$ near $s=+\infty$,
\item and $\Psi(L)$ is of the form $\Big\{\big(s,\beta(s)\big)\Big\}$, where $\beta:\bbR\to S^1$ is some smooth map whose image is contained in a sufficiently small neighborhood of $1\in S^1$.
\end{itemize}
\end{defin}

\begin{defin}
Given a popsicle stick $L\subset S$, a \emph{sprinkle} for $S$ is simply a choice of point on $L$.
\end{defin}

For any $v\in V_{pl}(T_\calC)$, we endow the universal family $\calC^{univ}_{E(v)}\to\calC_{E(v)}$ with fiberwise popsicle sticks which are consistent with respect to gluing. We now consider the moduli space $\calD_{q+2}$, $q\geq1$, defined as
    \begin{equation}
    \big\{(S,p):S\in\calC_{q+2},p\in L_S\big\},
    \end{equation}
where $L_S$ is the popsicle stick for $S$ induced by the universal family. In this way, we have a map
    \begin{equation}
    \calD_{q+2}\to\calC_{q+1},\;\;q\geq2,
    \end{equation}
given by forgetting the sprinkle, whose fiber is $L_S$. In particular, we see $\calD_{q+2}$ ``breaks the translational symmetry'' since it increases the dimension of $\calC_{q+1}$ by one. We define $\calD_2$ by endowing $\bbR\times S^1$ with the standard popsicle stick $\bbR\times\{1\}$.

In order to define $\overline{\calD}_{q+2}$, we need to describe popsicle sticks for $\calC_{E(v)}$; this in turn requires us to choose our fiberwise cylindrical ends a bit more carefully.\footnote{One could say our fiberwise cylindrical ends need to be chosen to be \emph{delicious}.} Consider the moduli space $T_{\calD_{q+2}}$ defined in the same way as $T_{\calC_{q+2}}$, except that exactly one plain vertex $v_{spr}$ is declared the \emph{sprinkle vertex} and is allowed to be valency two. Similarly to before, we define 
    \begin{equation}
    E_{pl}(T_\calD),\;\;E_{rd}(T_\calD),\;\;V_{pl}(T_\calD),\;\;\textrm{and}\;\;V_{rd}(T_\calD).
    \end{equation}
\begin{enumerate}
\item For $v\in V_{pl}(T_\calD)$ with $\operatorname{val}(v)\geq3$, we define $\calC_{E(v)}$ as before and equip the universal family $\calC^{univ}_{E(v)}\to\calC_{E(v)}$ with fiberwise cylindrical ends and popsicle sticks as follows: 
    \begin{itemize}
    \item if $v$ is a vertex of $T\in T_{\calD_{q+2}}$ such that $T$'s sprinkle vertex is valency at least 3, then we forget the sprinkle vertex in order to view $T$ as an element of $T_{\calC_{q+2}}$ and take the induced fiberwise cylindrical ends and popsicle sticks from the universal family $\calC^{univ}_{E(v)}\to\calC_{E(v)}$, where $v$ is viewed as an element of $V_{pl}(T_\calC)$;

    \item otherwise, if $T$'s sprinkle vertex is valency 2 (in particular, distinct from $v$ itself), we contract the incoming edge at the sprinkle in order to view $T$ as an element of $T_{\calC_{q+2}}$ and take the induced fiberwise cylindrical ends and popsicle sticks from the universal family $\calC^{univ}_{E(v)}\to\calC_{E(v)}$, where $v$ is viewed as an element of $V_{pl}(T_\calC)$.
    \end{itemize}
\item We perform the analogous procedure for $v\in V_{rd}(T_\calD)$ to induce fiberwise cylindrical ends for the universal family $\calN^{univ}_{E(v)}\to\calN_{E(v)}$. 
\item For $v_{spr}\in V_{pl}(T_\calD)$ with $\operatorname{val}(v_{spr})\geq3$, we consider the moduli space $\calD_{E(v_{spr})}$ defined in the same way as $\calD_{\operatorname{val}(v_{spr})}$, except instead of ordering the marked points we index by $E(v_{spr})$. We endow the universal family $\calD^{univ}_{E(v_{spr})}\to\calD_{E(v_{spr})}$ with fiberwise cylindrical ends and popsicle sticks induced from the universal family $\calC^{univ}_{E(v_{spr})}\to\calC_{E(v_{spr})}$. 
\item Finally, if $\operatorname{val}(v_{spr})=2$, we define $\calD_{E(v_{spr})}$ by endowing $\bbR\times S^1$ with the standard popsicle stick $\bbR\times\{1\}$.
\end{enumerate}

Now, for any $T\in T_{\calD_{q+2}}$ and $\epsilon\in\bbR_{>0}$ sufficiently small, we define 
    \begin{align}
    \overline{\calD}_T&\equiv\calD_{E(v_{spr})}\times\Bigg(\prod_{v\in V_{pl}(T),v\neq v_{spr}}\calC_{E(v)}\Bigg)\times\Bigg(\prod_{v\in V_{rd}(T)}\calN_{E(v)}\Bigg), \\
    \overline{\calD}^\epsilon_T&\equiv\overline{\calD}_T\times(-\epsilon,0]^{E_{pl}(T)}\times\bbD^{E_{rd}(T)}_\epsilon.
    \end{align}
Observe, our fiberwise cylindrical ends induce gluing maps,
    \begin{equation}
    \phi_{T,e}:\big\{T'\in\overline{\calD}^\epsilon_T:\rho_e\neq0\big\}\to\overline{\calD}^\epsilon_{T/e},
    \end{equation}
since they play nicely with popsicle sticks. Finally, we define the topological space
    \begin{equation}
    \overline{\calD}_{q+2}\equiv\Bigg(\prod_{T\in\calD_{q+2}}\overline{\calD}^\epsilon_T\Bigg)/\sim,
    \end{equation}
where the equivalence relation is given by $T'\sim\phi_{T,e}(T')$. 

\begin{rem}
Throughout the remainder of the present article, we will assume that all of our choices of fiberwise cylindrical ends and popsicle sticks are consistent in the usual sense.
\end{rem}

\subsubsection{Consistent universal Floer data}
We will now discuss domain-dependent perturbations. Once and for all, we fix a reference $\omega$-compatible almost complex structure $J_{\rm ref}$ of contact type at infinity. Consider an admissible pair $(H_0,J_0)$.

\begin{defin}
Let $S\in\calC_{q+2}$ be equipped with its cylindrical ends $\epsilon^\calC_-$ resp. $\epsilon^\calC_+$ at the output resp. input induced from the universal family. A \emph{perturbation datum} for $S$ consists of a pair $(K,J)$, where: $K=H\otimes\gamma$ for an $S$-dependent Hamiltonian of slope $\widetilde{\kappa}:S\to\bbR_+$, $\gamma$ is a closed 1-form such that $d(\widetilde{\kappa}\gamma)\leq0$, and $J$ is an $\omega$-compatible almost complex structure which agrees with $J_{\rm ref}$ at infinity. We require $(K,J)$ to satisfy the following two conditions. 
\begin{enumerate}
\item $\big(\epsilon^\calC_\pm\big)^*K=H_0\otimes dt$ and $\big(\epsilon^\calC_\pm\big)^*J=J_0$.
\item On each disk-like neighborhood of a marked point of $S$ which is not the output or input (i.e., the image of a ``sub-cylinder'' of the corresponding cylindrical end), we have $K=0$ and $J=J_{\rm ref}$.
\end{enumerate}
\end{defin}

Given an admissible pair $(H_0,J_0)$, we choose fiberwise perturbation data for the universal family 
    \begin{equation}
    \calC^{univ}_{E(v)}\to\calC_{E(v)}\;\;\textrm{resp.}\;\;\calN^{univ}_{E(v)}\to\calN_{E(v)},
    \end{equation}
where $v\in V_{pl}(T_\calC)$ resp. $v\in V_{rd}(T_\calC)$. 

\begin{rem}
Throughout the remainder of the present article, we will assume that all of our choices of fiberwise perturbation data are consistent in the usual sense.
\end{rem}

Similarly, consider admissible $(H_-,J_-)$ and $(H_+,J_+)$, where $\kappa_-\geq\kappa_+$.

\begin{defin}
Let $S\in\calD_{q+2}$ be equipped with its cylindrical ends $\epsilon^\calC_-$ resp. $\epsilon^\calC_+$ at the output resp. input induced from the universal family. A \emph{perturbation datum} for $S$ consists of a pair $(K,J)$, where: $K=H\otimes\gamma$ for an $S$-dependent Hamiltonian of slope $\widetilde{\kappa}:S\to\bbR_+$, $\gamma$ is a closed 1-form such that $d(\widetilde{\kappa}\gamma)\leq0$, and $J$ is an $\omega$-compatible almost complex structure which agrees with $J_{\rm ref}$ at infinity. We require $(K,J)$ to satisfy the following two conditions. 
\begin{enumerate}
\item $\big(\epsilon^\calC_\pm\big)^*K=H_\pm\otimes dt$ and $\big(\epsilon^\calC_\pm\big)^*J=J_\pm$.
\item On each disk-like neighborhood of a marked point of $S$ which is not the output or input (i.e., the image of a ``sub-cylinder'' of the corresponding cylindrical end), we have $K=0$ and $J=J_{\rm ref}$.
\end{enumerate}
In addition, we require the following two conditions.
\begin{enumerate}
\item $\big(\epsilon^\calC_\pm\big)^*H=H_\pm$ and $\big(\epsilon^\calC_\pm\big)^*\gamma=dt$.
\item $\gamma=\eta/\widetilde{\kappa}$ for a closed 1-form $\eta$ such that $\big(\epsilon^\calC_\pm\big)^*\eta=\kappa_\pm dt$ and $d\eta\leq0$.
\end{enumerate}
\end{defin}

\begin{rem}
In fact, such an $\eta$ exists since $\kappa_-\geq\kappa_+$.
\end{rem}

Now, we are already given perturbation data relative to both $(H_-,J_-)$ and $(H_+,J_+)$. We choose fiberwise perturbation data for the universal family 
    \begin{equation}
    \calC^{univ}_{E(v)}\to\calC_{E(v)}\;\;\textrm{resp.}\;\;\calN^{univ}_{E(v)}\to\calN_{E(v)},
    \end{equation}
where $v\in V_{pl}(T_\calD)$ is a non-sprinkle vertex resp. $v\in V_{rd}(T_\calD)$, by inducing from the choice for $(H_-,J_-)$ resp. $(H_+,J_+)$ if $v$ comes before resp. after the sprinkle vertex. We also choose fiberwise perturbation data for the universal family for each sprinkle vertex:
    \begin{equation}
    \calD^{univ}_{E(v_{spr})}\to\calD_{E(v_{spr})}.
    \end{equation}

\begin{rem}
Throughout the remainder of the present article, we will assume that all of our choices of fiberwise perturbation data are consistent in the usual sense.
\end{rem}

\subsection{Bulk-deformed symplectic cohomology}\label{subsec:bulkdeformed}
We are now ready to begin the construction of bulk-deformed symplectic cohomology. Given $\mu \geq 1$, let
    \begin{equation}
    \frakU_\rho\equiv\sum_{1\leq j_\rho\leq k_\rho}\frakU_{\rho,j_\rho}a_{\rho,j_\rho}\in CM_{2n-\ell_\rho}(M;\bbC),\;\;\frakU_{\rho,j_\rho}\in\bbC,a_{\rho,j_\rho}\in\crit(f),1\leq\rho\leq\mu
    \end{equation}
be a sequence of Morse cycles of degree $2n-\ell_\rho$, $\ell_\rho\geq3$. Of course, if we consider the Poincar\'e dual Morse cochain to $\frakU_\rho$ given by sending $f$ to $-f$, still denoted
    \begin{equation}
    \frakU_\rho\in CM^{\ell_\rho}(M;\bbC),
    \end{equation}
this is a Morse cocycle.

First, let $S$ be a Riemann sphere with $q\geq3$ ordered marked points equipped with perturbation data $(K,J)$ induced from the universal family. We define $Y$ to be the $\omega$-dual of $K$. We denote by $\scrM_S$ the moduli space of pseudoholomorphic spheres with domain $S$, i.e., spheres $u:S\to\widehat{M}$ satisfying
    \begin{equation}
    (du-Y)^{0,1}=0.
    \end{equation}
We define 
    \begin{equation}
    \scrM_q\equiv\Big\{(S,u):S\in\calN_q,u\in\scrM_{\calN^{univ}_q\vert_S}\Big\},\;\;q\geq3.
    \end{equation}
Of course, these spaces are not particularly interesting since $(M,\omega)$ is exact; that so, we still require them. Analogously, we may define 
    \begin{equation}
    \scrM_{E(v)},\;\;v\in V_{rd}(T_\calC),V_{rd}(T_\calD).
    \end{equation}

Now, consider an admissible pair $(H_0,J_0)$. Let $S\in\calC_{q+2}$ be equipped with its cylindrical ends $\epsilon^\calC_\pm$ and perturbation data $(K,J)$ induced from the universal family. For any two $x,y\in\chi(H_0)$, we denote by $\scrF_S(y,x)$ the moduli space of maps $u:S\to\widehat{M}$ satisfying 
    \begin{equation}
    \begin{cases}
    (du-Y)^{0,1}=0, \\
    \lim_{s\to-\infty}\big(u\circ\epsilon^\calC_-\big)(s,t)=x(t), \\
    \lim_{s\to+\infty}\big(u\circ\epsilon^\calC_+\big)(s,t)=y(t).
    \end{cases}
    \end{equation}
We define 
    \begin{equation}
    \scrF_q(y,x)\equiv\Big\{(S,u):S\in\calC_{q+2},u\in\scrF_{\calC^{univ}_{q+2}\vert_S}(y,x)\Big\},\;\;q\geq1.
    \end{equation}
We set 
    \begin{equation}
    \scrF_0(y,x)\equiv\scrF_{yx}.
    \end{equation}
Also, we may analogously define
    \begin{equation}
    \scrF_{E(v)}(y,x),\;\;v\in V_{pl}(T_\calC).
    \end{equation}

Consider the moduli space $T^{semi}_{\calC_{q+2}}$ defined in the same way as $T_{\calC_{q+2}}$, except that we now allow plain vertices of valency 2. 

\begin{defin}
Given any two $x,y\in\chi(H_0)$, a \emph{stable broken Floer cylinder} connecting $x$ to $y$ with $q$ interior marked points consists of the following data:
\begin{enumerate}
\item $T\in T^{semi}_{\calC_{q+2}}$;
\item $z_e\in\chi(H_0)$ for each $e\in E_{pl}(T)$, such that $z_e=x$ resp. $z_e=y$ if $e$ is the output resp. input edge;
\item $u_v\in\scrF_{E(v)}(z_{e_+},z_{e_-})$ for each $v\in V_{pl}(T)$, where $e_-$ resp. $e_+$ is the incoming resp. outgoing plain edge at $v$;
\item $u_v\in\scrM_{E(v)}$ for each $v\in V_{rd}(T)$;
\item for each $e\in E_{rd}(T)$, say with endpoints corresponding to the marked points $p$ and $p'$ of $u_{e_I(v)}$ and $u_{e_T(v)}$, respectively, where $e_I(v)$ resp. $e_T(v)$ is the intitial resp. terminal vertex of $e$, we have 
    \begin{equation}
    u_{e_I(v)}(p)=u_{e_T(v)}(p')
    \end{equation}
(i.e., bubbles connect).
\end{enumerate}
\end{defin}

We denote by $\bbF_T(y,x)$ the moduli space of stable broken Floer cylinders connecting $x$ to $y$ with $q$ interior marked points which are modeled on $T\in T^{semi}_{\calC_{q+2}}$. By evaluating at the marked points indexed by the round leaves, we have a map 
    \begin{equation}
    \eval_q:\bbF_T(y,x)\to \widehat{M}^q.
    \end{equation}
    
Let $\frakU\equiv(\frakU_1,\ldots,\frakU_\mu)$ be our $\mu$-tuple of Morse cycles from before and consider a partition 
    \begin{align}
    \{1,\ldots,q\}&=\{1,\ldots,q_1\}\amalg\{q_1+1,\ldots,q_2\}\amalg\cdots\amalg\{q_{\mu-1}+1,\ldots,q_\mu\} \\
    &\equiv\frakq_1\amalg\cdots\amalg\frakq_\mu;
    \end{align}
and for each $\rho$, choose 
    \begin{equation}
    j_{\rho_1},\ldots,j_{\rho_{\abs{\frakq_\rho}}}
    \end{equation}
between 1 and $k_\rho$. (We would like to emphasize that we do not require a relation between $j_{\rho_1},\ldots,j_{\rho_{\abs{\frakq_\rho}}}$.) We may define
    \begin{equation}\label{eqn:bulkdeformedaux}
    \bbF_{\frakq_1,\ldots,\frakq_\mu,\frakU}(y,x)\equiv\coprod_{a_{\frakq_1},\ldots,a_{\frakq_\mu}}\bbF_{a_{\frakq_1},\ldots,a_{\frakq_\mu}}(y,x),
    \end{equation}
where $a_{\frakq_\rho}\equiv(a_{\rho,j_{\rho_1}},\ldots,a_{\rho,j_{\abs{\frakq_\rho}}})$ is a $\abs{\frakq_\rho}$-tuple of terms appearing in $\frakU_\rho$ and 
    \begin{multline}
    \bbF_{a_{\frakq_1},\ldots,a_{\frakq_\mu}}(y,x)\equiv\coprod_{T\in T^{semi}_{\calC_{q+2}}}\bbF_T(y,x)\times_{\widehat{M}^q} \\
    \big(W^s(a_{1,j_1})\times\cdots\times W^s(a_{1,j_{\abs{\frakq_1}}})\times\cdots\times W^s(a_{\mu,j_1})\times\cdots\times W^s(a_{\mu,j_{\abs{\frakq_\mu}}})\big)
    \end{multline}
equipped with the Gromov topology. We use the shorthand
    \begin{multline}
    \bbF_{T,a_{\frakq_1},\ldots,a_{\frakq_\mu}}(y,x)\equiv\bbF_T(y,x)\times_{\widehat{M}^q} \\
    \big(W^s(a_{1,j_1})\times\cdots\times W^s(a_{1,j_{\abs{\frakq_1}}})\times\cdots\times W^s(a_{\mu,j_1})\times\cdots\times W^s(a_{\mu,j_{\abs{\frakq_\mu}}})\big).
    \end{multline}

\begin{rem}
Perhaps some discussion is warranted for exactly what elements of $\bbF_{a_{\frakq_1},\ldots,a_{\frakq_\mu}}(y,x)$ look like. First, we have a stable broken Floer cylinder connecting $x$ to $y$; this comes equipped with $q$ marked points. Since we are bulk-deforming by $(\frakU_1,\ldots,\frakU_\mu)$, we partition these $q$ marked points into ``clusters'' $\frakq_1,\ldots,\frakq_\mu$; the marked points in the cluster $\frakq_\rho$ are the ones which are allowed to interact with $\frakU_\rho$. Finally, since $\frakU_\rho$ is a Morse cycle, we must further specify which $a_{\rho,j_\rho}$ in the sum $\frakU_\rho=\sum\frakU_{\rho,j_\rho}a_{\rho,j_\rho}$ the marked points in the cluster $\frakq_\rho$ are allowed to interact with; this is the role of $j_{\rho_1},\ldots,j_{\rho_{\abs{\frakq_\rho}}}$.
\end{rem}

The basic Gromov-compactness and gluing result in this context is the following.

\begin{prop}[{\cite[Proposition 4.15]{Sie21}}]\label{prop:bulkdeformedcodiff}
We have the following two results.
\begin{enumerate}
\item When $\deg(x)=\deg(y)+\sum_{\rho=1}^\mu\abs{\frakq_\rho}(\ell_\rho-2)+1$, $\bbF_{\frakq_1,\ldots,\frakq_\mu,\frakU}(y,x)$ is a finite set of points.
\item  When $\deg(x)=\deg(y)+\sum_{\rho=1}^\mu\abs{\frakq_\rho}(\ell_\rho-2)+2$, $\bbF_{a_{\frakq_1},\ldots,a_{\frakq_\mu}}(y,x)$ is a compact 1-manifold with boundary given by
    \begin{align}
    \partial\bbF_{a_{\frakq_1},\ldots,a_{\frakq_\mu}}(y,x)=&\coprod_T\bbF_{T,a_{\frakq_1},\ldots,a_{\frakq_\mu}}(y,x) \\
    &\coprod_{T,\rho,\sigma}\coprod_{\substack{b\in\crit(f) \\ \deg(b)=\deg(a_{\rho,j_\sigma})+1}}\bbF_{T,a_{\frakq_1},\ldots,a_{\frakq_{\rho-1}},b_{\frakq_\rho,\sigma},\ldots,a_{\frakq_\mu}}(y,x)\times\bbM_{ba_{\rho,j_\sigma}}, \nonumber
    \end{align}
where: the first disjoint union only uses $T\in T_{\calC_{q+2}}$ having exactly one plain edge and no round edges, the second disjoint union only uses the unique $T\in T_{\calC_{q+2}}$ having exactly no plain edges and no round edges, and 
    \begin{equation}
    b_{\frakq_\rho,\sigma}\equiv(a_{\rho,j_1},\ldots,a_{\rho,j_{\sigma-1}},b,\widehat{a}_{\rho,j_\sigma},\ldots,a_{\rho,j_{\abs{\frakq_\rho}}}).
    \end{equation}
In particular, $\bbF_{\frakq_1,\ldots,\frakq_\mu,\frakU}(y,x)$ is a compact 1-manifold with boundary given by
    \begin{equation}
    \partial\bbF_{\frakq_1,\ldots,\frakq_\mu,\frakU}(y,x)\equiv\coprod_{a_{\frakq_1},\ldots,a_{\frakq_\mu}}\partial\bbF_{a_{\frakq_1},\ldots,a_{\frakq_\mu}}(y,x).
    \end{equation}
\end{enumerate}
\end{prop}

\begin{rem}
When $q=0$, the Gromov-compactification of $\scrF_0(y,x)=\scrF_{yx}$ is simply $\bbF_{yx}$.
\end{rem}

Meanwhile, consider the moduli space $T^{semi}_{\calD_{q+2}}$ defined in the same way as $T_{\calD_{q+2}}$, except that we now allow plain vertices of valency 2; moreover, consider admissible $(H_-,J_-)$ and $(H_+,J_+)$, where $\kappa_-\geq\kappa_+$. In a straightforward manner, given any two $x_-\in\chi(H_-)$ and $y_+\in\chi(H_+)$, we may define the moduli space 
    \begin{equation}
    \overline{\frakc}_{\frakq_1,\ldots,\frakq_\mu,\frakU,\kappa_\pm}(x_-,y_+)
    \end{equation}
analogously to \eqref{eqn:bulkdeformedaux}. The basic Gromov-compactness and gluing result in this context is the following.

\begin{prop}[{\cite[Proposition 4.16]{Sie21}}]\label{prop:bulkdeformedcont}
We have the following two results.
\begin{enumerate}
\item When $\deg(x_-)=\deg(y_+)+\sum_{\rho=1}^\mu\abs{\frakq_\rho}(\ell_\rho-2)$, $\overline{\frakc}_{\frakq_1,\ldots,\frakq_\mu,\frakU,\kappa_\pm}(x_-,y_+)$ is a finite set of points.
\item  When $\deg(x_-)=\deg(y_+)+\sum_{\rho=1}^\mu\abs{\frakq_\rho}(\ell_\rho-2)+1$, $\overline{\frakc}_{a_{\frakq_1},\ldots,a_{\frakq_\mu},\kappa_\pm}(x_-,y_+)$ is a compact 1-manifold with boundary given by
    \begin{align}
    \partial\overline{\frakc}_{a_{\frakq_1},\ldots,a_{\frakq_\mu},\kappa_\pm}(x_-,y_+)=&\coprod_T\overline{\frakc}_{T,a_{\frakq_1},\ldots,a_{\frakq_\mu},\kappa_\pm}(x_-,y_+) \\
    &\coprod_{T,\rho,\sigma}\coprod_{\substack{b\in\crit(f) \\ \deg(b)=\deg(a_{\rho,j_\sigma})+1}}\overline{\frakc}_{T,a_{\frakq_1},\ldots,a_{\frakq_{\rho-1}},b_{\frakq_\rho,\sigma},\ldots,a_{\frakq_\mu},\kappa_\pm}(x_-,y_+)\times\bbM_{ba_{\rho,j_\sigma}}, \nonumber
    \end{align}
where: the first disjoint union only uses $T\in T_{\calD_{q+2}}$ having exactly one plain edge and no round edges, the second disjoint union only uses the unique $T\in T_{\calD_{q+2}}$ having exactly no plain edges and no round edges, and 
    \begin{equation}
    b_{\frakq_\rho,\sigma}\equiv(a_{\rho,j_1},\ldots,a_{\rho,j_{\sigma-1}},b,\widehat{a}_{\rho,j_\sigma},\ldots,a_{\rho,j_\rho}).
    \end{equation}
In particular, $\overline{\frakc}_{\frakq_1,\ldots,\frakq_\mu,\frakU,\kappa_\pm}(x_-,y_+)$ is a compact 1-manifold with boundary given by
    \begin{equation}
    \partial\overline{\frakc}_{\frakq_1,\ldots,\frakq_\mu,\frakU,\kappa_\pm}(x_-,y_+)\equiv\coprod_{a_{\frakq_1},\ldots,a_{\frakq_\mu}}\partial\overline{\frakc}_{a_{\frakq_1},\ldots,a_{\frakq_\mu},\kappa_\pm}(x_-,y_+).
    \end{equation}
\end{enumerate}
\end{prop}

We define $\bbL\equiv\bbC[\hbar_1,\ldots,\hbar_\mu]$, where $\hbar_\rho$ is a formal variable of degree $2-\ell_\rho$. The \emph{$\frakU$-bulk-deformed $(H,J)$-Floer cochain complex}
    \begin{equation}
    CF^j_{\frakU}(M;H,J;\bbL)\equiv\bigoplus_{\deg(x)=j}\bbL\abs{\frako_x}
    \end{equation}
has the codifferential whose $\bbL\abs{\frako_y}-\bbL\abs{\frako_x}$ component is given by
    \begin{equation}\label{eq: reihgeropiubgeopubrg}
    \delta_{\frakU;y,x}\equiv\sum_{q\geq0}\sum_{\frakq_1,\ldots,\frakq_\mu}\dfrac{1}{\abs{\frakq_1}!\cdots\abs{\frakq_\mu}!}\hbar_1^{\abs{\frakq_1}}\cdots\hbar_\mu^{\abs{\frakq_\mu}}\delta_{\frakq_1,\ldots,\frakq_\mu,\frakU;y,x},
    \end{equation}
where 
    \begin{multline}
    \delta_{\frakq_1,\ldots,\frakq_\mu,\frakU;y,x}\equiv \\ \sum_{\substack{a_{\frakq_1},\ldots,a_{\frakq_\mu} \\ u\in\bbF_{a_{\frakq_1},\ldots,a_{\frakq_\mu}}(y,x) \\ \deg(x)=\deg(y)+\sum_{\rho=1}^\mu\abs{\frakq_\rho}(\ell_\rho-2)+1}}(\frakU_{1,j_1}\cdots\frakU_{1,j_{\abs{\frakq_1}}}\cdots\frakU_{\mu,j_1}\cdots\frakU_{\mu,j_{\abs{\frakq_\mu}}})\mu_u,
    \end{multline}
where $\mu_u:\frako_y\xrightarrow{\sim}\frako_x$ is the isomorphism induced on orientation lines by $u$. Two things to observe are the following: 
\begin{enumerate}
\item since $\chi(H)$ is finite, $\delta_{\frakq_1,\ldots,\frakq_\mu,\frakU;y,x}$ vanishes for $q\gg0$ and so (\ref{eq: reihgeropiubgeopubrg}) is a finite sum;
\item and Proposition \ref{prop:bulkdeformedcodiff}, plus the assumption that each $\frakU_\rho$ is a Morse cycle, shows $\delta_\frakU^2=0$.
\end{enumerate}
We denote by $HF^*_\frakU(M;H,J;\bbL)$ the \emph{$\frakU$-bulk-deformed $(H,J)$-Floer cohomology}. Because it will be useful in the sequel, we will sketch the proof of observation (2). 

\begin{lem}\label{lem:observation1}
$\delta_\frakU^2=0$.
\end{lem}
We provide a proof for completeness, but cf. \cite[Section 4.1.8]{Sie21} and \cite[Lemma 4.18]{Sheridan} for similar arguments.

\begin{proof}
To avoid clutter, we will consider the $\mu=1$ case. First,
    \begin{equation}
    \delta^2_{\frakU}(y)=\Bigg(\sum_z\sum_{q\geq0}\dfrac{1}{q!}\hbar^q\delta_{q,\frakU;y,z}\Bigg)\Bigg(\sum_x\sum_{\widetilde{q}\geq0}\dfrac{1}{\widetilde{q}!}\hbar^q\delta_{\widetilde{q},\frakU;z,x}\Bigg)x.
    \end{equation}
We will now immediately restrict to (1) a single $x$-component and (2) a single $\hbar^Q$-component for some fixed $Q$, i.e., we consider 
    \begin{equation}
    \sum_{z,Q=q+\widetilde{q}}\dfrac{1}{q!\widetilde{q}!}\delta_{q,\frakU;y,z}\delta_{\widetilde{q},\frakU;z,x}.
    \end{equation}
Expanding, this expression becomes
    \begin{equation}\label{eqn:observation2proofsketch1}
    \sum_{z,Q=q+\widetilde{q}}\dfrac{1}{q!\widetilde{q}!}\Bigg(\sum_{a_q}(\frakU_{j_1}\cdots\frakU_{j_q})\abs{\bbF_{a_q}(y,z)}_\bbZ\Bigg)\Bigg(\sum_{a_{\widetilde{q}}}(\frakU_{j_1}\cdots\frakU_{j_{\widetilde{q}}})\abs{\bbF_{a_{\widetilde{q}}}(z,x)}_\bbZ\Bigg),
    \end{equation}
where
    \begin{equation}
    \abs{\bbF_{a_q}(y,z)}_\bbZ\;\;{\rm resp.}\;\;\abs{\bbF_{a_{\widetilde{q}}}(z,x)}_\bbZ
    \end{equation}
denotes the signed count. In particular, this becomes
    \begin{equation}
    \sum_{Q=q+\widetilde{q}}\dfrac{1}{q!\widetilde{q}!}\sum_{a_Q}\frakU_{j_1}\cdots\frakU_{j_Q}\sum_{z,a_Q=a_q+a_{\widetilde{q}}}\abs{\bbF_{a_q}(y,z)}_\bbZ\cdot\abs{\bbF_{a_{\widetilde{q}}}(z,x)}_\bbZ.
    \end{equation}
By Proposition \ref{prop:bulkdeformedcodiff}, for fixed $a_Q$, we see that
    \begin{equation}
    \sum_{z, a_Q=a_q+a_{\widetilde{q}}}\dfrac{Q!}{q!\widetilde{q}!}\abs{\bbF_{a_q}(y,z)}_\bbZ\cdot\abs{\bbF_{a_{\widetilde{q}}}(z,x)}_\bbZ+\sum_{b,\sigma}\abs{\bbF_{b_{Q,\sigma}}(y,x)}_\bbZ\cdot\abs{\bbM_{ba_{j_\sigma}}}_\bbZ=0.
    \end{equation}
We now substitute the previous equation into \eqref{eqn:observation2proofsketch1}; this yields
    \begin{equation}
    -\sum_{Q=q+\widetilde{q}}\dfrac{1}{Q!}\sum_{a_Q}\frakU_{j_1}\cdots\frakU_{j_Q}\sum_{b,\sigma}\abs{\bbF_{b_{Q,\sigma}}(y,x)}_\bbZ\cdot\abs{\bbM_{ba_{j_\sigma}}}_\bbZ,
    \end{equation}
which we rewrite as
    \begin{equation}
    -\sum_{Q=q+\widetilde{q}}\dfrac{1}{Q!}\sum_{b,\sigma, a_Q-\{a_\sigma\}}\frakU_{j_1}\cdots\widehat{\frakU_{j_\sigma}}\cdots\frakU_{j_Q}\abs{\bbF_{b_{Q,\sigma}}(y,x)}_\bbZ\sum_{a_\sigma}{\frakU_{j_\sigma}}\abs{\bbM_{ba_{j_\sigma}}}_\bbZ.
    \end{equation}
Observe, the previous expression vanishes since $\frakU$ is a Morse cocycle, as desired.
\end{proof}

Finally, Proposition \ref{prop:bulkdeformedcont} shows we may define a cochain map
    \begin{equation}
    \overline{\frakc}_{\frakU,\kappa,\kappa'}:CF^*_\frakU(M;H,J;\bbL)\to CF^*_\frakU(M;H',J';\bbL),
    \end{equation}
where $\kappa'\geq\kappa$, whose $\bbL\abs{\frako_y}-\bbL\abs{\frako_x}$ component is given by 
    \begin{equation}
    \overline{\frakc}_{\frakU,\kappa,\kappa';y',x}=\sum_{q\geq0}\sum_{\frakq_1,\ldots,\frakq_\mu}\dfrac{1}{\abs{\frakq_1}!\cdots\abs{\frakq_\mu}!}\hbar_1^{\abs{\frakq_1}}\cdots\hbar_\mu^{\abs{\frakq_\mu}}\overline{\frakc}_{\frakq_1,\ldots,\frakq_\mu,\frakU,\kappa,\kappa';y,x},
    \end{equation}
where 
    \begin{multline}
    \overline{\frakc}_{\frakq_1,\ldots,\frakq_\mu,\frakU,\kappa,\kappa';y,x}\equiv \\
    \sum_{\substack{a_{\frakq_1},\ldots,a_{\frakq_\mu} \\ u\in\overline{\frakc}_{a_{\frakq_1},\ldots,a_{\frakq_\mu},\kappa,\kappa'}(x,y') \\ \deg(x)=\deg(y)+\sum_{\rho=1}^\mu\abs{\frakq_\rho}(\ell_\rho-2)}}(\frakU_{1,j_1}\cdots\frakU_{1,j_{\abs{\frakq_1}}}\cdots\frakU_{\mu,j_1}\cdots\frakU_{\mu,j_{\abs{\frakq_\mu}}})\mu_u,
    \end{multline}
where $\mu_u:\frako_{y'}\xrightarrow{\sim}\frako_x$ is the isomorphism induced on orientation lines by $u$. Observe, since $\chi(H)$ and $\chi(H')$ are finite, $\overline{\frakc}_{\frakq_1,\ldots,\frakq_\mu,\frakU,\kappa,\kappa'}$ vanishes for $q\gg0$. 

We define \emph{$\frakU$-bulk-deformed symplectic cohomology} via
    \begin{equation}
    SH^*_\frakU(M;\bbL)\equiv\varinjlim_{\ell}HF^*_\frakU(M;H^\ell,J^\ell;\bbL),
    \end{equation}
where the directed system is over a family of regular Floer data $\big\{(H^\ell,J^\ell)\big\}_{\ell\in\bbZ_{\geq0}}$ satisfying 
    \begin{equation}
    \kappa_\ell<\kappa_{\ell'},\;\;\forall\ell<\ell'.
    \end{equation}
It is known $SH^*_\frakU(M;\bbL)$ is independent of all auxiliary choices in the construction (in particular, $J_{\rm ref}$) and only depends on the homology class of $\frakU$, cf. \cite[Parts 4.2.7 \& 4.2.8]{Sie21} for the analogous statements using smooth relative cycles.

\section{de Rham to Morse bulk-deformations}
\subsection{Floer homotopy on Liouville manifolds}
In this subsection, we will digress a bit and briefly review how a graded Liouville manifold naturally yields a pre-GRR flow category which is filtered by finite pre-GRR flow categories, i.e., we review the construction of spectral symplectic cohomology.

Let $(H,J)$ be regular Floer data. It was shown in \cite[Section 6]{Lar21} (cf. also \cite[Section 8]{PS24a} and \cite[Section 6]{PS25c}) that each $\bbF_{yz}$ admits the structure of a compact smooth manifold with corners whose codimension 1 boundary strata are enumerated by gluing maps of the form 
    \begin{equation}
    \partial_z\equiv\bbF_{yz}\times\bbF_{zx}\hookrightarrow\bbF_{yx}.
    \end{equation}
We have the following two basic relations. First, a short exact sequence 
    \begin{equation}
    0\to\underline{\bbR}\to T\bbF_{yx}\vert_{\partial_z}\to T\bbF_{yz}\oplus T\bbF_{zx}\to0
    \end{equation}
given by taking a collar neighborhood of $\partial_z$. Second, a short exact sequence
    \begin{equation}
    0\to\underline{\bbR}\to\ind D\overline{\partial}_{H,J}^{yx}\to T\bbF_{yx}\vert_{{\rm int}\bbF_{yx}}\to0
    \end{equation}
given by the translational direction. There are straightforward relations when passing to higher codimension boundary strata. The proof of the following result is contained in \cite[Section 7]{Lar21} (cf. also \cite[Section 8]{PS24a} and \cite[Section 6]{PS25c}).

\begin{prop}\label{prop:indexbundlefloer}
There is an extension of $\ind D\overline{\partial}_{H,J}^{yx}$ to $\bbF_{yz}$ whose restriction to the interior of a codimension 1 boundary stratum is of the form 
    \begin{equation}
    \ind D\overline{\partial}_{H,J}^{yx}\vert_{\partial_z}=\ind D\overline{\partial}_{H,J}^{yz}\oplus\ind D\overline{\partial}_{H,J}^{zy}.
    \end{equation}
The natural associativity diagram,
    \begin{equation}
    \begin{tikzcd}
    & & \underline{\bbR}\arrow[d] \\
    \underline{\bbR}\arrow[d,"\Delta"]\arrow[r] & \ind D\overline{\partial}_{H,J}^{yx}\vert_{\partial_z}\arrow[r]\arrow[d,equals] & T\bbF_{yx}\vert_{\partial_z}\arrow[d] \\
    \underline{\bbR}^2\arrow[r] & \ind D\overline{\partial}_{H,J}^{yz}\oplus\ind D\overline{\partial}_{H,J}^{zx}\arrow[r] & T\bbF_{yz}\oplus T\bbF_{zx},
    \end{tikzcd}
    \end{equation}
commutes. Moreover, the natural associativity diagram associated to higher codimension boundary strata commute.
\end{prop}

Let $\bbF^{H,J}$ be the unstructured flow category with objects $\chi(H)$ and morphism spaces $\bbF^{H,J}(y,z)\equiv\bbF_{yz}$; the composition map is the standard gluing of Floer trajectories. We endow $\bbF^{H,J}$ with a pre-GRR structure as follows.
\begin{enumerate}
\item\emph{Universal curve}: We define 
    \begin{equation}
    \scrC^{H,J}_{yx}\equiv\bbF_{yx}\times(\bbR\times S^1)\to\bbF_{yx}
    \end{equation}
with the obvious projection map $\pi_{yx}$. Let 
    \begin{equation}
    \big\{\varphi_{yzx}:\bbF_{yz}\times\bbF_{zx}\times[0,\epsilon)_\tau\hookrightarrow\bbF_{yx}\big\}
    \end{equation}
be a system of collars on $\bbF^{H,J}$. We use the isomorphism 
    \begin{equation}
    \scrC^{H,J}_{yz}\#_{1/\tau}\scrC^{H,J}_{zx}\xrightarrow{\sim}\varphi^*_{yzx}\scrC_{yx}
    \end{equation}
induced by the gluing of domains used in the gluing map 
    \begin{equation}
    \bbF_{yz}\times\bbF_{zx}\times[0,\epsilon)_\tau\hookrightarrow\bbF_{yx}.
    \end{equation}
In particular, since the gluing maps for Floer trajectories satisfy the natural associativity diagrams, we have constructed a universal curve 
    \begin{equation}
    \pi:\scrC^{H,J}\to\bbF^{H,J}.
    \end{equation}
\item\emph{Complex vector bundle}: By construction of $\pi:\scrC^{H,J}\to\bbF^{H,J}$, we have a natural evaluation map
    \begin{equation}
    \eval:\scrC^{H,J}\to M
    \end{equation}
which is defined via its restriction to a fiber:
    \begin{equation}
    \eval_{yx}\vert_u(s,t)\equiv u(s,t),\;\;u\in\bbF_{yx}.
    \end{equation}
We now obtain a complex vector bundle $\scrE^{H,J}\to\scrC^{H,J}$ defined via pullback:
    \begin{equation}
    \scrE^{H,J}\equiv\eval^*TM.
    \end{equation}
\item\emph{Floer-Cauchy-Riemann operators}: We take our puncture datum on $\bbF^{H,J}$ to be defined via 
    \begin{equation}
    \frake^{H,J}\equiv\big\{D(\overline{\partial}_{H,J})_x\big\}_{x\in\chi(H)}.
    \end{equation}
We take our family of Floer-Cauchy-Riemann operators $D\overline{\partial}_{H,J}$ on $\scrE^{H,J}$ with puncture datum $\frake^{H,J}$ to be defined via
    \begin{equation}
    D\overline{\partial}_{H,J}^{yx}\equiv\big\{D(\overline{\partial}^{yx}_{H,J})_u:W^{k,p}(\Theta;u^*TM)\to W^{k-1,p}(\Theta;u^*TM)\big\},\;\; kp>2.
    \end{equation}
(Observe, we have implicitly used the fact that $\eval^*TM\vert_{\pi^{-1}(u)}=u^*TM$.)
\item Proposition \ref{prop:indexbundlefloer} shows we have an isomorphism 
    \begin{equation}
    T\bbF_{yx}\oplus\underline{\bbR}\cong\ind D\overline{\partial}^{yx}_{H,J}
    \end{equation}
which satisfies the natural associativity diagram.
\end{enumerate}
In particular, $\bbF^{H,J}$ lifts to a pre-GRR flow category; we denote this pre-GRR lift by $\bbF^{H,J,{\rm U}}$. It remains to show that Floer continuation maps admit lifts to Floer continuation pre-GRR flow bimodules.

Now, consider two pairs of regular Floer data, $(H,J)$ and $(H',J')$, satisfying $\kappa<\kappa'$. We choose monotonic Floer continuation data $(H_s,J_s)$ connecting $(H',J')$ to $(H,J)$. Again, it was shown in \cite[Section 6]{Lar21} (cf. also \cite[Section 8]{PS24a} and \cite[Section 6]{PS25c}) that each $\overline{\frakc}_{\kappa,\kappa'}(x,y')$ admits the structure of a compact smooth manifold with corners whose codimension 1 boundary strata are enumerated by gluing maps of the form 
    \begin{align}
    \overline{\frakc}_{\kappa,\kappa'}(x,x')\times\bbF_{x'y'}&\hookrightarrow\overline{\frakc}_{\kappa,\kappa'}(x,y'), \\
    \bbF_{xy}\times\overline{\frakc}_{\kappa,\kappa'}(y,y')&\hookrightarrow\overline{\frakc}_{\kappa,\kappa'}(x,y').
    \end{align}
There are short exact sequences associated to the codimension 1 boundary strata of $\overline{\frakc}_{\kappa,\kappa'}(x,y')$ which are completely analogous to those for the codimension 1 boundary strata of $\bbF_{yx}$; moreover, they are compatible in a completely analogous way to Proposition \ref{prop:indexbundlefloer}. Let 
    \begin{equation}
    \frakc_{\kappa,\kappa'}:\bbF^{H,J}\to\bbF^{H',J'}
    \end{equation}
be the unstructured flow bimodule between $\bbF^{H,J}$ and $\bbF^{H',J'}$ determined by the $\frakc_{\kappa,\kappa'}(x,y')$'s. We endow $\frakc_{\kappa,\kappa'}$ with a pre-GRR structure as follows.
\begin{enumerate}
\item The source and target are endowed with the pre-GRR structure defined above.
\item\emph{Universal curve}: We define 
    \begin{equation}
    \scrC^{\kappa,\kappa'}_{x,x'}\equiv\overline{\frakc}_{\kappa,\kappa'}(x,y')\times(\bbR\times S^1)\to\overline{\frakc}_{\kappa,\kappa'}(x,y')
    \end{equation}
with the obvious projection map $\pi_{yx}$. Let 
    \begin{align}
    \big\{\varphi_{yzx}:\bbF_{yz}\times\bbF_{zx}\times[0,\epsilon)_\tau\hookrightarrow\bbF_{yx}\big\}&\amalg \nonumber \\
    \big\{\varphi_{xx'y'}:\overline{\frakc}_{\kappa,\kappa'}(x,x')\times\bbF_{x'y'}\hookrightarrow\overline{\frakc}_{\kappa,\kappa'}(x,y')\big\}&\amalg \nonumber \\
    \big\{\varphi_{xyy'}\bbF_{xy}\times\overline{\frakc}_{\kappa,\kappa'}(y,y')\hookrightarrow\overline{\frakc}_{\kappa,\kappa'}(x,y')\big\}&\amalg \nonumber \\
    \big\{\varphi_{y'z'x'}:\bbF_{y'z'}\times\bbF_{z'x'}\times[0,\epsilon)_\tau\hookrightarrow\bbF_{y'x'}\big\}
    \end{align}
be a system of collars on $\overline{\frakc}_{\kappa,\kappa'}$. We use the isomorphisms 
    \begin{align}
    \scrC^{\kappa,\kappa'}_{xx'}\#_{1/\tau}\scrC^{H',J'}_{x'y'}&\xrightarrow{\sim}\varphi^*_{xx'y'}\scrC^{\kappa,\kappa'}_{xy'}, \\
    \scrC^{H,J}_{xy}\#_{1/\tau}\scrC^{\kappa,\kappa'}_{yy'}&\xrightarrow{\sim}\varphi^*_{xyy'}\scrC^{\kappa,\kappa'}_{xy'} 
    \end{align}
induced by the gluing of domains used in the gluing maps 
    \begin{align}
    \overline{\frakc}_{\kappa,\kappa'}(x,x')\times\bbF_{x'y'}\times[0,\epsilon)_\tau&\hookrightarrow\overline{\frakc}_{\kappa,\kappa'}(x,y'), \\
    \bbF_{xy}\times\overline{\frakc}_{\kappa,\kappa'}(y,y')\times[0,\epsilon)_\tau&\hookrightarrow\overline{\frakc}_{\kappa,\kappa'}(x,y').
    \end{align}
In particular, since the gluing maps for Floer continuation cylinders and Floer trajectories satisfy the natural (intertwining) associativity diagrams, we have constructed a universal curve 
    \begin{equation}
    \pi:\scrC^{\kappa,\kappa'}\to\overline{\frakc}_{\kappa,\kappa'}.
    \end{equation}
\item\emph{Complex vector bundle}: By construction of $\pi:\scrC^{\kappa,\kappa'}\to\overline{\frakc}_{\kappa,\kappa'}$, we have a natural evaluation map
    \begin{equation}
    \eval:\scrC^{\kappa,\kappa'}\to M
    \end{equation}
which is defined via its restriction to a fiber:
    \begin{equation}
    \eval_{xy'}\vert_u(s,t)\equiv u(s,t),\;\;u\in\overline{\frakc}_{\kappa,\kappa'}(x,y').
    \end{equation}
We now obtain a complex vector bundle $\scrE^{\kappa,\kappa'}\to\scrC^{\kappa,\kappa'}$ defined via pullback:
    \begin{equation}
    \scrE^{\kappa,\kappa'}\equiv\eval^*TM.
    \end{equation}
\item\emph{Floer-Cauchy-Riemann operators}: We take our puncture datum on $\bbF^{H,J}$ to be defined via 
    \begin{equation}
    \frake^{\kappa,\kappa'}\equiv\big\{D(\overline{\partial}_{H,J})_x\big\}_{x\in\chi(H)}\amalg\big\{D(\overline{\partial}_{H,J})_{y'}\big\}_{y'\in\chi(H')}.
    \end{equation}
We take our family of Floer-Cauchy-Riemann operators $D\overline{\partial}_{\kappa,\kappa'}$ on $\scrE^{\kappa,\kappa'}$ with puncture datum $\frake^{\kappa,\kappa'}$ to be defined via
    \begin{equation}
    D\overline{\partial}^{xy'}_{\kappa,\kappa'}\equiv\big\{D(\overline{\partial}^{xy'}_{\kappa,\kappa'})_u:W^{k,p}(\Theta;u^*TM)\to W^{k-1,p}(\Theta;u^*TM)\big\},\;\; kp>2
    \end{equation}
given by linearizing the Floer continuation equation. (Observe, we have implicitly used the fact that $\eval^*TM\vert_{\pi^{-1}(u)}=u^*TM$.)
\item The analogous statement to Proposition \ref{prop:indexbundlefloer} for Floer continuation cylinders shows we have an isomorphism 
    \begin{equation}
    T\overline{\frakc}_{\kappa,\kappa'}(x,y')\cong\ind D\overline{\partial}^{xy'}_{\kappa,\kappa'}
    \end{equation}
which satisfies the natural (intertwining) associativity diagram(s).
\end{enumerate}
In particular, $\overline{\frakc}_{\kappa,\kappa'}$ lifts to a pre-GRR flow bimodule; we denote this pre-GRR lift by $\overline{\frakc}^{\rm U}_{\kappa,\kappa'}$. We define spectral symplectic cohomology via 
    \begin{equation}
    \bbF^{\rm U}\equiv\varinjlim_{\ell}\bbF^{H^\ell,J^\ell,{\rm U}}
    \end{equation}
where the directed system is over a family of regular Floer data $\big\{(H^\ell,J^\ell)\big\}_{\ell\in\bbZ_{\geq0}}$ satisfying 
    \begin{equation}
    \kappa_\ell<\kappa_{\ell'},\;\;\forall\ell<\ell'.
    \end{equation}
By lifting the usual cohomological arguments, it straightforwardly follows $\bbF^{\rm U}$ is independent of all auxiliary choices in the construction.

\subsection{de Rham bulk-deformations, a second pass}\label{subsec:derhamsecondpass}

Consider the pre-GRR flow category $\bbF^{H,J,{\rm U}}$; using Subsection \ref{subsec:derhamfirstpass}, we have the chain complex
    \begin{equation}
    CM_*\Big(\bbF^{H,J,{\rm U}},\exp\big(\pi_!\frakc\frakh(\eval^*TM)\big);\bC[\underline b]\Big).
    \end{equation}
In this subsection, we will open up this construction and repackage it. 

First, let us define a more familiar notion of de Rham bulk-deformed symplectic cohomology. Recall, $\dim M=2n$. Let $\frakc\frakh\equiv\big(\ch_2(TM),\ldots,\ch_n(TM)\big)$ be our $(n-1)$-tuple of closed differential forms and consider a partition
    \begin{align}
    \{1,\ldots,q\}&=\{1,\ldots,q_1\}\amalg\{q_1+1,\ldots,q_2\}\amalg\cdots\amalg\{q_{n-2}+1,\ldots,q_{n-1}\} \\
    &\equiv\frakq_1\amalg\cdots\amalg\frakq_{n-1}.
    \end{align}
For any two $x,y\in\chi(H)$, we have a map 
    \begin{equation}
    \eval_{yx;q}:(\scrC^{H,J}_{yx})^q\to M;
    \end{equation}
equivalently, we have a map
    \begin{equation}
    \eval_{yx;\frakq_1}\times\cdots\times\eval_{yx;\frakq_{n-1}}:(\scrC^{H,J}_{yx})^{\abs{\frakq_1}}\times\cdots\times(\scrC^{H,J}_{yx})^{\abs{\frakq_{n-1}}}\to M.
    \end{equation}
We define $R_{n-1}\equiv\bbC[\frakb_1,\ldots,\frakb_{n-1}]$, where $\frakb_\rho$ is a formal variable of degree $-2\rho$. The $\frakc\frakh$-bulk-deformed $(H,J)$-Floer cochain complex 
    \begin{equation}
    CF^j_{\frakc\frakh}(M;H,J;R_{n-1})\equiv\bigoplus_{\deg(x)=j}R_{n-1}\abs{\frako_x}
    \end{equation}
has the codifferential whose $R_{n-1}\abs{\frako_y}-R_{n-1}\abs{\frako_x}$ component is given by 
    \begin{equation}
    \delta_{\frakc\frakh;y,x}\equiv\sum_{q\geq0}\sum_{\frakq_1,\ldots,\frakq_{n-1}}\dfrac{1}{\abs{\frakq_1}!\cdots\abs{\frakq_{n-1}}!}\frakb^{\abs{\frakq_1}}_1\cdots\frakb^{\abs{\frakq_{n-1}}}_{n-1}\delta_{\frakq_1,\ldots,\frakq_{n-1},\frakc\frakh;y,x}, 
    \end{equation}
where
    \begin{multline}
    \delta_{\frakq_1,\ldots,\frakq_{n-1},\frakc\frakh;y,x}\equiv \\
    \int_{\bbF_{yx}}\int_{(\scrC^{H,J}_{yx})^q/\bbF_{yx}}\eval_{yx;q}^*\Big(\ch_2(TM)^{\abs{\frakq_1}}\wedge\cdots\wedge\ch_n(TM)^{\abs{\frakq_{n-1}}}\Big).
    \end{multline}
Two things to observe are the following:
\begin{enumerate}
\item since $\chi(H)$ is finite, $\delta_{\frakq_1,\ldots,\frakq_{n-1},\frakc\frakh;y,x}$ vanishes for $q\gg0$;
\item and the fact that $\pi_!\frakc\frakh(\eval^*TM)$ is a multiplicatively coherent cocycle shows $\delta_{\frakc\frakh}^2=0$.
\end{enumerate}

\begin{rem}
Observe, the proof that $\delta_{\frakc\frakh}^2=0$ is analogous to the proof of Lemma \ref{lem:observation1}; the combinatorial manipulation is the same (in fact, simpler since there is no Morse breaking), and one replaces the Morse cocycle observation with the multiplicatively coherent cocycle observation.
\end{rem}

Second, let us open up the definition of 
    \begin{equation}
    \exp\big(\pi_!\frakc\frakh(\eval^*TM)\big).
    \end{equation}
For any two $x,y\in\chi(H)$, we have that 
    \begin{align}
    \exp\big(\pi_!\frakc\frakh(\eval^*TM)\big)_{yx}&=\prod_{\rho=0}^{n-2}\exp\Big(\pi_!\big(\eval_{yx}^*\ch_{\rho+2}(TM)\big)b_{\rho+1}\Big) \\
    &=\prod_{\rho=0}^{n-2}\Big(1+\pi_!\big(\eval_{yx}^*\ch_{\rho+2}(TM)\big)b_{\rho+1}+\cdots\Big).
    \end{align}
Observe, the last line of the previous equation equals
    \begin{multline}
    \sum_{q\geq0}\sum_{\frakq_1,\cdots,\frakq_{n-1}}\dfrac{1}{\abs{\frakq_1}!\cdots\abs{\frakq_{n-1}}!}b^{\abs{\frakq_1}}_1\cdots b^{\abs{\frakq_{n-1}}}_{n-1} \\
    \pi_!\big(\eval_{yx}^*\ch_2(TM)\big)^{\abs{\frakq_1}}\cdots\pi_!\big(\eval_{yx}^*\ch_n(TM)\big)^{\abs{\frakq_{n-1}}}.
    \end{multline}
If we consider a single summand, we have the following by definition:
    \begin{multline}
    \pi_!\big(\eval_{yx}^*\ch_2(TM)\big)^{\abs{\frakq_1}}\cdots\pi_!\big(\eval_{yx}^*\ch_n(TM)\big)^{\abs{\frakq_{n-1}}}= \\
    \Bigg(\int_{\scrC^{H,J}_{yx}/\bbF_{yx}}\eval_{yx}^*\ch_2(TM)\Bigg)^{\abs{\frakq_1}}\cdots\Bigg(\int_{\scrC^{H,J}_{yx}/\bbF_{yx}}\eval_{yx}^*\ch_n(TM)\Bigg)^{\abs{\frakq_{n-1}}}.
    \end{multline}
Finally, performing Fubini-Tonelli in coordinates shows the previous equation yields
    \begin{equation}
    \int_{(\scrC^{H,J}_{yx})^{\abs{\frakq_1}}\times\cdots\times(\scrC^{H,J}_{yx})^{\abs{\frakq_{n-1}}}/\bbF_{yx}}\eval_{yx}^*\ch_2(TM)^{\abs{\frakq_1}}\wedge\cdots\wedge\eval_{yx}^*\ch_n(TM)^{\abs{\frakq_{n-1}}}, 
    \end{equation}
i.e., we have that $\exp\big(\pi_!\frakc\frakh(\eval^*TM)\big)_{yx}$ equals
    \begin{multline}
    \sum_{q\geq0}\sum_{\frakq_1,\ldots,\frakq_{n-1}}\dfrac{1}{\abs{\frakq_1}!\cdots\abs{\frakq_{n-1}}!}b_1^{\abs{\frakq_1}}\cdots b_{n-1}^{\abs{\frakq_{n-1}}} \\
    \int_{(\scrC^{H,J}_{yx})^q/\bbF_{yx}}\eval_{yx;q}^*\Big(\ch_2(TM)^{\abs{\frakq_1}}\wedge\cdots\wedge\ch_n(TM)^{\abs{\frakq_{n-1}}}\Big).
    \end{multline}
Therefore, the codifferential of $CM_*\Big(\bbF^{H,J,{\rm U}},\exp\big(\pi_!\frakc\frakh(\eval^*TM)\big);\bbC[\underline{b}]\Big)$ is seen to be the $\bbC[\underline{b}]$-linear extension of
    \begin{multline}
    y\mapsto\sum_{q\geq0}\sum_{\frakq_1,\ldots,\frakq_{n-1}}\dfrac{1}{\abs{\frakq_1}!\cdots\abs{\frakq_{n-1}}!}b_1^{\abs{\frakq_1}}\cdots b_{n-1}^{\abs{\frakq_{n-1}}} \\
    \int_{\bbF_{yx}}\int_{(\scrC^{H,J}_{yx})^q/\bbF_{yx}}\eval_{yx;q}^*\Big(\ch_2(TM)^{\abs{\frakq_1}}\wedge\cdots\wedge\ch_n(TM)^{\abs{\frakq_{n-1}}}\Big).
    \end{multline}
The upshot of the discussion so far is the following result. 

\begin{lem}\label{lem: piorehjgpowehgpih}
We have the following equality of cochain complexes over $\bC[\underline b]$:
    \begin{equation}
    CF^{-*}_{\frakc\frakh}(M;H,J;R_{n-1})\otimes_\bbC\bbC[b_n,\ldots]=CM_*\Big(\bbF^{H,J,{\rm U}},\exp\big(\pi_!\frakc\frakh(\eval^*TM)\big);\bC[\underline b]\Big).
    \end{equation}
\end{lem}

In a straightforward manner, we may repeat the previous discussion for Floer continuation maps; this will yield the following result.

\begin{lem}
We have the commutative diagram 
    \begin{equation}
    \begin{tikzcd}[column sep=small, center picture]
    CF^{-*}_{\frakc\frakh}(M;H,J;R_{n-1})\otimes_\bbC\bbC[b_n,\ldots]\arrow[r]\arrow[d,equals] & CF^{-*}_{\frakc\frakh}(M;H',J';R_{n-1})\otimes_\bbC\bbC[b_n,\ldots]\arrow[d,equals] \\
    CM_*\Big(\bbF^{H,J,{\rm U}},\exp\big(\pi_!\frakc\frakh(\eval^*TM)\big);\bC[\underline b]\Big)\arrow[r] & CM_*\Big(\bbF^{H',J',{\rm U}},\exp\big(\pi_!\frakc\frakh(\eval^*TM)\big);\bC[\underline b]\Big),
    \end{tikzcd}
    \end{equation}
where: the top arrow is the $\frakc\frakh$-bulk-deformed Floer continuation map, the bottom arrow is the map induced by the image of $\overline{\frakc}^{\rm U}_{\tau,\tau'}$ under $\smallint{}_{\scrC/\bbX}\ch(\scrE)$, and the equalities are from the previous lemma.
\end{lem}

\subsection{Interpolating moduli spaces}
Recall, we have $\bbC$-linear comparison map 
    \begin{align}
    \Phi:\Omega^j(M;\bbC)&\to CM_{2n-j}(M;\bbC) \\
    \eta&\mapsto\sum_{a\in\crit(f)}\Bigg(\int_{W^s(a)}\eta\Bigg)a \nonumber
    \end{align}
which satisfies $[\Phi,\partial]=0$. We define the Morse cycle
    \begin{equation}\label{eqn:definuch}
    \frakU^{\frakc\frakh}_\rho\equiv\Phi\big(\ch_{\rho+1}(TM)\big).\footnote{I.e., we consider the $(n-1)$-tuple $(\frakU^{\frakc\frakh}_1,\ldots,\frakU^{\frakc\frakh}_{n-1})$ induced by the $(n-1)$-tuple $\big(\ch_2(TM),\ldots,\ch_n(TM)\big)$.}
    \end{equation}
As before, we write
    \begin{equation}
    \frakU^{\frakc\frakh}_\rho\equiv\sum_{1\leq j_\rho\leq k_\rho}\frakU^{\frakc\frakh}_{\rho,j_\rho}a_{\rho,j_\rho}\in CM_{2n-2\rho}(M;\bbC).
    \end{equation}

The following two subsections are dedicated to proving the following result.

\begin{prop}\label{prop:quasi-isomorphismbd}
There is a quasi-isomorphism
    \begin{equation}\label{eqn:quasi-isomorphismbd}
    CF^*_{\frakc\frakh}(M;H,J;R_{n-1})\to CF^*_{\frakU^{\frakc\frakh}}(M;H,J;\bbL)
    \end{equation}
of cochain complexes over $\bbC$ which, after passing to homology, is an isomorphism of modules over $R_{n-1}\cong\bbL$. Moreover, the aforementioned quasi-isomorphism is compatible with the appropriate Floer continuation maps:
    \begin{equation}
    \begin{tikzcd}
    CF^*_{\frakc\frakh}(M;H,J;R_{n-1})\arrow[r]\arrow[d] & CF^*_{\frakU^{\frakc\frakh}}(M;H,J;\bbL)\arrow[d] \\
    CF^*_{\frakc\frakh}(M;H',J';R_{n-1})\arrow[r] & CF^*_{\frakU^{\frakc\frakh}}(M;H',J';\bbL).
    \end{tikzcd}
    \end{equation}
\end{prop}

Consider the following differential graded algebra (DGA): 
    \begin{equation}
    \Big(\widetilde{\bbL}\equiv\bbC[\frakb_1,\ldots,\frakb_{n-1},\frakL_1,\ldots,\frakL_{n-1},\hbar_1,\ldots,\hbar_{n-1}],d\Big),\;\;\deg(\frakL_\rho)=-2\rho+1,
    \end{equation}
where 
    \begin{equation}
    d\frakb_\rho=d\hbar_\rho=\frakL_\rho\;\;{\rm and}\;\;d\frakL_\rho=0
    \end{equation}
extended $\bC$-linearly and via Leibniz. We endow $R_{n-1}$ and $\bbL$ with the trivial differential. Observe, we have obvious projections: 
    \begin{equation}\label{eqn:obviouscoeffproj}
    \begin{tikzcd}
    & \widetilde{\bbL}\arrow[dl,"\pi_\frakb",swap]\arrow[dr,"\pi_\hbar"] & \\
    R_{n-1} & & \bbL.
    \end{tikzcd}
    \end{equation}

\begin{lem}
Both $\pi_\frakb$ and $\pi_\hbar$ are quasi-isomorphisms of DGAs.
\end{lem}

\begin{proof}
We compute: 
    \begin{align}
    H_*(\widetilde{\bbL})&=\dfrac{\langle\frakb_1-\hbar_1,\ldots,\frakb_{n-1}-\hbar_{n-1},\frakL_1,\ldots,\frakL_{n-1}\rangle}{\langle\frakL_1,\ldots,\frakL_{n-1}\rangle} \\
    &=\langle\frakb_1-\hbar_1,\ldots,\frakb_{n-1}-\hbar_{n-1}\rangle.
    \end{align}
Clearly, $\pi_\frakb$ resp. $\pi_\hbar$, which sets $\frakL_\rho$ and $\hbar_\rho$ resp. $\frakb_\rho$ and $\frakL_\rho$ to zero, induces an isomorphism on homology.
\end{proof}

We will now construct ``interpolating $(\frakc\frakh,\frakU^{\frakc\frakh})$-bulk-deformed symplectic cohomology''. Consider a partition
    \begin{equation}\label{eqn:partition}
    \{1,\ldots,q\}\equiv\frakq^\frakb_1\amalg\cdots\amalg\frakq^\frakb_{n-1}\amalg\frakq^\frakL_1\amalg\cdots\amalg\frakq^\frakL_{n-1}\amalg\frakq^\hbar_1\amalg\cdots\amalg\frakq^\hbar_{n-1};
    \end{equation}
choose 
    \begin{equation}
    j_{\rho_1},\ldots,j_{\rho_{\abs{\frakq^\hbar_\rho}}}
    \end{equation}
between 1 and $k_\rho$. (We would like to emphasize that we do not require a relation between $j_{\rho_1},\ldots,j_{\rho_{\abs{\frakq^\hbar_\rho}}}$.) Analogously to \eqref{eqn:bulkdeformedaux}, given any two $x,y\in\chi(H)$, we may define
    \begin{equation}
    \bbF_{\frakq^\hbar_1,\ldots,\frakq^\hbar_{n-1},\frakU^{\frakc\frakh}}(y,x)\equiv\coprod_{a_{\frakq^\hbar_1},\ldots,a_{\frakq^\hbar_{n-1}}}\bbF_{a_{\frakq^\hbar_1},\ldots,a_{\frakq^\hbar_{n-1}}}(y,x),
    \end{equation}
where $a_{\frakq^\hbar_\rho}\equiv(a_{\rho,j_{\rho_1}},\ldots,a_{\rho,j_{\abs{\frakq^\hbar_\rho}}})$ is a $\abs{\frakq^\hbar_\rho}$-tuple of terms appearing in $\frakU^{\frakc\frakh}_\rho$ and 
    \begin{multline}
    \bbF_{a_{\frakq^\hbar_1},\ldots,a_{\frakq^\hbar_{n-1}}}(y,x)\equiv\coprod_{T\in T^{semi}_{\calC_{q+2}}}\bbF_T(y,x)\times_{\widehat{M}^{\abs{\frakq^\hbar_1}+\cdots+\abs{\frakq^\hbar_{n-1}}}} \\
    \big(W^s(a_{1,j_1})\times\cdots\times W^s(a_{1,j_{\abs{\frakq^\hbar_1}}})\times\cdots\times W^s(a_{{n-1},j_1})\times\cdots\times W^s(a_{{n-1},j_{\abs{\frakq^\hbar_{n-1}}}})\big)
    \end{multline}
equipped with the Gromov topology (here, we implicitly only use the evaluation maps on $\bbF_T(y,x)$ at the marked points labeled by elements of $\{1,\ldots,q\}$ which are contained in the subset $\frakq^\hbar_1\amalg\cdots\amalg\frakq^\hbar_{n-1}$; in particular, here we are thinking of an element of $\bbF_T(y,x)$ as having three kinds of marked points given by the partition \eqref{eqn:partition}: $\frakb$ ones, $\frakL$ ones, and $\hbar$ ones). We use the shorthand
    \begin{multline}
    \bbF_{T,a_{\frakq^\hbar_1},\ldots,a_{\frakq^\hbar_{n-1}}}(y,x)\equiv\bbF_T(y,x)\times_{\widehat{M}^{\abs{\frakq^\hbar_1}+\cdots+\abs{\frakq^\hbar_{n-1}}}} \\
    \big(W^s(a_{1,j_1})\times\cdots\times W^s(a_{1,j_{\abs{\frakq^\hbar_1}}})\times\cdots\times W^s(a_{{n-1},j_1})\times\cdots\times W^s(a_{{n-1},j_{\abs{\frakq^\hbar_{n-1}}}})\big).
    \end{multline}
As usual, we have a map 
    \begin{equation}\label{eqn:fjkld;ajfdkl;asjf}
    \eval_{\frakq^\frakb}:\bbF_{\frakq^\hbar_1,\ldots,\frakq^\hbar_{n-1},\frakU^{\frakc\frakh}}(y,x)\to\widehat{M}^{\abs{\frakq^\frakb_1}+\cdots+\abs{\frakq^\frakb_{n-1}}}
    \end{equation}
given by evaluation at the $\frakb$ marked points. Meanwhile, we have a map
    \begin{equation}\label{eqn:fkldjal;fjkl;ajlkfa}
    \eval_{\frakq^\frakL}:[0,+\infty)^{\abs{\frakq^\frakL_1}+\cdots+\abs{\frakq^\frakL_{n-1}}}\times\bbF_{\frakq^\hbar_1,\ldots,\frakq^\hbar_{n-1},\frakU^{\frakc\frakh}}(y,x)\to\widehat{M}^{\abs{\frakq^\frakL_1}+\cdots+\abs{\frakq^\frakL_{n-1}}}
    \end{equation}
whose $k$-th component is given by evaluating at the marked point with label $k\in\frakq^\frakL_\rho$ after following the flow of $-\grad f$ for time $\frakl_k\in[0,+\infty)$. Clearly, \eqref{eqn:fjkld;ajfdkl;asjf} extends over 
    \begin{equation}
    \scrF_{\frakL,\frakq^\hbar_1,\ldots,\frakq^\hbar_{n-1},\frakU^{\frakc\frakh}}(y,x)\equiv[0,+\infty)^{\abs{\frakq^\frakL_1}+\cdots+\abs{\frakq^\frakL_{n-1}}}\times\bbF_{\frakq^\hbar_1,\ldots,\frakq^\hbar_{n-1},\frakU^{\frakc\frakh}}(y,x).
    \end{equation}
The basic Gromov-compactness and gluing result in this context is the following.

\begin{prop}
$\scrF_{\frakL,\frakq^\hbar_1,\ldots,\frakq^\hbar_{n-1},\frakU^{\frakc\frakh}}(y,x)$ admits a Gromov-compactification, denoted $\bbF_{\frakL,\frakq^\hbar_1,\ldots,\frakq^\hbar_{n-1},\frakU^{\frakc\frakh}}(y,x)$, which is a compact smooth manifold with corners whose codimension 1 boundary strata are of the following forms.
\begin{itemize}
\item First, the subspace corresponding to Morse degenerations:
    \begin{align}
    &\coprod_T\bbF_{T,a_{\frakq^\hbar_1},\ldots,a_{\frakq^\hbar_{n-1}}}(y,x) \\
    &\coprod_{T,\rho,\sigma}\coprod_{\substack{b\in\crit(f) \\ \deg(b)=\deg(a_{\rho,j_\sigma})+1}}\bbF_{T,a_{\frakq^\hbar_1},\ldots,a_{\frakq^\hbar_{\rho-1}},b_{\frakq^\hbar_\rho,\sigma},\ldots,a_{\frakq^\hbar_{n-1}}}(y,x)\times\bbM_{ba_{\rho,j_\sigma}}, \nonumber
    \end{align}
where: the first disjoint union only uses $T\in T_{\calC_{q+2}}$ having exactly one plain edge and no round edges, the second disjoint union only uses the unique $T\in T_{\calC_{q+2}}$ having exactly no plain edges and no round edges, and 
    \begin{equation}
    b_{\frakq^\hbar_\rho,\sigma}\equiv(a_{\rho,j_1},\ldots,a_{\rho,j_{\sigma-1}},b,\widehat{a}_{\rho,j_\sigma},\ldots,a_{\rho,j_{\abs{\frakq^\hbar_\rho}}}).
     \end{equation}
\item Second, the boundary stratum corresponding to when any $\bbR$-coordinate goes to 0, i.e., when any label in a partition of the form $\frakq^\frakL_\rho$ becomes a label in $\frakq^\frakb_\rho$.
\item And, the boundary stratum corresponding to when any $\bbR$-coordinate goes to $+\infty$, i.e., when any label in a partition of the form $\frakq^\frakL_\rho$ becomes a label in $\frakq^\hbar_\rho$. (Note, this will use \eqref{eqn:definuch}.)
\end{itemize}
In particular, 
    \begin{equation}
    \bbF_{\frakL,\frakq^\hbar_1,\ldots,\frakq^\hbar_{n-1},\frakU^{\frakc\frakh}}(y,x)\equiv\coprod_{a_{\frakq^\hbar_1},\ldots,a_{\frakq^\hbar_{n-1}}}\partial\bbF_{\frakL,a_{\frakq^\hbar_1},\ldots,a_{\frakq^\hbar_{n-1}}}(y,x)
    \end{equation}
is a compact smooth manifold with corners satisfying
    \begin{equation}
    \partial\bbF_{\frakL,\frakq^\hbar_1,\ldots,\frakq^\hbar_{n-1},\frakU^{\frakc\frakh}}(y,x)=\coprod_{a_{\frakq^\hbar_1},\ldots,a_{\frakq^\hbar_{n-1}}}\partial\bbF_{\frakL,a_{\frakq^\hbar_1},\ldots,a_{\frakq^\hbar_{n-1}}}(y,x).
    \end{equation}
\end{prop}

\begin{rem}
Observe, \eqref{eqn:fjkld;ajfdkl;asjf} and \eqref{eqn:fkldjal;fjkl;ajlkfa} extend over $\bbF_{\frakL,\frakq^\hbar_1,\ldots,\frakq^\hbar_{n-1},\frakU^{\frakc\frakh}}(y,x)$; we write $\eval^*_{yx; \frakq^\frakb+\frakq^\frakL}$ for their product.
\end{rem}

We define the shorthand
    \begin{multline}\
    \Psi_{a_{\frakq^\hbar_1},\ldots,a_{\frakq^\hbar_{n-1}},\frakq^\frakb+\frakq^\frakL;y,x}\equiv \\
    \int_{\bbF_{\frakL,a_{\frakq^\hbar_1},\ldots,a_{\frakq^\hbar_{n-1}}}(y,x)}\int\eval_{yx;\frakq^\frakb+\frakq^\frakL}^*\Big(\ch_2(TM)^{\abs{\frakq^\frakb_1}}\wedge\cdots\wedge\ch_n(TM)^{\abs{\frakq^\frakL_{n-1}}}\Big),\footnotemark
    \end{multline}
where the inner integral is implicitly the fiber integral\footnotetext{Observe, this integral is only possibly non-trivial when $\deg(x)=\deg(y)+\sum_{\rho=1}^{n-1}2\rho\abs{\frakq^\frakb_\rho}+\sum_{\rho=1}^{n-1}2\rho\abs{\frakq^\frakL_\rho}+\sum_{\rho=1}^{n-1}2\rho\abs{\frakq^\hbar_\rho}+1$.}
    \begin{multline}
    \int_{\scrC_{\frakL,a_{\frakq^\hbar_1},\ldots,a_{\frakq^\hbar_{n-1}}}(y,x)^{\frakq^\frakb+\frakq^\frakL}/\bbF_{\frakL,a_{\frakq^\hbar_1},\ldots,a_{\frakq^\hbar_{n-1}}}(y,x)}\eval_{yx;\frakq^\frakb+\frakq^\frakL}^*\Big(\ch_2(TM)^{\abs{\frakq^\frakb_1}}\wedge \\
    \cdots\wedge\ch_n(TM)^{\abs{\frakq^\frakb_{n-1}}}\wedge\ch_2(TM)^{\abs{\frakq^\frakL_1}}\wedge\cdots\wedge\ch_n(TM)^{\abs{\frakq^\frakL_{n-1}}}\Big).
    \end{multline}
The interpolating $(\frakc\frakh,\frakU^{\frakc\frakh})$-bulk-deformed $(H,J)$-Floer cochain complex
    \begin{equation}
    CF^j_{\frakc\frakh,\frakU^{\frakc\frakh}}(M;H,J;\widetilde{\bbL})\equiv\bigoplus_{\deg(x)=j}\widetilde{\bbL}\abs{\frako_x}
    \end{equation}
has the codifferential whose $\widetilde{\bbL}\abs{\frako_y}-\widetilde{\bbL}\abs{\frako_x}$ component is given by 
    \begin{equation}
    \delta_{\frakc\frakh,\frakU^{\frakc\frakh};y,x}\equiv\sum_{q\geq0}\sum_{\frakq^\frakb_1,\ldots,\frakq^\hbar_{n-1}}\dfrac{1}{\abs{\frakq^\frakb_1}!\cdots\abs{\frakq^\hbar_{n-1}}!}\frakb_1^{\abs{\frakq^\frakb_1}}\cdots\hbar_{n-1}^{\abs{\frakq^\hbar_{n-1}}}\delta_{\frakq^\frakb_1,\ldots,\frakq^\hbar_{n-1},\frakc\frakh,\frakU^{\frakc\frakh};y,x}, 
    \end{equation}
(extended by Leibniz) where:
    \begin{multline}
    \delta_{\frakq^\frakb_1,\ldots,\frakq^\hbar_{n-1},\frakc\frakh,\frakU^{\frakc\frakh};y,x}\equiv \\
    \sum_{a_{\frakq^\hbar_1},\ldots,a_{\frakq^\hbar_{n-1}}}(\frakU^{\frakc\frakh}_{1,j_1}\cdots\frakU^{\frakc\frakh}_{1,j_{\abs{\frakq^\hbar_1}}}\cdots\frakU^{\frakc\frakh}_{n-1,j_1}\cdots\frakU^{\frakc\frakh}_{n-1,j_{\abs{\frakq^\hbar_{n-1}}}})\Psi_{a_{\frakq^\hbar_1},\ldots,a_{\frakq^\hbar_{n-1}},\frakq^\frakb+\frakq^\frakL;y,x}
    \end{multline}
if at least one of 
    \begin{equation}
    \abs{\frakq^\frakb_1},\ldots,\abs{\frakq^\frakL_{n-1}}
    \end{equation}
is non-zero, otherwise
    \begin{multline}
    \delta_{\frakq^\frakb_1,\ldots,\frakq^\hbar_{n-1},\frakc\frakh,\frakU^{\frakc\frakh};y,x}\equiv \\
    \sum_{\substack{a_{\frakq^\hbar_1},\ldots,a_{\frakq^\hbar_{n-1}} \\ u\in\bbF_{\frakL,a_{\frakq^\hbar_1},\ldots,a_{\frakq^\hbar_{n-1}}}(y,x) \\ \deg(x)=\deg(y)+\sum_{\rho=1}^{n-1}2\rho\abs{\frakq^\hbar_\rho}+1}}(\frakU_{1,j_1}\cdots\frakU_{1,j_{\abs{\frakq^\hbar_1}}}\cdots\frakU_{n-1,j_1}\cdots\frakU_{n-1,j_{\abs{\frakq^\hbar_{n-1}}}})\mu_u.
    \end{multline}
Two things to observe are the following:
\begin{enumerate}
\item since $\chi(H)$ is finite, $\delta_{\frakq^\frakb_1,\ldots,\frakq^\hbar_{n-1},\frakc\frakh,\frakU^{\frakc\frakh};y,x}$ vanishes for $q\gg0$;
\item and, essentially by combining the proofs of the claims in Subsections \ref{subsec:bulkdeformed} and \ref{subsec:derhamsecondpass}, $\delta_{\frakc\frakh,\frakU^{\frakc\frakh}}^2=0$.
\end{enumerate}

\subsection{Zig-zag of quasi-isomorphisms}
Observe, \eqref{eqn:obviouscoeffproj}, being a diagram of quasi-isomorphism of $\bbC$-DGAs, induces obvious quasi-isomorphisms of cochain complexes over $\bbC$:
    \begin{equation}\label{eqn:quasi-isooverc}
    \begin{tikzcd}[column sep=-5ex]
    & CF^*_{\frakc\frakh,\frakU^{\frakc\frakh}}(M;H,J;\widetilde{\bbL})\arrow[dl,"\pi_\frakb",swap]\arrow[dr,"\pi_\hbar"] & \\
    CF^*_{\frakc\frakh,\frakU^{\frakc\frakh}}(M;H,J;\widetilde{\bbL})\otimes_{\widetilde{\bbL}} R_{n-1} & & CF^*_{\frakc\frakh,\frakU^{\frakc\frakh}}(M;H,J;\widetilde{\bbL})\otimes_{\widetilde{\bbL}}\bbL.
    \end{tikzcd}
    \end{equation}
Moreover, we have that the previous diagram, after passing to homology, is an isomorphism of modules over 
    \begin{equation}
    R_{n-1}\cong H_*(\widetilde{\bbL})\cong\bbL.
    \end{equation}

\begin{proof}[Proof of Proposition \ref{prop:quasi-isomorphismbd}]
Tautologically, we have the following equality of cochain complexes over $R_{n-1}$:
    \begin{equation}
    CF^*_{\frakc\frakh,\frakU^{\frakc\frakh}}(M;H,J;\widetilde{\bbL})\otimes_{\widetilde{\bbL}}R_{n-1}=CF^*_{\frakc\frakh}(M;H,J;R_{n-1});
    \end{equation}
this is because the tensor product on the left hand side has the effect of setting $\frakL_\rho$ and $\hbar_\rho$ to zero, and then the codifferentials on the left and right are, by definition, the same on $\frakb_\rho$ terms. Moreover, tautologically, we have the following equality of cochain complexes over $\bbL$:
    \begin{equation}
    CF^*_{\frakc\frakh,\frakU^{\frakc\frakh}}(M;H,J;\widetilde{\bbL})\otimes_{\widetilde{\bbL}}\bbL=CF^*_{\frakU^{\frakc\frakh}}(M;H,J;\bbL);
    \end{equation}
this is because the tensor product on the left hand side has the effect of setting  $\frakb_\rho$ and $\frakL_\rho$ to zero, and then the codifferentials on the left and right are, by definition, the same on $\hbar_\rho$ terms. In particular, using \eqref{eqn:quasi-isooverc} and the fact that we can invert quasi-isomorphism of cochain complexes over $\bbC$, we obtain a quasi-isomorphism of cochain complexes over $\bbC$,
    \begin{equation}
    CF^*_{\frakc\frakh}(M;H,J;R_{n-1})\to CF^*_{\frakU^{\frakc\frakh}}(M;H,J;\bbL),
    \end{equation}
which, after passing to homology, is an isomorphism of modules over
    \begin{equation}
    R_{n-1}\cong\bbL.
    \end{equation}

The statement about compatibility with appropriate Floer continuation maps follows because, in a straightforward manner, we may repeat the discussion of the previous subsection for Floer continuation maps;  in particular, we obtain a diagram of quasi-isomorphisms of cochain complexes over $\bbC$ of the form
    \begin{equation}
    \begin{tikzcd}[column sep=-5ex]
    & CF^*_{\frakc\frakh,\frakU^{\frakc\frakh}}(M;H,J;\widetilde{\bbL})\arrow[dl,"\pi_\frakb",swap]\arrow[dr,"\pi_\hbar"]\arrow[dd] & \\
    CF^*_{\frakc\frakh,\frakU^{\frakc\frakh}}(M;H,J;\widetilde{\bbL})\otimes_{\widetilde{\bbL}} R_{n-1}\arrow[dd] & & CF^*_{\frakc\frakh,\frakU^{\frakc\frakh}}(M;H,J;\widetilde{\bbL})\otimes_{\widetilde{\bbL}}\bbL\arrow[dd] \\
    & CF^*_{\frakc\frakh,\frakU^{\frakc\frakh}}(M;H',J';\widetilde{\bbL})\arrow[dl,"\pi_\frakb",swap]\arrow[dr,"\pi_\hbar"] & \\
    CF^*_{\frakc\frakh,\frakU^{\frakc\frakh}}(M;H',J';\widetilde{\bbL})\otimes_{\widetilde{\bbL}} R_{n-1} & & CF^*_{\frakc\frakh,\frakU^{\frakc\frakh}}(M;H',J';\widetilde{\bbL})\otimes_{\widetilde{\bbL}}\bbL
    \end{tikzcd}
    \end{equation}
which, after passing to homology, is an isomorphism of modules over 
    \begin{equation}
    R_{n-1}\cong H_*(\widetilde{\bbL})\cong\bbL.
    \end{equation}
\end{proof}

\begin{proof}[Proof of Theorem \ref{thm:main}]

    Follows by combining Corollary \ref{cor:firpg-iegn}, Theorem \ref{thm:tech}, Lemma \ref{lem: piorehjgpowehgpih}, and Proposition \ref{prop:quasi-isomorphismbd}.
\end{proof}

\section{Base change}
\subsection{Spectral sequences}
This subsection is dedicated to studying the change-of-coefficient spectral sequence for the DG-modules studied throughout the paper and proving that many of the quasi-isomorphisms constructed have associated isomorphisms of spectral sequences.

Recall the standard $t$-structure on cochain complexes:
    \begin{equation}
        \tau^{\geq i} C^* \equiv \begin{cases}
            C^* & \;\textrm{ if }* > i \\
            \coker(d^{i-1}) & \; \textrm{ if }*=i\\
            0 & \; \textrm{ if }*<i,
        \end{cases}
        \;\;
        \tau^{\leq i} C^* \equiv \begin{cases}
            0 & \;\textrm{ if }* > i \\
            \ker(d^{i-1}) & \; \textrm{ if }*=i\\
            C^* & \; \textrm{ if }*<i.
        \end{cases} 
    \end{equation}

Let $S$ be a DGA and $Y$ a DG $S$-module; we will assume that both are cohomologically graded. Assume that $S$ is non-positively graded with $S^0=\bbC$ and $Y$ is semi-free.\footnote{Alternatively, we may consider derived tensor products.} There is a decreasing filtration $F^*Y\equiv\{F^iY\}_{i\in\bbZ}$ on $Y$ in the derived category given by setting 
    \begin{equation}
    F^iY\equiv\tau^{\leq -i}S\otimes_SY; 
    \end{equation}
this induces a spectral sequence with first page
    \begin{equation}
        E_1^{p,q} = H^{p+q}\left(\mathrm{Cone}(F^{p+1}Y \to F^pY)\right) \simeq H^{p+q}\big(H^{-p}(S) \otimes_S Y [p]\big),
    \end{equation}
cf. \cite[12.24]{STK}. We call this the \emph{change-of-coefficients spectral sequence} associated to the augmentation $S \to S^0=\bC$. Note, this converges if $Y$ is a perfect $S$-module or a directed colimit of perfect $S$-modules. Also, applying the reversed $t$-structure yields an analogous spectral sequence in the homologically graded case; this utilizes the analogous filtration 
    \begin{equation}
    F_iY\equiv\tau_{\geq-i}S\otimes_SY.
    \end{equation}

\subsubsection{Obstructions to base change}
    \begin{lem}\label{lem:71}
        Let $Z$ be a cochain complex over $\bbC$. The change-of-coefficient spectral sequence for the $S$-module $S \otimes_\bC Z$ degenerates at the $E_2$-page.
    \end{lem}
    \begin{proof}
    Observe, the filtered complex $F^*(S \otimes_\bC Z)$ is equivalent to $(F^*S) \otimes_\bC Z$. Since $Z$ is free over $\bC$, we find that (from the $E_2$-page onwards) the spectral sequence for $S \otimes_\bC Z$ is determined by the spectral sequence for $S$ via tensoring with $H_*(Z)$; the result follows since the spectral sequence for $S$ has no differentials.
    \end{proof}
    Let $\flow^{\rm fr}$ be the stable $\infty$-category of framed flow categories, cf. \cite{AB24}. This is defined identically to $\flow^{\rm decU}$ in Definition  \ref{def: dec cx flow simp}, except all the bundles $I(x,y)$ are required to be trivial, i.e., simply complex vector spaces.\footnote{In \cite{AB24}, they allow trivial real vector spaces, and do not require the data of a decoration, but these does not change the equivalence class of the ensuing category by the same argument as Lemmas \ref{lem:obviousforgetful} and \ref{L219}.} In particular, we have an obvious inclusion functor 
        \begin{equation}
        \iota:\flow^{\rm fr} \to \flow^{\rm decU}.
        \end{equation}
    \begin{lem}\label{lem: rwipgjepign}
    For any $\bbX\in \flow^{\rm fr}$, we have that
        \begin{equation} 
            CM_*\Big(\iota(\bbX),\exp\big(\frakc\frakh(I)\big);{\bC[\underline b]}\Big) \simeq CM_*(\bbX;\bbC)\otimes_\bC {\bC[\underline b]}.
        \end{equation}
    \end{lem}
    \begin{rem}
    In particular, the change-of-coefficient spectral sequence for 
        \begin{equation}
        CM_*\Big(\iota(\bbX),\exp\big(\frakc\frakh(I)\big);{\bC[\underline b]}\Big) \in \Mod_{\bC[\underline b]}
        \end{equation}
    (1) converges and (2) degenerates at the $E_2$-page.
    \end{rem}
    \begin{proof}
        Since all the bundles $I(x,y)$ are trivial for $\bbX$, we may ensure the Hermitian connections chosen in Section \ref{sec: chern revisited} are flat for these bundles; in particular, all Chern-Weil forms vanish except for the $\ch_0$ term. This implies the only non-zero term in \eqref{eqn:twistedcodiff} is the usual Morse differential for $\bbX$, i.e., there are no $b_\rho^k$-terms for $k>0$.
    \end{proof}

    Given a complex-oriented flow category $\bbX$, we denote by $\frakX$ its associated ${\rm MU}$-module under the equivalence $\flow^{\rm U} \simeq \Mod_{{\rm MU}}$.
    
    \begin{cor}\label{cor; poeirjgpiengiegnpir}
    Let $\bbX\in\flow^{\rm U}$. Assume the change-of-coefficient spectral sequence for $CM_*\Big(\bbX,\exp\big(\frakc\frakh(I)\big);{\bC[\underline b]}\Big)$ is non-degenerate, then 
        \begin{equation}
        \frakX \not\simeq \frakY \otimes_\bS {\rm MU}
        \end{equation}
    for any spectrum $\frakY$.
    \end{cor}
    \begin{proof}
        Under the given hypotheses, Lemmas \ref{lem:71} and \ref{lem: rwipgjepign} imply that $\bbX$ cannot possibly be in the essential image of $\iota$. We have a commutative square of stable $\infty$-categories:
        \begin{equation}
            \begin{tikzcd}
                \flow^{\rm fr}
                \arrow[r, "\simeq"]
                \arrow[d, "\iota"]
                &
                \Mod_\bS = \mathrm{Sp} 
                \arrow[d, "- \otimes_\bS {\rm MU}"]
                \\
                \flow^{\rm U}
                \arrow[r, "\simeq"]
                &
                \Mod_{{\rm MU}}
            \end{tikzcd}
        \end{equation}
        Here, the horizontal arrows are the equivalences provided by \cite[Theorems 0.0.1]{HO} and \cite[Theorem 0.7]{HM}, and commutativity follows from the naturality statements in both theorems; the result follows.
    \end{proof}

\subsubsection{The $E_2$-page}\label{sec:ofrejgobtrbrptgne}
We consider $\bC[\hbar]$ and we assume $Y$ is of the form arising in bulk-deformation theory (and twisted Morse theory). More precisely, we assume 
    \begin{equation}
    Y\cong\bigoplus_{x\in\bbX}\bC[\hbar]\abs{\frako_x}
    \end{equation}
as $\bC[\hbar]$-modules, for $\bbX$ some set of formal generators with degree denoted $\deg(x)\in\bbZ$, such that the codifferential is of the form $d\equiv d^{(1)}+d^{(2)}+\cdots$, where each term $d^{(\rho)}$ has non-zero $\abs{\frako_y}-\abs{\frako_x}$ coefficient if and only if $\deg(x)-\deg(y)=\rho$. Note, $\overline{Y}\equiv Y\otimes_{\bC[\hbar]}\bbC$ is given by the cochain complex 
    \begin{equation}
    \bigoplus_{x\in\bbX}\bbC\abs{\frako_x}
    \end{equation}
with codifferential $d^{(1)}$. In this setting, the differential on the $E_1$-page of the change-of-coefficients spectral sequence is $d^{(1)}$, and we may compute the $E_2$-page to be
    \begin{equation}\label{eq: oidwhgepwughpoge}
        E_2^{p,q} \cong H^{-p}(\bC[\hbar]) \otimes_\bC H^{2p+q}(\overline Y).
    \end{equation}
The next differential is given by the next non-zero $d^{(\rho)}$ acting on the $E_{\rho}$-page.

\subsubsection{Compatibility with base change}
In the case of bulk-deformations, we would like to study the behavior of the spectral sequence under setting some $\hbar_\rho$'s to zero; we consider an abstraction of this situation.

Let $\phi: S \to S'$ be a map of DGAs, with both $S$ and $S'$ satisfying the condition that they are non-positively graded and given by $\bC$ in degree 0. For each $i$, we have an induced map
    \begin{equation}
    \tau^{\leq -i}S \to \tau^{\leq -i}S'
    \end{equation}
of $S$-modules. In particular, for any DG $S$-module $Y$, there is a map of filtered chain complexes:
    \begin{equation}
        F^i_SY = \tau^{\leq -i}S \otimes^\bL_S Y \to \tau^{\leq -i}S' \otimes^\bL_S Y = \tau^{\leq -i}S' \otimes^\bL_{S'} S' \otimes^\bL_S Y = F^i_{S'} (S' \otimes^\bL_S Y),
    \end{equation}
    where the subscript on the filtration denotes which DGA we take the filtration with respect to. In particular, we obtain the following result.
    \begin{lem}\label{lem: rpihgeoinhgpioebgp}
        There is a map of spectral sequences from the change-of-coefficient spectral sequence for $Y$ as an $S$-module to the change-of-coefficient spectral sequence for $S' \otimes_S^\bL Y$ as an $S'$-module.
    \end{lem}

    From our computation of the $E_2$-page, i.e., \eqref{eq: oidwhgepwughpoge}, we have the following result.
    \begin{cor}\label{cor:foriwengoienbp}
        Assume $Y$ is of the form considered in Part \ref{sec:ofrejgobtrbrptgne} and that $H^*(\phi)$ is surjective, then the map of change-of-coefficient spectral sequences induced by $\phi$ is surjective on the $E_2$-page. In particular, if the change-of-coefficient spectral sequence for $S' \otimes^\bL_S Y$ is non-degenerate, then the change-of-coefficient spectral sequence for $Y$ is non-degenerate.
    \end{cor}

\subsubsection{Compatibility with autoequivalences}
    Fortunately, the isomorphism of Lemma \ref{lem: rioehgeoubgoeshrg} is compatible with these change-of-coefficient spectral sequences, in the following sense. 
    
    \begin{warning}
    Note, in the subsequent lemma and proof, we work with homological gradings.
    \end{warning}
    
    \begin{lem}
        Let $M$ be the ${\bC[\underline b]}$-${\bC[\underline b]}$ bimodule from Section \ref{sec: oeirhngouebgoue} and $Y$ an ${\bC[\underline b]}$-module that is a directed colimit of perfect ${\bC[\underline b]}$-modules. The change-of-coefficient spectral sequences for $Y$ and $M \otimes_{\bC[\underline b]}^\bL Y$ are isomorphic (from the $E_1$-page onwards).
    \end{lem}
    \begin{proof}
        We may assume that $Y$ is perfect. It suffices to prove there is a quasi-isomorphism of filtered chain complexes: 
            \begin{equation}
            F_*Y \simeq F_*(M \otimes_{\bC[\underline b]}^\bL Y).
            \end{equation}
        Recall, $M \simeq {\bC[\underline b]}$ as right ${\bC[\underline b]}$-modules, so $Y$ and $M \otimes^\bL_{\bC[\underline b]} Y$ are already quasi-isomorphic as chain complexes. Moreover, since $M \simeq {\bC[\underline b]}$ as left ${\bC[\underline b]}$-modules, we have that   
            \begin{equation}
            \tau_{\geq i}{\bC[\underline b]} \otimes_{\bC[\underline b]}^\bL M \simeq \tau_{\geq i}M \simeq \tau_{\geq i} {\bC[\underline b]}
            \end{equation}
        as right ${\bC[\underline b]}$-modules. These equivalences are natural in $i$, whence the claim.
    \end{proof}

    \begin{cor}\label{cor:ornbpihrnhbpirtgpnr}
        The change-of-coefficient spectral sequences for the chain complexes underlying the two sides of \eqref{eq: rwigpiweg} are isomorphic.
    \end{cor}
    \begin{rem}
    Corollary \ref{cor:ornbpihrnhbpirtgpnr} would follow from the conjectural quasi-isomorphism of Remark \ref{rem:eopjfpwnwGN}, but does not follow directly from Theorem \ref{thm:main}.
    \end{rem}

\subsection{Operations on symplectic cohomology}
\subsubsection{Piunikhin-Salamon-Schwarz map}
We briefly recall the construction of the PSS map. Let 
    \begin{equation}
    \Sigma\equiv\bbC P^1-\{0\}
    \end{equation}
be equipped with a negative cylindrical end $\epsilon^-$. Consider admissible Floer data $(H,J)$ and a Morse-Smale pair $(f,g)$. Let $\eta\in\Omega^1(\Sigma)$ be a 1-form which (1) vanishes in a neighborhood of $+\infty$ and (2) restricts to $dt$ at $\epsilon^-$. Let $J_\Sigma\equiv\{J_z\}_{z\in\Sigma}$ be a $\Sigma$-dependent $\omega$-tame almost complex structure which (1) is $z$-independent near $+\infty$ and (2) satisfies $J_\Sigma\vert_{\epsilon^-}=J$. 

Given any $x\in\chi(H)$, we consider the moduli space $\scrA_x$ of maps $u:\Sigma\to M$ satisfying
    \begin{equation}
    \begin{cases}
    (du-X_H\otimes\eta)^{0,1}=0 \\
    \lim_{s\to-\infty}u\big(\epsilon^-(s,t)\big)=x(t).
    \end{cases}   
    \end{equation}  
We have a natural evaluation map
    \begin{equation}
    \eval_{+\infty}:\scrA_x\to M.
    \end{equation}
In particular, given any $b\in\crit(f)$, we may consider the fiber product
    \begin{equation}
    \scrA_{bx}\equiv\scrA_x\times_M W^s(b);
    \end{equation}
this is, for generic data, a smooth manifold of dimension $\deg(x)-I(x)$ which admits a Gromov-compactification $\bbA_{bx}$ given by allowing breakings at Hamiltonian orbits and critical points. In the usual way, we obtain a cochain map 
    \begin{equation}
    {\rm PSS}:CM^*(M;\bbZ)\to CF^*(M;H,J;\bbZ)
    \end{equation}
whose $\abs{\frako_b}-\abs{\frako_x}$ component is given by 
    \begin{equation}
    {\rm PSS}_{b,x}\equiv\sum_{\substack{u\in\bbA_{bx} \\ \deg(x)=I(a)}}\mu_u,
    \end{equation}
where $\mu_u:\frako_b\xrightarrow{\sim}\frako_x$ is the isomorphism on orientation lines induced from the product orientation on $\bbA_{bx}$. This is compatible with continuation maps by the usual argument; hence, we obtain the PSS map:
    \begin{equation}
    {\rm PSS}:HM^*(M;\bbZ)\to SH^*(M;\bbZ).
    \end{equation}
    
\subsubsection{Lie bracket}
Now, we define the Lie bracket on symplectic cohomology loosely following \cite[Section 2.5.1]{Abo14}. Let 
    \begin{equation}
    \widetilde{\Sigma}\equiv\bbC P^1-\{0,+\infty\}\cong\bbC-\{0\},
    \end{equation}
where the identification equips a neighborhood of 0 with a negative cylindrical end $\epsilon^-$ and a neighborhood of $+\infty$ with a positive cylindrical end $\epsilon^+_1$. Let $z_\theta\equiv e^{i\theta}\in\widetilde{\Sigma}$, $z_\theta\in S^1$, be an additional marked point on $\widetilde{\Sigma}$ and consider 
    \begin{equation}
    \Sigma_\Theta\equiv\big\{\widetilde{\Sigma}-\{z_\theta\}\big\}_{z_\theta\in S^1},
    \end{equation}
where we equip the family $\Sigma_\Theta$ with a $\theta$-family of positive cylindrical ends $\epsilon^+_2\equiv\{\epsilon^+_{z_\theta}\}_{z_\theta\in S^1}$, which is canonical up to homotopy, such that each $\epsilon^+_{z_\theta}$ is a positive cylindrical end in a neighborhood of $z_\theta$. Consider admissible Floer data $(H_-,J_-)$ and $(H_i,J_i)$, $i=1,2$, whose slopes satisfy $\kappa_1+\kappa_2\leq\kappa_-$. Let $\eta=\Omega^1(\Sigma_\Theta)$ be a $\theta$-family of 1-forms which restricts to $dt$ at $\epsilon^-$ and $\epsilon^+_i$. Let $H_{\Sigma_\Theta}\equiv\{H_{\theta,z}\}_{z\in\Sigma_\theta}$ be a $\theta$-family of $\Sigma_\theta$-dependent Hamiltonians satisfying $H_{\theta,\epsilon^-}=H_-$ and $H_{\theta,\epsilon^+_i}=H_i$. Let $J_{\Sigma_\Theta}\equiv\{J_{\theta,z}\}_{z\in\Sigma_\theta}$ be a $\theta$-family of $\Sigma_\theta$-dependent $\omega$-tame almost complex structures which satisfies $J_{\Sigma_\Theta}\vert_{\epsilon^-}=J_-$ and $J_{\Sigma_\Theta}\vert_{\epsilon^+_i}=J_i$.

Given any three $x_-\in\chi(H_-)$ and $y_i\in\chi(H_i)$, we consider the moduli space $\scrB_{y_i;x_-}$ of pairs $(\theta,u_\theta)$, where $u_\theta:\Sigma_\theta\to M$ is a map satisfying
    \begin{equation}
    \begin{cases}
    (du-X_{H_{\Sigma_\Theta}}\otimes\eta)^{0,1}=0 \\
    \lim_{s\to-\infty}u\big(\epsilon^-(s,t)\big)=x_-(t) \\
    \lim_{s\to+\infty}u\big(\epsilon^+_1(s,t)\big)=y_1(t) \\
    \lim_{s\to+\infty}u\big(\epsilon^+_{z_\theta}(s,t)\big)=y_2(t);
    \end{cases}   
    \end{equation}  
this is, for generic data, a smooth manifold of dimension $\deg(x_-)-\deg(y_1)-\deg(y_2)+1$ which admits a Gromov-compactification $\bbB_{y_i;x_-}$ given by allowing breakings at Hamiltonian orbits. In the usual way, we obtain a cochain map
    \begin{equation}
    [-,-]:CF^*(M;H_1,J_1;\bbZ)\otimes_\bbZ CF^*(M;H_2,J_2;\bbZ)\to CF^*(M;H_-,J_-;\bbZ)
    \end{equation}
whose $\abs{\frako_{y_1}}\otimes_\bbZ\abs{\frako_{y_2}}-\abs{\frako_{x_-}}$ component is given by 
    \begin{equation}
    [-,-]_{y_i;x_-}\equiv\sum_{\substack{u\in\bbB{y_i;x_-} \\ \deg(x_-)=\deg(y_1)+\deg(y_2)-1}}\mu_u,
    \end{equation}
where $\mu_u:\frako_{y_1}\otimes_\bbZ\frako_{y_2}\xrightarrow{\sim}\frako_{x_-}$ is the isomorphism on orientation lines induced from $u$. This is compatible with continuation maps by the usual argument; hence, we obtain the Lie bracket:
    \begin{equation}
    [-,-]:SH^*(M;\bbZ)\otimes_\bbZ SH^*(M;\bbZ)\to SH^*(M;\bbZ).
    \end{equation}
\cite[Lemma 2.5.5]{Abo14} shows the Lie bracket is given by the pair-of-pants product and the BV operator.

\subsubsection{The composition}
Since this will be needed in the sequel, we now take a brief detour to explicitly describe the composition 
    \begin{equation}\label{eqn:comp1}
    SH^*(M;\bbZ)\otimes_\bbZ HM^*(M;\bbZ)\xrightarrow{{\rm Id}\otimes_\bbZ{\rm PSS}} SH^*(M;\bbZ)\otimes_\bbZ SH^*(M;\bbZ)\xrightarrow{[-,-]} SH^*(M;\bbZ).
    \end{equation}
Again, consider $\widetilde{\Sigma}$ together with an additional marked point $z_\theta\equiv e^{i\theta}\in\widetilde{\Sigma}$, $z_\theta\in S^1$. We define 
    \begin{equation}
    \Sigma_{\Theta}'\equiv\big\{(\widetilde{\Sigma},z_\theta)\big\}_{z\in S^1}.
    \end{equation}
In particular, the family $\Sigma_{\Theta}'$ differs from the family $\Sigma_{\Theta}$; the former does not remove the marked point while the latter does. Consider admissible Floer data $(H_-,J_-)$ and $(H_1,J_1)$ whose slopes satisfy $\kappa_1\leq\kappa_-$. Moreover, consider a Morse-Smale pair $(f,g)$. Let $\eta=\Omega^1(\Sigma_\Theta')$ be a $\theta$-family of 1-forms which (1) vanishes in a neighborhood of $z_\theta$ and (2) restricts to $dt$ at $\epsilon^-$ and $\epsilon^+_1$. Let $H_{\Sigma_\Theta'}\equiv\{H_{\theta,z}\}_{z\in\widetilde{\Sigma}}$ be a $\theta$-family of $\widetilde{\Sigma}$-dependent Hamiltonians satisfying $H_{\theta,\epsilon^-}=H_-$ and $H_{\theta,\epsilon^+_1}=H_1$. Let $J_{\Sigma_\Theta'}\equiv\{J_{\theta,z}\}_{z\in\widetilde{\Sigma}}$ be a $\theta$-family of $\widetilde{\Sigma}$-dependent $\omega$-tame almost complex structures which (1) is $z$-independent near $z_\theta$ and (2) satisfies $J_{\Sigma_\Theta'}\vert_{\epsilon^-}=J_-$ and $J_{\Sigma_\Theta'}\vert_{\epsilon^+_1}=J_1$.

Given any two $x_-\in\chi(H_-)$ and $y_1\in\chi(H_1)$, we consider the moduli space $\scrD_{y_1;x_-}$ of pairs $(\theta,u_\theta)$, where $u_\theta:\widetilde{\Sigma}\to M$ is a map satisfying
    \begin{equation}
    \begin{cases}
    (du-X_{H_{\Sigma_\Theta'}}\otimes\eta)^{0,1}=0 \\
    \lim_{s\to-\infty}u\big(\epsilon^-(s,t)\big)=x_-(t) \\
    \lim_{s\to+\infty}u\big(\epsilon^+_1(s,t)\big)=y_1(t).
    \end{cases}   
    \end{equation}  
We have a natural evaluation map
    \begin{equation}
    \eval:\scrD_{y_1;x_-}\to M.
    \end{equation}
In particular, given any $b\in\crit(f)$, we may consider the fiber product
    \begin{equation}
    \scrD_{y_1,b;x_-}\equiv\scrD_{y_1;x_-}\times_M W^s(b);
    \end{equation}
this is, for generic data, a smooth manifold of dimension $\deg(x_-)-\deg(y_1)-I(b)+1$ which admits a Gromov-compactification $\bbD_{y_1,b;x_-}$ given by allowing breakings at Hamiltonian orbits and critical points. In the usual way, we obtain a cochain map
    \begin{equation}
    CF^*(M;H_1,J_1;\bbZ)\otimes_\bbZ CM^*(M;\bbZ)\to CF^*(M;H_-,J_-;\bbZ)
    \end{equation}
whose $\abs{\frako_{y_1}}\otimes_\bbZ\abs{\frako_{b}}-\abs{\frako_{x_-}}$ component is given by 
    \begin{equation}
    \sum_{\substack{u\in\bbD_{y_1,b;x_-} \\ \deg(x_-)=\deg(y_1)+I(b)-1}}\mu_u,
    \end{equation}
where $\mu_u:\frako_{y_1}\otimes_\bbZ\frako_b\xrightarrow{\sim}\frako_{x_-}$ is the isomorphism on orientation lines induced from the product orientation on $\bbD_{y_1,b;x_-}$. This is compatible with continuation maps by the usual argument; hence, we obtain a map
    \begin{equation}\label{eqn:comp2}
    \bG:SH^*(M;\bbZ)\otimes_\bbZ HM^*(M;\bbZ)\to SH^*(M;\bbZ).
    \end{equation}
    
\begin{lem}\label{lem:78}
The composition \eqref{eqn:comp1} equals \eqref{eqn:comp2}.
\end{lem}

\begin{proof}
The lemma follows in the standard way by using interpolating moduli spaces, interpolating from the gluing of the data used in the composition \eqref{eqn:comp1} to the data used in \eqref{eqn:comp2}, to show the cochain maps underlying \eqref{eqn:comp1} and \eqref{eqn:comp2} are in fact cochain homotopic.
\end{proof}

\subsection{Spectral sequence from bulk-deformations}
We are now ready to compute the first allowable differential for the change-of-coefficients spectral sequence for bulk-deformed symplectic cohomology.

    \begin{lem}
        There is an isomorphism of change-of-coefficient spectral sequences (from the $E_2$-pages onwards) between the two sides of \eqref{eqn:quasi-isomorphismbd}
    \end{lem}
    \begin{proof}
        Follows from Lemma \ref{lem: rpihgeoinhgpioebgp} applied to the zig-zag \eqref{eqn:quasi-isooverc}.
    \end{proof}

For a moment, we restrict to bulk-deformations in a single cohomology class. Let $\bL' \equiv \bC[\hbar]$ and $\frakU$ a Morse cocycle of degree $\ell\geq3$ (in particular, $\deg(\hbar)=2-\ell$); we use notation as in Subsection \ref{subsec:bulkdeformed}. Consider the linear map $\delta^{(1)}_\frakU$, defined to be the $\hbar^{\frakq}$-term of \eqref{eq: reihgeropiubgeopubrg}, where $q=1$ and $\frakq$ is simply the partition $\{1\}$ of $[1]$. Now, by identifying $S^1_\theta$ and the family of Riemann surfaces $\Sigma_\Theta$ over it with $\cC_{1+2}$ and the family of Riemann surfaces $\cC_{1+2}^{univ}$ over it, making sure our choices of perturbation data agree under this identification, we see that $\bG\big(-, \mathrm{PSS}([\frakU])\big)$ (before passing to the colimit using Floer continuation maps) is precisely $\delta^{(1)}_\frakU$.

    \begin{lem}\label{lem: riogeoingiengpr}
        For $2 \leq r \leq \ell-2$, the differential $d_r$ on the $E_r$-page of the change-of-coefficient spectral sequence for $CF^*_\frakU(M; H,J; \bL')$ vanishes. The differential $d_{\ell-1}$ on the $E_{\ell-1}$-page is given by (a direct sum of copies of)
        \begin{equation}
            \big[-, [\frakU]\big]: HF^*(M; H, J;\bbC) \to HF^{*+\ell-1}(M; H, J;\bbC).
        \end{equation}
    \end{lem}
    \begin{proof}
        Follows from Lemma \ref{lem:78} combined with the description of the first allowable differential from Part \ref{sec:ofrejgobtrbrptgne}.
    \end{proof}
    \begin{proof}[Proof of Corollary \ref{cor:main}]
        Let $\bL' \equiv \bC[\hbar_{\rho_0}]$ and consider the map $f: \bL \to \bL'$ obtained via setting each $\hbar_\rho$, $\rho\neq\rho_0$, to zero. By construction, we have an equality of cochain complexes over $\bL'$:
        \begin{equation}
            SC^*_{\frakU^{\frakc\frakh}}(M; \bL) \otimes_\bL \bL' \cong SC^*_{\frakU^{\frakc\frakh}_{\rho_0}}(M; \bL');
        \end{equation}
        note, the left hand side does in fact computed the derived tensor product since $SC^*_{\frakU^{\frakc\frakh}}(M; \bL)$ is a directed colimit of perfect $\bbL$-modules. By Corollary \ref{cor:foriwengoienbp}, there is a surjection $g$ of spectral sequences from the change-of-coefficient spectral sequence for $SC^*_{\frakU^{\frakc\frakh}}(M; \bL)$ to the change-of-coefficient spectral sequence for $SC^*_{\frakU^{\frakc\frakh}_{\rho_0}}(M; \bL')$. By combining \eqref{eq:18} with Lemma \ref{lem: riogeoingiengpr}, the target of $g$ has a non-zero differential; hence, so does the domain. Using Corollary \ref{cor:ornbpihrnhbpirtgpnr} followed by Corollary \ref{cor; poeirjgpiengiegnpir} finishes the proof. 
    \end{proof}

\subsection{Non-base change computation}\label{subsec:nonbasechange}
In this subsection, we compute Example \ref{example:nonbasechange}; we repeat it here for convenience.

\begin{claim}
The hypothesis of Corollary \ref{cor:main} holds in the case of $X=T^*\bbC P^n$, $n\in\bbZ$ is odd and at least 3.
\end{claim}

\begin{proof}[Proof of claim]
We may check this as follows. Let $H\in H^2(\bbC P^n;\bbZ)$ be the standard generator corresponding to the Poincar\'e dual of the class of a hyperplane. Recall, in this case, $\bbC P^n$ satisfies $c_1(\bbC P^n)=(n+1)H\in H^2(\bbC P^n;\bbZ)$ which implies $w_2(\bbC P^n)=0\in H^2(\bbC P^n;\bbZ/2)$, i.e., $\bbC P^n$, in this case, is spin. Now, a classical theorem of Viterbo \cite{Vit98} shows 
    \begin{equation}
    SH^*(T^*\bbC P^n;\bbZ)\cong H_{2n-*}(\calL\bbC P^n;\bbZ),
    \end{equation}
where $\calL\bbC P^n$ is the free loop space of $\bbC P^n$. As already mentioned, the left hand side of the Viterbo isomorphism can be endowed with a pair-of-pants product and BV operator, turning it into a BV algebra. Meanwhile, the right hand side of the Viterbo isomorphism can be endowed with the Chas-Sullivan loop product and a BV operator, turning it into a BV algebra. Finally, the Viterbo isomorphism is an isomorphism of BV algebras, cf. \cite{Abo14}. In particular, since the product and BV operator determine the bracket, the Lie bracket $[-,-]$ on ordinary symplectic cohomology of a cotangent bundle, under the Viterbo isomorphism, corresponds to the string bracket $\{-,-\}$ on the homology of the free loop space of the base. In \cite{Hep09}, the string bracket for $\calL\bbC P^n$ was computed:
    \begin{equation}
    H_{*-2n}(\calL\bbC P^n;\bbZ)\cong\dfrac{\Lambda_\bbZ[w]\otimes_\bbZ\bbZ[H,v]}{\big\langle H^{n+1},(n+1)H^n\cdot v,w\cdot H^n\big\rangle},
    \end{equation}
where $\deg(w)=1$, $\deg(H)=2$, and $\deg(v)=-2n$; and the only non-trivial string brackets are
    \begin{align}
    \{H,w\}=-\{w,H\}&=-H, \\
    \{v,w\}=-\{w,v\}&=(n+1)v+\binom{n+1}{2}H^n\cdot v^2.
    \end{align}
It remains to verify the hypotheses of Corollary \ref{cor:main}, but this is straightforward. First, we recall 
    \begin{equation}
    c_1(T^*\bbC P^n)=c_1(\bbC P^n)-c_1(\bbC P^n)=0.
    \end{equation}
Second, we recall
    \begin{equation}
    c_2(T^*\bbC P^n)=c_2(\bbC P^n)-c_1(\bbC P^2)^2+c_2(\bbC P^n)=2\binom{n+1}{2}H^2-(n+1)^2H^2.
    \end{equation}
Third, we compute 
    \begin{align}
    \ch_1(T^*\bbC P^n)&=0 \\
    \ch_2(T^*\bbC P^n)&=\dfrac{1}{2}\big(c_1(T^*\bbC P^n)^2-2c_2(T^*\bbC P^n)\big) \\
    &=-2\binom{n+1}{2}H^2+(n+1)^2H^2.
    \end{align}
Fourth, we compute 
    \begin{equation}
    \{H^2,w\}=-\{w,H^2\}=-\{w,H\}\cdot H-H\cdot\{w,H\}=-2H^2
    \end{equation}
using the Poisson identity for the string bracket.
Finally, we verify
    \begin{equation}
    \big\{\ch_2(T^*\bbC P^n),w\big\}=-2\binom{n+1}{2}\{H^2,w\}+(n+1)^2\{H^2,w\}=-2\ch_2(T^*\bbC P^n);
    \end{equation}
i.e., $\big\{\ch_2(T^*\bbC P^n),w\big\}$ does not vanish when $n\in\bbZ$ is odd and at least 3.
\end{proof}

\subsection{Ample divisor complement computation}\label{subsec:ampledc}
In this subsection, we prove the two claims from the introduction; we repeat them here for convenience.

\begin{claim}
In the case $n\in\bbZ$ is odd and at least 3, $\calG\calW^{\rm U}_{T^*\bbC P^n}\neq0$.
\end{claim}

\begin{proof}[Proof of claim]
If the conclusion were not true, then $\bbF^{\rm U}_{T^*\bbC P^n}$ would degenerate into its associated graded; this is an immediate contradiction to Example \ref{example:nonbasechange} since the wedge over the associated graded can be lifted to the sphere spectrum.
\end{proof}

\begin{claim}
In the case $n\in\bbZ$ is odd and at least 3, $\calG\calW^{\rm U}_{T^*\bbC P^n}\otimes_{\rm MU}\bbZ=0$.
\end{claim}

\begin{proof}[Proof of claim]
In order to show $\calG\calW^{\rm U}_{T^*\bbC P^n}\otimes_{\rm MU}\bbZ=0$, it suffices to show its associated morphism, 
    \begin{equation}
    \frakD\Sigma^{-2n+1}(S_DZ)_+\otimes_\bbS\bbZ\to\frakD\bbC P^n_+\otimes_\bbS\bbZ,
    \end{equation}
is (1) trivial on homotopy groups and (2) trivial on homotopy groups after base change along $\bbZ\to\bbZ/m$. We proceed as follows.

We have that the total space of the unit sphere bundle of $\bbC P^n$ is diffeomorphic to the total space of $S_DZ$, essentially by construction. Now, using \cite[Theorem 6.1]{BO18}, we have 
    \begin{equation}
    H_*(\calL\bbC P^n;\bbZ)\cong H_*(\bbC P^n;\bbZ)\oplus\bigoplus_{k\geq1}H_{*+2n-1-2nk}(S_DZ;\bbZ).
    \end{equation}
In particular, since $\bbC P^n$ is spin by assumption ($n\in\bbZ$ is odd and at least 3), the usual Viterbo isomorphism, combined with the previous equation and Poincar\'e duality, shows 
    \begin{align}
    SH^{-*}(T^*\bbC P^n;\bbZ)&\cong H_{*+2n}(\calL\bbC P^n;\bbZ) \\
    &\cong H_{*+2n}(\bbC P^n;\bbZ)\oplus\bigoplus_{k\geq1}H_{*+4n-1-2nk}(S_DZ;\bbZ) \\
    &\cong H^{-*}(\bbC P^n;\bbZ)\oplus\bigoplus_{k\geq1}H^{-*+2nk}(S_DZ;\bbZ).
    \end{align}
Meanwhile, we have already seen that writing $T^*\bbC P^n$ as $Z-D$ yields a filtration of $SH^*(T^*\bbC P^n;\bbZ)$ whose associated graded is 
    \begin{equation}
    Gr_kSH^*(T^*\bbC P^n;\bbZ)\cong
    \begin{cases}
    H^*(\bbC P^n;\bbZ)& k=1 \\
    H^{*+2nk}(S_DZ;\bbZ)& k>1 \\
    0 & {\rm otherwise}.
    \end{cases}
    \end{equation}
I.e., we have shown that, abstractly, $SH^*(T^*\bbC P^n;\bbZ)$ is isomorphic to the direct sum over its associated graded. In particular, the spectral sequence arising from the aforementioned filtration of $SH^*(T^*\bbC P^n;\bbZ)$ cannot possibly have a non-trivial differential, for if it did, we would contradict the abstract isomorphism of $SH^*(T^*\bbC P^n;\bbZ)$ to the direct sum over its associated graded. Observe, since each map on homotopy groups induced by $\calG\calW^{\rm U}_{T^*\bbC P^n}\otimes_{\rm MU}\bbZ$,
    \begin{equation}
    H^{\ell+2n+1}(S_DM;\bbZ)\to H^\ell(\bbC P^n;\bbZ),\;\;\ell\in\bbZ,
    \end{equation}
arises as a differential in the aforementioned spectral sequence, each must vanish; this shows (1). Now, (2) follows by repeating the previous argument with $\bbZ$ replaced by $\bbZ/m$; the claim follows.
\end{proof}

\bibliography{References}{}
\bibliographystyle{alpha.bst}
\end{document}